\documentclass[10pt,oneside]{article}
\usepackage{mathrsfs}
\usepackage{amsmath,amsfonts,amssymb,amsthm}
\usepackage{caption}
\usepackage{subfigure}
\usepackage{cite}
\usepackage{color}
\usepackage{graphicx}
\usepackage{subfigure}
\usepackage{float}
\usepackage{paralist}
\usepackage{enumerate}
\usepackage{dsfont}
\usepackage{bm}
\usepackage[colorlinks,
linkcolor=red,
citecolor=blue
]{hyperref}
\usepackage{authblk}
\usepackage[T1]{fontenc}
\allowdisplaybreaks

\renewcommand{\thefootnote}{}
\newcommand\blfootnote[1]{%
	\begingroup
	\renewcommand\thefootnote{}\footnote{#1}%
	\addtocounter{footnote}{-1}%
	\endgroup
}

\newtheorem{thm}{Theorem}[section]
\newtheorem{lem}[thm]{Lemma}
\newtheorem{prop}[thm]{Proposition}
\newtheorem{rem}{Remark}[section]
\def\QEDopen{{\setlength{\fboxsep}{0pt}\setlength{\fboxrule}{0.2pt}\fbox{\rule[0pt]{0pt}{1.3ex}\rule[0pt]{1.3ex}{0pt}}}} %
\def\QED{\QEDopen} %
\def\endproof{\hspace*{\fill}~\QED\par\endtrivlist\unskip}%
\def\bma#1\ema{{\allowdisplaybreaks\begin{split}#1\end{split}}}

\numberwithin{equation}{section}

\begin{document}
	\title{Global well-posedness and large-time behavior of the compressible Navier-Stokes equations with hyperbolic heat conduction}
	\author[a]{ Fucai Li}
	\author[b]{ Houzhi Tang}
	\author[c]{Shuxing Zhang} 
	\affil[a] {School of Mathematics, Nanjing University, Nanjing 210093, P. R.  China}
    	\affil[b] { School of Mathematics and Statistics, Anhui Normal University, Wuhu 241002, P. R. China}
	\affil[c] {School of Mathematical Sciences, Jiangsu University, Zhenjiang 212013, P. R. China }
	\date{}
	\renewcommand*{\Affilfont}{\small\it}	
	\maketitle
	\blfootnote{
		{Email: fli@nju.edu.cn(F.-C. Li), houzhitang@amss.ac.cn(H.-Z. Tang),  zhangsx@ujs.edu.cn (Z.-S. Zhang)}
	}
	\allowdisplaybreaks
	\begin{abstract}
		The classical Fourier's law, which states that the heat flux is proportional to the temperature gradient, induces the paradox of  infinite propagation speed for heat conduction. To accurately simulate the real physical process, the hyperbolic model of heat conduction named Cattaneo's law was proposed, which leads to the finite speed of heat propagation. A natural question is whether the large-time behavior of the heat flux for compressible flow would be different for these two laws. In this paper, we aim to address this question by studying the global well-posedness and the optimal time-decay rates of classical solutions to the compressible Navier-Stokes system with Cattaneo's law. By designing a new method, we obtain the optimal time-decay rates for the highest order derivatives of the heat flux, which cannot be derived for the system with Fourier's law by Matsumura and Nishida [Proc. Japan Acad. Ser. A Math. Sci., 55(9):337-342, 1979]. In this sense, our results first reveal the essential differences between the two laws.
	\end{abstract}
	\textbf{2020 MR Subject Classification:}\  35Q30, 35B40, 35B65, 76N10.
	\section{Introduction}
	\subsection{The model}
	\hspace{2em}The compressible Navier-Stokes equations with heat conduction can be written in the following form
	$$
	\left\{\begin{array}{llll}
		\displaystyle \partial_t\rho+\textrm{div}(\rho u)=0,\\
		\displaystyle \partial_t(\rho u)+\mathrm{div}\big(\rho u\otimes u\big)+\nabla P=\text{div}S,\\
		\displaystyle \partial_t(\rho E)+\mathrm{div}(\rho Eu+Pu)+\text{div}q=\text{div}(uS).
	\end{array}\right.
	$$
	The unknown functions $\rho(x,t)$, $u(x,t)=(u_1(x,t),u_2(x,t),u_3(x,t))^\top$, $P(x,t)$, $e(x,t)$, $ E(x,t)=e+\frac{1}{2}|u|^2$  and $q(x,t)$ represent the fluid density, velocity, pressure, internal energy, total energy, and heat flux, respectively.  The stress tensor $S$ is given by
	$$
	S=2\tilde{\nu}\mathbb{D}u+\tilde{\eta}(\text{div}u)\mathbb{I}_3,
	$$
	where  $\mathbb{D}u=\frac{1}{2}(\nabla u+\nabla u^\top)$ is the so-called deformation tensor, $\mathbb{I}_3$ denotes the $3\times 3$ identity matrix, $\tilde{\nu}$ and $\tilde{\eta}$ are viscosity coefficients  satisfying
	\begin{align*}
		\tilde{\nu}>0, \quad\tilde{\eta}+\frac{2}{3}\tilde{\nu}\geq 0.	
	\end{align*}
	The heat flux $q$ is assumed to satisfy the following Cattaneo's law
	\begin{equation}\label{CLAW}
		\tau\partial_tq+q+\kappa\nabla\theta=0,	
	\end{equation}
	which gives rise to heat waves with finite propagation speeds. Here, $ \theta$ is the absolute temperature and $\kappa>0$ denotes the heat conductivity coefficient. The parameter $\tau$ is a relaxation time which represents the time is required to establish a steady state of heat conduction in an element suddenly exposed to the heat flux.
	
	In this paper, we consider the ideal gas model, i.e., the pressure $P$ and the internal energy $e$ satisfy
	$$
	P=R\rho\theta,~~e=\frac{R}{\gamma-1}\theta,
	$$
	where $\gamma>1$ is the adiabatic constant and $R>0$ represents the ideal gas constant.
	
	Therefore, we conclude from above to obtain the following compressible Navier-Stokes-Cattaneo (NSC) system
	\begin{equation}\label{NSC}
		\left\{\begin{array}{llll}
			\displaystyle \partial_t\rho+\textrm{div}(\rho u)=0,\\
			\displaystyle \partial_t(\rho u)+\mathrm{div}\big(\rho u\otimes u\big)+R\nabla (\rho\theta)=\tilde{\nu}\Delta u+(\tilde{\nu}+\tilde{\eta})\nabla\text{div} u,\\
			\displaystyle \frac{R}{\gamma-1} (\rho\partial_t\theta+\rho u\cdot\nabla \theta)+R\rho\theta \text{div}u+\text{div}q=2\tilde{\nu}|\mathbb{D}u|^2+\tilde{\eta}(\text{div}u)^2,  \\
			\displaystyle \tau \partial_tq+q+\kappa \nabla\theta=0.
		\end{array}\right.
	\end{equation}
	We consider the above system in the whole space $\mathbb{R}^3$ and supplement it with the initial data
	\begin{equation}\label{NSI-1}
		(\rho,u,\theta,q)(x,0)=(\rho_{0},u_{0},\theta_{0},q_{0})(x),
	\end{equation}
	and the far-field states
	\begin{equation}\label{NSI-2}
		\lim_{|x|\rightarrow \infty}(\rho,u,\theta,q)(x,t)=(\rho_*,0,\theta_*,0),
	\end{equation}
	where $\rho_*$ and $\theta_*$ denote the positive constants.
	
For the limit case $\tau=0$, the system \eqref{NSC} turns into the classical compressible Navier-Stokes equations, in which the relation between the heat flux and the temperature is governed by the classical Fourier's law,
	\begin{equation}\label{FL}
		q=-\kappa \nabla\theta.
	\end{equation}		
	As it's well known, Fourier's law is often used to describe the heat conduction phenomenon. Nevertheless,  it has a drawback of an inherent infinite propagation speed of signals, which implies that any local temperature disturbance causes an instantaneous perturbation in the temperature field at any point. In some situations, Fourier's law does not coincide with the experimental observations, especially when high-temperature gradients are considered, for example, see \cite{Liu2005Analysis,2000Effect,maxwell1867iv,onsager1931reciprocal}.  To redeem this deficiency, Cattaneo \cite{M1} proposed a damped version of  Fourier's law, i.e., \eqref{CLAW} by
	introducing a heat flux relaxation term to describe the finite speed of
	heat conduction.  Now, Cattaneo's law \eqref{CLAW} has been widely used in thermoelasticity and fluid models, see \cite{MR1888164,MR2562166,MR3186239} and the references therein. It is very interesting to point out that sometimes the results which hold for Fourier's law may be not true for Cattaneo's law. For example, Fern\'{a}ndez Sare and Racke \cite{MR2533927} showed that Fourier's law preserves the property of exponential stability while Cattaneo's law destroys such a property for certain Timoshenko-type thermoelastic system.
\subsection{Related Literature}
	\hspace{2em} Now we review some mathematical results on the compressible Navier-Stokes equations with Fourier's law and Cattaneo's law. First, due to its physical importance and mathematical challenge, there is a lot of works on the well-posedness and large-time behavior of solutions to the compressible Navier-Stokes equations combined with Fourier's law \eqref{FL}, for example,  see \cite{MR2225997,MR1970039,MR3708192,MR1867887,MR896014,MR1339675,MR2877344,MR1810944,MR2005201,MR106646,MR1488513,MR564670,MR149094,MR713680,MR2164944}. The local existence and uniqueness of smooth solutions were initiated by Serrin \cite{MR106646} and Nash \cite{MR149094} for initial density far away from vacuum. Later, Matsumura and Nishida \cite{MR564670} proved the global existence and uniqueness of the classical solution for small initial data without vacuum.  The $L^2$-time decay estimates were investigated by Mastumura and Nishida \cite{MR555060} for the small solution. Later, Li \cite{MR2164944} analyzed the detailed wave structure of the Green function for linearized system. Du and Wu \cite{MR3708192} established the pointwise estimates of the classical solution for nonlinear system.
	
	When the heat conduction is described by Cattaneo's law, there are also a few of contributions to the well-posedness and singular limit of the compressible NSC type system, see \cite{MR3810847,MR4312119,MR3518771,MR4436140,MR4413297,MR4988883} and the references therein. In \cite{MR3518771}, Hu and Racke considered the pressure $P=P(\rho,\theta)$ and the constant equilibrium state $(\rho,u,\theta,q)=(1,0,1,0)$. Under an additional assumption on the constant relaxation time $0<\tau<\frac{2\kappa}{P_\theta(1,1)^2}$, they proved the global existence and relaxation limit of the small smooth solution. Generally speaking, the relaxation constant $\tau$ is fixed for a specific fluid. Therefore, their results cannot be applied to the general compressible flow since the relaxation constant $\tau$ in \cite{MR3518771} is bounded by the constant $\frac{2\kappa}{P_\theta(1,1)^2}$ which depends on the heat conduction coefficient $\kappa$. Later, Hu, Racke and Wang \cite{MR4413297} investigated the formation of singularities in one-dimensional hyperbolic compressible Navier-Stokes equations with the nonlinear Cattaneo's law for
	heat conduction as well as through the constitutive Maxwell type relations for the stress tensor. Wu, Zhou, Li \cite{MR4720751} established the time-decay rates of the solution of the NSC system by using the pure energy method under some assumptions on the $\tau$. Later, Liu and Wu \cite{MR4649149} studied the space-time behavior of this system by using Green method under the same condition. Very recently, the second author and his collaborators \cite{MR4686805} obtained the same decay rates without the assumption on the $\tau$ by using the pure energy method. It should be mentioned that the optimal time-decay rates of the highest order derivatives of the heat flux can not be achieved in this method.

\subsection{Main Motivation}
	\hspace{2em}In previous works  \cite{MR4720751,MR4649149,MR4686805}, the time-decay rates of solutions were not shown to be optimal, as no lower bound of the decay rates was provided.
Moreover, due to the damping structure in Cattaneo's law, the decay rates of heat flux should be faster than density and velocity even for highest order derivatives. To fill these gaps, we aim to establish the optimal decay estimates of the solution to the NSC system by using spectral analysis instead of the pure energy method in this paper. Different from the large-time behavior of  compressible Navier-Stokes equations with Fourier's law has been studied in \cite{MR555060}, the  motivation of this paper is to address this question by studying the global well-posedness and large-time behavior of the classical solutions to the compressible Navier-Stokes system with the Cattaneo's law without any additional conditions on the relaxation time $\tau$. Moreover, due to the damping structure in Cattaneo's law, we first obtain the optimal time-decay rates for the highest order derivatives of the heat flux, which cannot be derived for the system with Fourier's law.

More precisely, our first results on the global existence and upper bound time decay rates of classical solutions to the Cauchy problem \eqref{NSC}-\eqref{NSI-2} is given as
	\begin{thm}\label{thm1}
		Let $s\geq 3$ be an integer. Assume that the initial data $(\rho_0-\rho_*,u_0,\theta_0-\theta_*,q_0)\in H^s(\mathbb{R}^3)$ satisfying
		\begin{equation}\label{in-data-e0}
			\|(\rho_0-\rho_*,u_{0},\theta_0-\theta_*,q_0)\|_{H^s}\leq \delta_0,
		\end{equation}
		for some sufficiently small constant $\delta_0>0$, then the Cauchy problem \eqref{NSC}-\eqref{NSI-2} admits a unique global classical solution $(\rho,u,\theta,q)$ enjoying
		\begin{align}
			&\|(\rho-\rho_*,u,\theta-\theta_*,q)(t)\|_{H^s}^2\notag\\
			&\quad +\int_0^t\big( \|\nabla \rho(r)\|_{H^{s-1}}^2+\|\nabla u(r)\|_{H^{s}}^2+\|\nabla \theta(r) \|_{H^{s-1}}^2+\|q(r)\|_{H^{s}}^2\big)\mathrm{d}r
			\leq C\delta_0^2.\label{main-est}
		\end{align}
		Moreover, if the initial perturbation $(\rho_0-\rho_*,u_0,\theta_0-\theta_*,q_0) \in L^1(\mathbb{R}^3)$, then it holds  that, for $0\leq k\leq s$,
		\begin{equation}\label{upper}
			\|\nabla^k(\rho-\rho_*,u,\theta-\theta_*)(t)\|_{L^2}\leq C(1+t)^{-\frac{3}{4}-\frac{k}{2}},\quad \|\nabla^kq(t)\|_{L^2}\leq C(1+t)^{-\frac{5}{4}-\frac{k}{2}},
		\end{equation}
		where the positive constant $C$ is independent of time.
	\end{thm}
	
	Basing on the global existence and the upper bound time-decay rates, we can further show that the above rates are optimal.
	\begin{thm}\label{thm3}
		Suppose that the conditions in Theorem \ref{thm1} hold. If the Fourier transform of the initial perturbation $(\hat{n}_0,\hat{w}_0,\hat{\phi}_0,\hat{\psi}_0) \triangleq \Big(\frac{\hat{\rho}_0-\rho_*}{\rho_*},\frac{\hat{u}_0}{\sqrt{R\theta_*}},\frac{\hat{\theta}_0-\theta_*}{\sqrt{\gamma-1}\,\theta_*},\sqrt{\frac{\tau}{\kappa\rho_*R\theta_*^2}}\hat{q}_0\Big)$ satisfies
		\begin{equation}\label{in-data-optimal}
			\begin{aligned}
				&	\inf_{ |\xi|\leq r_0}|\hat{n}_0(\xi)|\geq \mu_0, ~~\text{and}~~\hat{w}_0(\xi)=\hat{\phi}_0(\xi)=0~~\text{for}~~|\xi|\leq r_0,
			\end{aligned}
		\end{equation}
		where $\mu_0$ is a positive constant and $r_0>0$ denotes a sufficiently small constant, then the global solution $(\rho,u,\theta,q)$ given by Theorem \ref{thm1} satisfies, for large time and $0\leq k\leq s$, that
		\begin{equation}\label{optimal-est}
			\bar{C}_*(1+t)^{-\frac{3}{4}-\frac{k}{2}}\leq \|\nabla^k(\rho-\rho_*,u,\theta-\theta_*)(t)\|_{L^2}\leq C(1+t)^{-\frac{3}{4}-\frac{k}{2}},
		\end{equation}
		and
		\begin{equation}\label{optimal-est1}
			\bar{C}_*(1+t)^{-\frac{5}{4}-\frac{k}{2}}\leq \|\nabla^kq(t)\|_{L^2}\leq C(1+t)^{-\frac{5}{4}-\frac{k}{2}},
		\end{equation}
		where $\bar{C}_*$ and $C$ are positive constants independent of time.
	\end{thm}
	
	\begin{rem}
		Following the spectral  analysis and energy method developed in this paper, the time-decay of results in \cite{MR555060}: \begin{equation*}
			\|(\rho-\rho_*,u,\theta-\theta_*)(t)\|_{H^3}\leq C(1+t)^{-\frac{3}{4}},
		\end{equation*}
		can be generalized as
		\begin{equation*}
			\|\nabla^k(\rho-\rho_*,u,\theta-\theta_*)(t)\|_{L^2}\leq C(1+t)^{-\frac{3}{4}-\frac{k}{2}}, \quad  \text{for}\;\; 0\leq k\leq s.
		\end{equation*}
		However, the classical Fourier's law $q=-\kappa\nabla \theta$ implies for $0\leq k\leq s-1$ that
		\begin{equation}\label{1017}
			\|\nabla^kq(t)\|_{L^2}\leq C(1+t)^{-\frac{5}{4}-\frac{k}{2}}.
		\end{equation}
		It is obviously observed from \eqref{upper} and \eqref{1017} that Cattaneo's law can provide a faster time-decay rate for the highest derivative of the heat flux but Fourier's law fails for this.
	\end{rem}
	
	\begin{rem}
		Due to the damping structure of $q$ in Cattaneo's law, the time decay rate of $\|\nabla ^k q\|_{L^2}$ is faster than $\|\nabla^k(\rho-\rho_*,u,\theta-\theta_*)\|_{L^2}$.
	\end{rem}
	
	
	\begin{rem}\label{rem1}
		By Sobolev's inequality and \eqref{upper}, it follows that, for any $p\in [2,6]$ and $0\leq k\leq s-1$,
		\begin{align*}
			\|\nabla^k(\rho-\rho_*,u,\theta-\theta_*)(t)\|_{L^p}\leq  C(1+t)^{-\frac{3}{2}(1-\frac{1}{p})-\frac{k}{2}},\quad \|\nabla^kq(t)\|_{L^p}\leq C(1+t)^{-2(1-\frac{3}{4p})-\frac{k}{2}}.
		\end{align*}
		Furthermore, it holds, for $0\leq k\leq s-2$,
		\begin{align*}
			\|\nabla^k(\rho-\rho_*,u,\theta-\theta_*)(t)\|_{L^\infty}\leq  C(1+t)^{-\frac{3+k}{2}},\quad
			\|\nabla^kq(t)\|_{L^\infty}\leq C(1+t)^{-\frac{4+k}{2}}.
		\end{align*}
	\end{rem}

\begin{rem}
In fact, we could derive the decay rates by using pure energy method in \cite{MR4686805} when the initial data belongs to some negative Sobolev space. However, it is not able to show the optimal decay rates for the highest order derivatives for the velocity due to the limitations of Sobolev inequality. Besides, we cannot prove the obtained decay rates are optimal, where the lower bound of the decay rates are unknown. Therefore, we need to investigate the detailed spectral analysis for the Green function.
\end{rem}
	\begin{rem}
		It should be pointed out that the approach developed in this paper can be extended to other fluid models with Cattaneo's law instead of Fourier's law, for instance, the non-isentropic compressible magnetohydrodynamic system. Furthermore, we can study the pointwise estimates of the solution to the NSC system by using Green function method.
	\end{rem}
\subsection{Primary challenge and ideas of the proof}
	\hspace{2em} Now, we introduce the main difficulties in this paper and the ideas used in the proofs of Theorems \ref{thm1} and \ref{thm3}. Notice that Hu and Racke \cite{MR3518771} verified  the Kawashima condition to prove the global existence by the energy method developed in \cite{1984Systems}, and the condition $0<\tau<\frac{2\kappa}{P_\theta(1,1)^2}$ is necessary in their arguments. Since the relaxation constant $\tau$ is fixed for a specific fluid, we need develop new ideas and methods to establish the global well-posedness for the general fluid without any restriction on $\tau$. We first take advantage of the ingenious structure of the system to obtain the basic energy estimate \eqref{lem-b1} without any restriction on $\tau$. Then, utilizing the formulated system \eqref{nsac}-\eqref{non-f}, we establish the uniform estimates for high-order derivatives with the help of the damping structure of $q$ and  \eqref{lem-b1}. Therefore, combining the uniform estimates and continuous argument yields global existence.
	
	Since the heat flux $q(x,t)$ satisfies Cattaneo's law, it is natural to solve it as
	\begin{equation}\label{32201}
		q(x,t)=e^{-\frac{1}{\tau} t}q_0-\frac{\kappa}{\tau}\int_0^te^{-\frac{1}{\tau}(t-r)}\nabla\theta(x,r)\mathrm{d}r.
	\end{equation}
	It is observed from above that the upper bound decay rate of $\|q\|_{L^2}$ is the same order as $\|\nabla\theta\|_{L^2}$. This enlightens us the time-decay rates of the lower order derivatives  $\|\nabla^kq\|_{L^2}(0\leq k\leq s-1)$ can be obtained by calculating $\|\nabla^{k+1}\theta\|_{L^2}$. Nevertheless, when considering the highest derivative $\|\nabla^sq\|_{L^2}$, it can not be achieved by the same method as above since the regularities of $q$ and $\theta$ are same. In order to surround this difficulty, we first rewrite the perturbated equations in the operator form. Then, by the spectral analysis and energy method, it is shown that
	\begin{equation}\label{32202}
		\|\nabla^k(\rho-\rho_*,u,\theta-\theta_*,q)(t)\|_{L^2}\leq C(1+t)^{-\frac{3}{4}-\frac{k}{2}}, \quad 0\leq k\leq s.
	\end{equation}
	In terms of \eqref{32201}, we enhance the time-decay rate of $q$ as
	\begin{equation}\label{32203}
		\|\nabla^kq(t)\|_{L^2}\leq C(1+t)^{-\frac{5}{4}-\frac{k}{2}},\quad 0\leq k\leq s-1.
	\end{equation}
Inspired by the previous works \cite{MR4175837,MR4776378,MR4445672}, by introducing the low-high frequency decomposition of the solution defined in \eqref{l-h-com}, we find  that the low-frequency part of the highest order derivatives $\nabla^sq^\ell$ satisfies
	\begin{equation}\label{32204}
		\|\nabla^s q^\ell(t)\|_{L^2}\leq C(1+t)^{-\frac{5}{4}-\frac{s}{2}}.
	\end{equation}
	After an important observation, i.e., the decay rates of the high-frequency part of the highest order derivatives $\|\nabla^sq^h\|_{L^2}$ are faster than $\|\nabla^sq^\ell\|_{L^2}$ given in \eqref{32204}, we arrive at
	\begin{equation}\label{32205}
		\|\nabla^s q(t)\|_{L^2}\leq \|\nabla^s q^\ell(t)\|_{L^2}+\|\nabla^s q^h(t)\|_{L^2} \leq C(1+t)^{-\frac{5}{4}-\frac{s}{2}}.
	\end{equation}
	Thus, combining \eqref{32202}, \eqref{32203} and \eqref{32205} up yields the  upper bound time-decay rates \eqref{upper}.
	
	To illustrate the rates in \eqref{upper} are optimal, we need to establish the lower bound decay estimates for nonlinear system \eqref{NSC}-\eqref{NSI-2}. Under some suitable assumptions on the initial data, we show that the optimal decay rates of solutions for the linearized system by using the spectral analysis. Due to the smallness of the initial data, we can prove that the solution of the nonlinear system satisfies
	$$
	\underline{C}(1+t)^{-\frac{3}{4}}\leq \|(\rho-\rho_*,u,\theta-\theta_*)(t)\|_{L^2}\leq \bar{C}(1+t)^{-\frac{3}{4}},~~
	\underline{C}(1+t)^{-\frac{5}{4}}\leq \|q(t)\|_{L^2}\leq \bar{C}(1+t)^{-\frac{5}{4}},
	$$
	where the positive constants $\underline{C}$ and $\bar{C}$ are independent of time. Then, with the help of \eqref{upper} and the Sobolev interpolation inequality \eqref{lem-ff-1}, we derive the lower bound of time-decay rates for high-order derivatives, which implies that the rates in \eqref{upper} are optimal.
	
	The rest of the paper is organized as follows.
	In Section \ref{Sec2}, we introduce some notations and recall a few auxiliary lemmas.
	Section \ref{Sec3} is devoted to the \emph{a priori} estimates which guarantees the local solution can be extended to a global one. In Section \ref{Sec4}, the upper bound of time-decay rates of the linearized problem are investigated. The proofs of Theorems \ref{thm1} and \ref{thm3} are completed in Section \ref{Sec5}.
	
	\section{Preliminaries}\label{Sec2}
	\subsection{Notation}
	\hspace{2em}In this subsection, we explain the notation and conventions used throughout the paper.
		$L^p(\mathbb{R}^3)$ and $W^{k,p}(\mathbb{R}^3)$ denote the usual Lebesgue and Sobolev space on $\mathbb{R}^3$, with norms $\|\cdot\|_{L^p}$ and $\|\cdot\|_{W^{k,p}}$, respectively.  When $p=2$, we denote $W^{k,p}(\mathbb{R}^3)$ by $H^k(\mathbb{R}^3)$ with the norm $\|\cdot\|_{H^k}$. We denote by $C$ a generic positive constant which may vary in different estimates. $f_1\lesssim f_2$ and $f_1=O(f_2)$ describes that there exists a constant $C>0$ such that $|f_1|\leq C|f_2|$. $f_1\sim f_2$ represents that $f_1\lesssim f_2$ and $f_2\lesssim f_1$. For an integer $\alpha$,  the symbol $\nabla^\alpha$ denotes the summation of all terms $\partial_{x_1}^{\alpha_1}\partial_{x_2}^{\alpha_2}\partial_{x_3}^{\alpha_3}$ with the multi-index $(\alpha_1,\alpha_2,\alpha_3)$ satisfying $\alpha=\alpha_1+\alpha_2+\alpha_3$. For a function $f$, $\|f\|_{X}$ represents the norm of $f$ on $X$, $\|(f,g)\|_{X}$ denotes the summation $\|f\|_{X}+\|g\|_{X}$. As mentioned before, the Fourier transform of $f$ is denoted by $\hat{f}$ or $\mathscr{F}[f]$ satisfying
		$$\hat{f}(\xi)
		=\mathscr{F}[f](\xi)
		=\int_{\mathbb{R}^3} f(x)e^{-i2\pi x\cdot\xi}\mathrm{d}x,\quad \xi\in\mathbb{R}^3.
		$$
		Correspondingly, the inverse Fourier transform of $f$ is denoted by $\check{f}$ or $\mathscr{F}^{- 1}[f]$ such that $$\check{f}(x)=\mathscr{F}^{- 1}[f](x)=\int_{\mathbb{R}^3} f(\xi)e^{i2\pi\xi\cdot x}\mathrm{d}\xi,\quad x\in\mathbb{R}^3.
		$$
		Let $\Lambda^\ell$ be the pseudodifferential operator defined by
		$$\Lambda^\ell f=\mathscr{F}^{-1}(|\xi|^\ell \hat{f})\ \ \text{for}\ \ell\in \mathbb{R}.$$
		We define operators $P_{1}$ and $P_{\infty}$ on $L^{2}$ by
		\begin{equation*}
			P_{j}f=\mathscr{F}^{- 1}\Big(\hat\chi_{j}\mathscr{F}[f]\Big),\quad j=1~\text{or}~\infty,
		\end{equation*}
		where $\hat{\chi}_{j}(j=1~\text{or}~\infty)\in C^{\infty}(\mathbb{R}^{3})$, $0\leq \hat\chi_{j}\leq 1$ are cut-off functions defined by
		\begin{equation*}
			\hat\chi_{1}(\xi)=\left\{
			\begin{array}{l}
				1\quad (|\xi|\leq r_{0}),\\
				0\quad (|\xi|\geq R_0),
			\end{array}
			\right.
			\quad\hat\chi_{\infty}(\xi)=1-\hat\chi_{1}(\xi),
		\end{equation*}
		for constants $r_0$ and $R_0$ satisfying $0<r_0<R_0$.
		\subsection{Auxiliary lemmas}
		\hspace{2em}In this subsection, we collect some elementary inequalities and important lemmas that are used extensively in this paper. We first recall the following Gagliardo-Nirenberg's inequality.
		
		\begin{lem}\label{Ap2} \emph{(\!\cite{mb2002,MR4400698} or \cite[Lemma A.1]{gw2012-decay})}\label{lem-gn}
			Let $l,s$ and $k$ be any real numbers satisfying $0\leq
			l,s<k$, and let  $p, r, q \in [1,\infty]$ and $\frac{l}{k}\leq
			\theta\leq 1$ such that
			$$\frac{l}{3}-\frac{1}{p}=\left(\frac{s}{3}-\frac{1}{r}\right)(1-\theta)+\left(\frac{k}{3}-\frac{1}{q}\right)\theta.
			$$ Then, for any $u\in W^{k,p}(\mathbb{R}^3),$  we have
			\begin{equation}
				\label{h20}
				\|\nabla^l u\|_{L^p}\lesssim \|\nabla^s u\|_{L^r}^{1-\theta}\|\nabla^k
				u\|_{L^q}^{\theta}.
			\end{equation}	
		\end{lem}
		
		Next, we introduce the following commutator estimates.
		\begin{lem}\emph{(\!\cite[Lemma 3.4]{mb2002})}\label{lem-commu}
			For any integer $k\geq1$ and the pair of functions $f,g\in H^k(\mathbb{R}^3)\cap L^{\infty}(\mathbb{R}^3)$, it holds
			\begin{equation}\label{lem-commu1}
				\|\nabla^k(fg)\|_{L^2}\leq C \|f\|_{L^{\infty}}\|\nabla^kg\|_{L^2}+C\|\nabla^kf\|_{L^2}\|g\|_{L^{\infty}}.
			\end{equation}
			Moreover, if $\nabla f\in L^{\infty}(\mathbb{R}^3)$, we have
			\begin{equation}\label{lem-commu2}
				\big\|\nabla^{k}(fg)-f\nabla^{k}g\big\|_{L^2}\leq C \|\nabla f\|_{L^{\infty}}\|\nabla^{k-1}g\|_{L^2}+C\|\nabla^kf\|_{L^2}\|g\|_{L^{\infty}}.
			\end{equation}
		\end{lem}
		Finally, we introduce a special Sobolev interpolation inequality (see \cite[Lemma A.4]{gw2012-decay}).
		\begin{lem}\label{lem-ff}
			Let $\ell\geq 0$ and the integer $k \geq 0$, then we have
			\begin{equation}\label{lem-ff-1}
				\|\nabla^k f\|_{L^2}\lesssim \|\nabla^{k+1}f\|_{L^2}^{1-\theta}\|\Lambda^{-\ell}f\|_{L^2}^\theta,
			\end{equation}
			where\  $\theta=\frac{1}{1+k+\ell}$ and $\Lambda^{-\ell}f=\mathscr{F}^{-1}\big(|\xi|^{-\ell}\hat{f}\big)$.
		\end{lem}
		
		
		\section{Global existence of the classical solution}\label{Sec3}
		\subsection{Reformulated problem} 
		\label{sub:reformulated_problem}
		\hspace{2em}
		In this subsection, we linearize the solution $(\rho,u,\theta,q)$ of the problem \eqref{NSC}-\eqref{NSI-2} around the constant state $(\rho_*,0,\theta_*,0)$.  For notation convenience, we introduce some positive constants
		\begin{equation}\label{new-const}
			c=\sqrt{R\theta_*},\quad \sigma=\sqrt{(\gamma-1)R\theta_*},\quad \nu=\frac{\tilde{\nu}}{\rho_*},\quad \eta=\frac{\tilde{\eta}}{\rho_*},\quad a=\sqrt{\frac{\kappa\rho_*R\theta_*^2}{\tau}},~b=\sqrt{\frac{\kappa(\gamma-1)}{\tau\rho_*R}}.
		\end{equation}
		Define the perturbed variables as
		\begin{equation}\label{new-var}
			n=\frac{\rho-\rho_*}{\rho_*},\quad w=\frac{u}{c},\quad \phi=\frac{\theta-\theta_*}{\sqrt{\gamma-1}\,\theta_*}, \quad \psi=\frac{q}{a}.
		\end{equation}
		Then, the problem \eqref{NSC}-\eqref{NSI-2} is reformulated to
		\begin{equation}\label{nsac}
			\left\{
			\begin{aligned}
				&\partial_tn +c\,{\rm{div}}w=f_1, \\
				&\partial_tw+c\nabla n+\sigma \nabla \phi=\nu\Delta w+(\nu+\eta)\nabla{\rm{div}}w+f_2,\\
				&\partial_t\phi+\sigma {\rm{div}} w+b\text{div}\psi=f_3,\\
				&\partial_t\psi+\frac{1}{\tau}\psi+b\nabla\phi=0,
			\end{aligned}
			\right.
		\end{equation}
		subject to the initial data
		\begin{equation}\label{in-data2}
			(n,w,\phi,\psi)(x,0)=(n_0,w_0,\phi_0,\psi_0)=\Big(\frac{\rho_0-\rho_*}{\rho_*},\frac{u_0}{c},\frac{\theta_0-\theta_*}{\sqrt{\gamma-1}\,\theta_*},\frac{q_0}{a}\Big),
		\end{equation}
		and the far-field states
		\begin{equation}\label{in-data3}
			\lim_{|x|\rightarrow \infty}(n,w,\phi,\psi)=(0,0,0,0),
		\end{equation}
		where the nonlinear terms $f_{i}$ satisfy
		\begin{equation}\label{non-f}
			\left\{
			\begin{aligned}
				&f_{1}=-c\,{\rm{div}}(nw),\\
				&f_{2}=-cw\cdot\nabla w+\frac{cn\nabla n}{1+n}-\frac{\sigma\phi\nabla n}{1+n}-\frac{\nu n\Delta w}{1+n}-\frac{(\nu+\eta)n\nabla{\rm{div}}w}{1+n},\\
				& f_{3}=-cw\cdot\nabla\phi-c(\gamma-1)\phi{\rm{div}}w+\frac{\sigma(2\nu|\mathbb{D}w|^2+\eta({\rm{div}}w)^2)}{c(1+n)}+\frac{b n\text{div} \psi}{1+n}.
			\end{aligned}
			\right.
		\end{equation}
		\vskip 2mm
		To prove the global existence of the classical solution to the problem \eqref{nsac}-\eqref{in-data3}, we first state the following local existence of the reformulated system \eqref{nsac}-\eqref{non-f}. The proof can be followed in a similar manner to that in \cite{MR564670,1984Systems} and we omit the details.
		\begin{thm}\emph{(Local existence)}\label{lem-loc}
			Let $s\geq 3$ be an integer and	assume that the initial data $(n_0,w_0,\phi_0,\psi_0)\in H^s(\mathbb{R}^3)$. Then, for any positive constant $\epsilon_0$, there exists a positive constant $T_*$ depending only on $\epsilon_0$ such that if $\|(n_0,w_0,\phi_0,\psi_0)\|_{H^s}\leq \epsilon_0$, then the reformulated problem \eqref{nsac}-\eqref{non-f} admits a unique classical solution $(n,w,\phi,\psi)$ satisfying
			$$
			\left\{
			\begin{aligned}
				&n\in \mathcal{C}([0,T_*],H^s(\mathbb{R}^3))\cap \mathcal{C}^1([0,T_*],H^{s-1}(\mathbb{R}^3)),\\
				&w\in \mathcal{C}([0,T_*],H^s(\mathbb{R}^3))\cap \mathcal{C}^1([0,T_*],H^{s-2}(\mathbb{R}^3)),\\
				&\phi\in \mathcal{C}([0,T_*],H^s(\mathbb{R}^3))\cap \mathcal{C}^1([0,T_*],H^{s-1}(\mathbb{R}^3)),\\
				&\psi\in \mathcal{C}([0,T_*],H^{s}(\mathbb{R}^3))\cap \mathcal{C}^1([0,T_*],H^{s-1}(\mathbb{R}^3))\cap L^2([0,T_*],H^s(\mathbb{R}^3)).
			\end{aligned}
			\right.
			$$
		\end{thm}
		\vskip 4mm
		To extend the local classical solution to a global one, it is essential to establish  adequate uniform estimates. We first give the \emph{a priori} assumption of the classical solution for the sufficiently small constant $\delta>0$ and any given time $T>0$ such that
		\begin{equation}\label{a-priori est}
			\sup_{0\leq t\leq T}\|(n,w,\phi,\psi)(t)\|_{H^s}\leq \delta.
		\end{equation}
		With the help of \eqref{a-priori est}, Sobolev's inequalities, and the smallness of $\delta$, it yields
		\begin{gather}
			\frac{1}{2}\rho_*\leq
			\rho_*-C\delta\leq\rho(x,t)
			\leq\rho_*+C\delta\leq \frac{3}{2}\rho_*,
			\quad \frac{1}{2}\theta_*\leq
			\theta_*-C\delta\leq
			\theta(x,t)\leq\theta_*+C\delta
			\leq \frac{3}{2}\theta_*,\nonumber\\
			\|\nabla^k(n,w,\phi,\psi)\|_{L^\infty}\leq C\delta,\quad 0\leq k\leq s-2.\label{l-inf-est}
		\end{gather}
		\subsection{Time-independent energy estimates}
		\label{sub:a_priori_estimates}
		\hspace{2em} To obtain the uniform bound of the solutions to the reformulated problem \eqref{nsac}-\eqref{non-f}, we first give the basic $L^2$ energy estimate as
		\begin{lem}\label{lem-b}
			Assume that $(n,w,\phi,\psi)$ is the classical solution of the problem \eqref{nsac}-\eqref{non-f} satisfying the a priori assumption \eqref{a-priori est}, then it holds for $0\leq t\leq T$ that
			\begin{equation}\label{lem-b1}
				\|(n,w,\phi,\psi)(t)\|_{L^2}^2+\int_0^t\left(\|\nabla w(r)\|_{L^2}^2+\| \psi(r)\|_{L^2}^2\right)\mathrm{d}r
				\leq C\delta_0^2,
			\end{equation}	
			where the positive constant $C$ is independent of time.
		\end{lem}
		\begin{proof}
			Multiplying $\eqref{NSC}_2$ by $u$ and integrating it over $\mathbb{R}^3$, we have
			\begin{equation}\label{lem-b21}
				\frac{1}{2}\frac{\mathrm{d}}{\mathrm{d}t}\int_{\mathbb{R}^3}\rho |u|^2\mathrm{d}x
				+\tilde{\nu}\|\nabla u\|_{L^2}^2+(\tilde{\nu}+\tilde{\eta})\|\nabla\text{div}u\|_{L^2}^2=\int_{\mathbb{R}^3} R\rho \theta \text{div}u \mathrm{d}x.
			\end{equation}
			In terms of $\eqref{NSC}_{3}$ and \eqref{lem-b21}, one has
			\begin{equation}\label{ox1}
				\frac{\mathrm{d}}{\mathrm{d}t}\int_{\mathbb{R}^3}\frac{R}{\gamma-1}\rho(\theta-\theta_*)\mathrm{d}x-\tilde{\nu}\|\nabla u\|_{L^2}^2
				-(\tilde{\nu}+\tilde{\eta})\|\text{div}u\|_{L^2}^2=-\int_{\mathbb{R}^3}R\rho\theta\text{div}u\mathrm{d}x.
			\end{equation}
			Combining \eqref{lem-b21} and \eqref{ox1} together gives
			\begin{equation}\label{lem-b3}
				\frac{\mathrm{d}}{\mathrm{d}t}\int_{\mathbb{R}^3}\Big(\frac{1}{2}\rho |u|^2+\frac{R}{\gamma-1}\rho(\theta-\theta_*)\Big)\mathrm{d}x=0.
			\end{equation}
			Multiplying $\eqref{NSC}_3$ by $-\frac{\theta_*}{\theta}$, integrating the resulting equation over $\mathbb{R}^3$ and using $\eqref{NSC}_1$ give
			\begin{align}
				&\frac{\mathrm{d}}{\mathrm{d}t}\int_{\mathbb{R}^3}\left(R\theta_* \big(\rho \ln\frac{\rho}{\rho_*}-\rho+\rho_*\big)-\frac{R}{\gamma-1}\rho\theta_*\ln\frac{\theta}{\theta_*}\right)\mathrm{d}x \notag\\
				&\quad+\int_{\mathbb{R}^3}\left(\frac{\theta_*}{\theta}(2\tilde{\nu}|\mathbb{D}u|^2+\tilde{\eta}(\mathrm{div}u)^2)-\frac{\theta_*q\cdot\nabla\theta}{\theta^2}\right)\mathrm{d}x=0.
				\label{lem-b4}
			\end{align}
			
			Next, multiplying $\eqref{NSC}_4$ by $\frac{\theta_*q}{\kappa\theta^2}$ and integrating it over $\mathbb{R}^3$ yield
			\begin{align*}
				\int_{\mathbb{R}^3}\frac{\tau\theta_*}{\kappa\theta^2}q\cdot\partial_tq\mathrm{d}x+\int_{\mathbb{R}^3}\frac{\theta_*}{\kappa\theta^2}|q|^2\mathrm{d}x+\int_{\mathbb{R}^3}\frac{\theta_*q\cdot\nabla\theta}{\theta^2}\mathrm{d}x=0.	
			\end{align*}
			It is easy to verify that
			\begin{align*}
				\int_{\mathbb{R}^3} \frac{\tau\theta_*}{\kappa\theta^2}q\cdot\partial_tq\mathrm{d}x&=\frac{1}{2}\frac{\mathrm{d}}{\mathrm{d}t}\int_{\mathbb{R}^3}\frac{\tau\theta_*}{\kappa\theta^2}|q|^2\mathrm{d}x+\int_{\mathbb{R}^3}\frac{\tau\theta_*\theta_t}{\kappa\theta^3}|q|^2\mathrm{d}x.	
			\end{align*}
			Rewrite  $\eqref{NSC}_3$ into the following form
			\begin{align*}
				\theta_t=-u\cdot\nabla\theta-(\gamma-1)\theta\text{div}u+\frac{(\gamma-1)(-\text{div}q+2\tilde{\nu}|\mathbb{D}u|^2+\tilde{\eta}(\text{div}u)^2)}{R\rho}.
			\end{align*}
			In view of  \eqref{l-inf-est}, \eqref{new-var} and \eqref{a-priori est}, we have
			$$
			\begin{aligned}
				\int_{\mathbb{R}^3}\frac{\tau\theta_*\theta_t}{\kappa\theta^3}|q|^2\mathrm{d}x&\leq\Big| \int_{\mathbb{R}^3}\frac{\tau\theta_*}{\kappa\theta^3}\bigg(-u\cdot\nabla\theta-(\gamma-1)\theta\text{div}u+\frac{(\gamma-1)(-\text{div}q+2\tilde{\nu}|\mathbb{D}u|^2+\tilde{\eta}(\text{div}u)^2)}{R\rho}\bigg)|q|^2\mathrm{d}x\Big|\\
				&\leq  \int_{\mathbb{R}^3}\frac{\tau \theta_*}{\kappa\theta^3} |q|^2\mathrm{d}x
				(\|u\|_{L^\infty}\|\nabla\theta\|_{L^\infty}+\|\theta\|_{L^\infty}\|\text{div}u\|_{L^\infty}+\|\text{div}q\|_{L^\infty}+\|\nabla u\|_{L^\infty}^2) \\
				&\leq  \int_{\mathbb{R}^3}\frac{C\theta_*}{\kappa\theta^2} |q|^2\mathrm{d}x
				(\|\nabla w\|_{H^1}\|\nabla^2\phi\|_{H^1}+\|\nabla\phi\|_{H^1}\|\nabla^2w\|_{H^1}+\|\nabla^2q\|_{H^1}+\|\nabla^2 w\|_{H^1}^2) \\
				&\leq C(\delta+\delta^2)\int_{\mathbb{R}^3}\frac{\theta_*}{\kappa\theta^2}|q|^2\mathrm{d}x.
			\end{aligned}
			$$
			It is shown from the smallness of $\delta$ that
			\begin{equation}\label{A1}
				\begin{aligned}
					\frac{1}{2}\frac{\mathrm{d}}{\mathrm{d}t}\int_{\mathbb{R}^3}\frac{\tau\theta_*}{\kappa\theta^2}|q|^2\mathrm{d}x	+\frac{1}{2}\int_{\mathbb{R}^3}\frac{\theta_*}{\kappa\theta^2}|q|^2\mathrm{d}x+\int_{\mathbb{R}^3}\frac{\theta_*q\cdot\nabla\theta}{\theta^2}\mathrm{d}x\leq 0.	
				\end{aligned}	
			\end{equation}
			Adding $\eqref{lem-b3}$, $\eqref{lem-b4}$, and \eqref{A1} together gives
			\begin{align}
				&\frac{\mathrm{d}}{\mathrm{d}t}\int_{\mathbb{R}^3}\left(R\theta_* \big(\rho\ln\frac{\rho}{\rho_*}-\rho+\rho_*\big)+\frac{1}{2}\rho |u|^2+\frac{R}{\gamma-1}\rho(\theta-\theta_*-\theta_*\ln\frac{\theta}{\theta_*})
				+\frac{\tau\theta_*}{2\kappa\theta^2}|q|^2\right)\mathrm{d}x \notag\\				&\quad+\int_{\mathbb{R}^3}\Big(\frac{\theta_*}{2\kappa\theta^2}|q|^2+\frac{\theta_*}{\theta}(2\tilde{\nu}|\mathbb{D}u|^2+\tilde{\eta}(\mathrm{div}u)^2)\Big)\mathrm{d}x\leq 0.\label{lem-b5}
			\end{align}
			
			By \eqref{l-inf-est}, direct calculations yield that
			\begin{equation}\label{lem-b6}
				\begin{aligned}
					\rho\ln\frac{\rho}{\rho_*}-\rho+\rho_*\sim \frac{1}{4\rho_*}(\rho-\rho_*)^2,\quad \theta-\theta_*-\theta_*\ln\frac{\theta}{\theta_*}\sim \frac{1}{8\theta_*}(\theta-\theta_*)^2.
			\end{aligned}\end{equation}
			Thus, integrating \eqref{lem-b5} with respect to $t$ and using \eqref{lem-b6} gives
			\begin{equation}\label{lem-b9}
				\begin{aligned}
					&\|(\rho-\rho_*,u,\theta-\theta_*,q)(t)\|_{L^2}^2+\int_0^t(\|\nabla u(r)\|_{L^2}^2+\|q(r)\|_{L^2}^2)\mathrm{d}r\\
					&\leq C\|(\rho_0-\rho_*,u_0,\theta_0-\theta_*,q_0)\|_{L^2}^2%
					\leq C\delta_0^2,
				\end{aligned}
			\end{equation}
			which, together with \eqref{new-var}, gives \eqref{lem-b1}. This completes the proof of Lemma \ref{lem-b}.
		\end{proof}
		
		\begin{lem}\label{lem-n}
			Assume that $(n,w,\phi,\psi)$ is the classical solution of the problem \eqref{nsac}-\eqref{non-f} satisfying the a priori assumption \eqref{a-priori est}, then it holds for $0\leq t\leq T$ and $1\leq \alpha\leq s$ that
			\begin{equation}\label{lem-n1}
				\frac{1}{2}\frac{\mathrm{d}}{\mathrm{d}t}\|\nabla^\alpha n\|_{L^2}^2+c\int_{\mathbb{R}^3} \nabla^\alpha {\rm{div}}w\cdot\nabla^\alpha n \mathrm{d}x
				\leq C\delta(\|\nabla^\alpha n\|_{L^2}^2+\|\nabla^\alpha w\|_{L^2}^2+\|\nabla^{\alpha}{\rm{div}}w\|_{L^2}^2),
			\end{equation}
			where the positive constant $C$ is independent of time.
		\end{lem}
		\begin{proof}
			Applying the operator $\nabla^\alpha$ to $\eqref{nsac}_1$, multiplying the resulting equation by $\nabla^\alpha n$, and integrating over $\mathbb{R}^3$, we obtain
			\begin{align}
				&\frac{1}{2}\frac{\mathrm{d}}{\mathrm{d}t}\|\nabla^\alpha n\|_{L^2}^2+c\int_{\mathbb{R}^3} \nabla^\alpha\text{div}w\cdot\nabla^\alpha n \mathrm{d}x \notag\\
				&=-c\int_{\mathbb{R}^3} \nabla^\alpha (w\cdot\nabla n)\cdot\nabla^\alpha n \mathrm{d}x
				-c\int_{\mathbb{R}^3}\nabla^\alpha (n\text{div}w)\cdot\nabla^\alpha n \mathrm{d}x.\label{lem-n2}
			\end{align}
			By Lemma \ref{lem-commu}, \eqref{l-inf-est} and H$\ddot{\text{o}}$lder's inequality, one has
			\begin{align}
				&c\left| \int_{\mathbb{R}^3} \nabla^\alpha (w\cdot\nabla n)\cdot\nabla^\alpha n \mathrm{d}x\right| \notag\\
				&\leq c\left| \int_{\mathbb{R}^3} (\nabla^\alpha(w\cdot\nabla n)-w\cdot\nabla^\alpha\nabla n)\cdot\nabla^\alpha n\mathrm{d}x\right|+c\left|\int_{\mathbb{R}^3} w\cdot\nabla\nabla^\alpha n\cdot\nabla^\alpha n\mathrm{d}x\right|\notag\\
				&\leq C(\|\nabla w\|_{L^\infty}\|\nabla^\alpha n\|_{L^2}+\|\nabla n\|_{L^\infty}\|\nabla^\alpha w\|_{L^2})\|\nabla^\alpha n\|_{L^2}+C\|\text{div} w\|_{L^\infty}\|\nabla^\alpha n\|_{L^2}^2\notag\\
				&\leq C\delta(\|\nabla^\alpha n\|_{L^2}^2+\|\nabla^\alpha w\|_{L^2}^2).\label{lem-n2-1}
			\end{align}
			Similarly, we also have
			\begin{align}
				&\left| c\int_{\mathbb{R}^3} \nabla^\alpha (n\text{div}w)\cdot\nabla^\alpha n \mathrm{d}x\right|\notag\\
				&\leq C(\|\nabla n\|_{L^\infty}\|\nabla^\alpha w\|_{L^2}+\|\text{div}w\|_{L^\infty}\|\nabla^\alpha n\|_{L^2})\|\nabla^\alpha n\|_{L^2}+\|n\|_{L^\infty}\|\nabla^\alpha\text{div}w\|_{L^2}\|\nabla^\alpha n\|_{L^2}\notag\\
				&\leq C\delta(\|\nabla^\alpha n\|_{L^2}^2+\|\nabla^\alpha w\|_{L^2}^2+\|\nabla^{\alpha}\text{div}w\|_{L^2}^2).\label{lem-n2-2}
			\end{align}
			Substituting \eqref{lem-n2-1}  and \eqref{lem-n2-2} into \eqref{lem-n2} gives
			\eqref{lem-n1}.  This completes the proof of Lemma \ref{lem-n}.
		\end{proof}
		
		\begin{lem}\label{lem-w}
			Assume that $(n,w,\phi,\psi)$ is the classical solution of the problem \eqref{nsac}-\eqref{non-f} satisfying the a priori assumption \eqref{a-priori est}, then it holds for $0\leq t\leq T$ and $1\leq \alpha\leq s$ that
			\begin{align}
				&\frac{1}{2}\frac{\mathrm{d}}{\mathrm{d}t}\|\nabla^\alpha w\|_{L^2}^2-c\int_{\mathbb{R}^3} \nabla^\alpha {\rm{div}}w\cdot\nabla^\alpha n \mathrm{d}x-\sigma\int_{\mathbb{R}^3} \nabla^\alpha {\rm{div}}w\cdot\nabla^\alpha \phi \mathrm{d}x\notag\\	
				&\quad+\frac{3}{4}\nu\|\nabla^{\alpha}\nabla w\|_{L^2}^2+\frac{3}{4}(\nu+\eta)\|\nabla^\alpha{\rm{div}} w\|_{L^2}^2\notag\\
				&\leq C\delta(\|\nabla^\alpha n\|_{L^2}^2+\|\nabla^\alpha w\|_{L^2}^2+\|\nabla^\alpha \phi\|_{L^2}^2),\label{lem-w1}
			\end{align}
			where the positive constant $C$ is independent of time.
		\end{lem}
		\begin{proof}
			Applying the operator $\nabla^\alpha$ to $\eqref{nsac}_2$, multiplying the resulting equation by $\nabla^\alpha w$, and then integrating over $\mathbb{R}^3$ give
			\begin{align}
				&\frac{1}{2}\frac{\mathrm{d}}{\mathrm{d}t}\|\nabla^\alpha w\|_{L^2}^2-c\int_{\mathbb{R}^3} \nabla^\alpha\text{div}w\cdot\nabla^\alpha n \mathrm{d}x-\sigma\int_{\mathbb{R}^3} \nabla^\alpha {\rm{div}}w\cdot\nabla^\alpha \phi \mathrm{d}x \notag\\
				&\quad
				+\nu\|\nabla^{\alpha}\nabla w\|_{L^2}^2+(\nu+\eta)\|\nabla^\alpha \text{div} w\|_{L^2}^2\notag\\
				&=-c\int_{\mathbb{R}^3} \nabla^\alpha(w\cdot\nabla w)\cdot \nabla^\alpha w\mathrm{d}x+c\int_{\mathbb{R}^3}  \nabla^\alpha\Big(\frac{n  \nabla n}{1+n}\Big)\cdot\nabla^\alpha w \mathrm{d}x-\sigma\int_{\mathbb{R}^3}  \nabla^\alpha\Big(\frac{\phi\nabla n}{1+n}\Big)\cdot\nabla^\alpha w \mathrm{d}x\notag\\
				&\quad-\int_{\mathbb{R}^3} \nabla^\alpha\Big(\frac{\nu n\Delta w}{1+n}+\frac{(\nu+\eta)n\nabla\text{div}w}{1+n}\Big)\cdot\nabla^\alpha w\mathrm{d}x\notag\\
				&\triangleq I_1+I_2+I_3+I_4.\label{lem-w2}
			\end{align}
			By Lemma \ref{lem-commu} and \eqref{l-inf-est}, the term $I_1$ is estimated  as follows
			\begin{align}
				|I_1|&\leq C\|\nabla^\alpha(w\cdot\nabla w)-w\nabla^\alpha\nabla w\|_{L^2}\|\nabla^\alpha w\|_{L^2}
				+C\|\text{div}w\|_{L^\infty}\|\nabla^\alpha w\|_{L^2}^2\notag\\
				&\leq C\|\nabla w\|_{L^\infty}\|\nabla^\alpha w\|_{L^2}^2
				+C\|\text{div} w\|_{L^\infty}\|\nabla^\alpha w\|_{L^2}^2\notag\\
				&\leq C\delta\|\nabla^\alpha w\|_{L^2}^2.\label{lem-w3}
			\end{align}
			As for the term $I_2$, it is easy to verify
			\begin{align}
				|I_2|&\leq C \left|\int_{\mathbb{R}^3}  \nabla^\alpha\Big(\frac{n}{1+n}\nabla n\Big)\cdot \nabla^\alpha w\mathrm{d}x\right|\notag\\
				&\leq C\left|\int_{\mathbb{R}^3}\Big(\nabla^\alpha\Big(\frac{n}{1+n}\nabla n\Big)-\frac{n}{1+n}\nabla^\alpha \nabla n\nabla\Big)\cdot \nabla^\alpha w\mathrm{d}x\right|+C\left|\int_{\mathbb{R}^3}  \frac{n}{1+n}\nabla^\alpha \nabla n \cdot \nabla^\alpha w\mathrm{d}x\right|\notag\\
				&\leq C\Big(\Big\|\nabla\Big(\frac{n}{1+n}\Big)\Big\|_{L^\infty}\|\nabla^\alpha n\|_{L^2}+\Big\|\nabla^\alpha\Big(\frac{n}{1+n}\Big)\Big\|_{L^2}\|\nabla n\|_{L^\infty}\Big)\|\nabla^\alpha w\|_{L^2}\notag\\
				&\quad+C\Big\|\nabla\Big(\frac{n}{1+n}\Big)\Big\|_{L^\infty}\|\nabla^\alpha n\|_{L^2}\|\nabla^\alpha w\|_{L^2}
				+C\Big\|\frac{n}{1+n}\Big\|_{L^\infty}\|\nabla^\alpha n\|_{L^2}\|\nabla^\alpha\text{div}w\|_{L^2}\notag\\
				&\leq C\delta(\|\nabla^\alpha n\|_{L^2}^2+\|\nabla^\alpha w\|_{L^2}^2+\|\nabla^{\alpha}\text{div}w\|_{L^2}^2).
			\end{align}	
			Similarly, we have the following results for the term $I_3$:
			\begin{align}
				|I_3|&\leq C\Big\|\nabla^\alpha\Big(\frac{\phi}{1+n}\cdot\nabla n\Big)-\frac{\phi}{1+n}\nabla^\alpha\nabla n\Big\|_{L^2}\|\nabla^\alpha w\|_{L^2}+\int_{\mathbb{R}^3} \frac{\phi}{1+n} \nabla^\alpha\nabla n \cdot\nabla^\alpha w\mathrm{d}x \notag\\
				&\leq C(\|\nabla \phi\|_{L^\infty}\|\nabla^\alpha n\|_{L^2}+\|\nabla n\|_{L^\infty}\|\nabla^\alpha n\|_{L^2}+\|\nabla n\|_{L^\infty}\|\nabla^\alpha \phi\|_{L^2})\|\nabla^\alpha w\|_{L^2}\notag\\
				&\quad\quad +C\|\phi\|_{L^\infty}\|\nabla^\alpha n\|_{L^2}\|\nabla^\alpha\text{div}w\|_{L^2}+C\|\nabla \phi\|_{L^\infty}\|\nabla^\alpha n\|_{L^2}\|\nabla^\alpha w\|_{L^2} \notag\\
				&\leq C\delta(\|\nabla^\alpha n\|_{L^2}^2+\|\nabla^\alpha w\|_{L^2}^2+\|\nabla^\alpha \text{div} w\|_{L^2}^2+\|\nabla^\alpha \phi\|_{L^2}^2).
			\end{align}
			Finally, for the term of $I_4$, we have
			\begin{align}
				|I_4|&\leq C\left|\int_{\mathbb{R}^3}\nabla^\alpha \Big(\frac{n}{1+n}\Delta w\Big)\cdot  \nabla^\alpha w\mathrm{d}x\right|+C\left|\int_{\mathbb{R}^3} \nabla^\alpha \Big(\frac{n}{1+n}\nabla \text{div}w\Big)\cdot\nabla^\alpha w\mathrm{d}x\right|\notag\\
				&\leq C\delta(\|\nabla^\alpha n\|_{L^2}^2+\|\nabla^{\alpha}\nabla w\|_{L^2}^2+\|\nabla^{\alpha}\text{div} w\|_{L^2}^2).\label{lem-w4}
			\end{align}
			
			Combining \eqref{lem-w3}-\eqref{lem-w4}, we conclude from \eqref{lem-w2} that
			\begin{align}
				&\frac{1}{2}\frac{\mathrm{d}}{\mathrm{d}t}\|\nabla^\alpha w\|_{L^2}^2-c\int_{\mathbb{R}^3} \nabla^\alpha\text{div}w\cdot\nabla^\alpha n \mathrm{d}x-\sigma\int_{\mathbb{R}^3} \nabla^\alpha {\rm{div}}w\cdot\nabla^\alpha \phi \mathrm{d}x\notag\\
				&\quad+\nu\|\nabla^{\alpha}\nabla w\|_{L^2}^2+(\nu+\eta)\|\nabla^\alpha \text{div} w\|_{L^2}^2\notag\\
				&\leq C\delta(\|\nabla^\alpha n\|_{L^2}^2+\|\nabla^\alpha w\|_{L^2}^2+\|\nabla^{\alpha}\nabla w\|_{L^2}^2+\|\nabla^\alpha \text{div}w\|_{L^2}^2+\|\nabla^\alpha \phi\|_{L^2}^2).\label{lem-w5}
			\end{align}
			Due to the smallness of $\delta$, we get \eqref{lem-w1}. This completes the proof of Lemma \ref{lem-w}.
		\end{proof}
		
		\begin{lem}\label{lem-th}
			Assume that $(n,w,\phi,\psi)$ is the classical solution of the problem \eqref{nsac}-\eqref{non-f} satisfying the a priori assumption \eqref{a-priori est}, then it holds for $0\leq t\leq T$ and $1\leq \alpha\leq s$ that
			\begin{align}
				&\frac{1}{2}\frac{\mathrm{d}}{\mathrm{d}t}\Big(\|\nabla^\alpha\phi\|_{L^2}^2+\frac{1}{1+n}\|\nabla^\alpha\psi\|_{L^2}^2\Big)+\frac{1}{2\tau}\|\nabla^\alpha\psi\|_{L^2}^2+\sigma\int_{\mathbb{R}^3} \nabla^\alpha {\rm{div}}w\cdot\nabla^\alpha \phi \mathrm{d}x \notag\\
				&\leq C\delta(\|\nabla^\alpha n\|_{L^2}^2+\|\nabla^\alpha w\|_{L^2}^2+\|\nabla^{\alpha}\nabla w\|_{L^2}^2+\|\nabla^\alpha {\rm{div}}w\|_{L^2}^2+\|\nabla^\alpha \phi\|_{L^2}^2),\label{lem-th1}
			\end{align}
			where the positive constant $C$ is independent of time.
		\end{lem}
		
		\begin{proof}
			Applying the operator $\nabla^\alpha$ to $\eqref{nsac}_3$, multiplying the resulting equation by $\nabla^\alpha \phi$, and then integrating over $\mathbb{R}^3$, we obtain
			\begin{align}
				&\frac{1}{2}\frac{\mathrm{d}}{\mathrm{d}t}\|\nabla^\alpha \phi\|_{L^2}^2+\sigma\int_{\mathbb{R}^3} \nabla^\alpha {\rm{div}}w\cdot\nabla^\alpha \phi \mathrm{d}x+b\int_{\mathbb{R}^3}\nabla^\alpha{\rm{div}}\psi\cdot\nabla^\alpha\phi \mathrm{d}x\notag\\
				&=-c\int_{\mathbb{R}^3} \nabla^\alpha(w\cdot\nabla \phi)\cdot \nabla^\alpha \phi \mathrm{d}x-c(\gamma-1)\int_{\mathbb{R}^3}  \nabla^\alpha(\phi\text{div}w)\cdot\nabla^\alpha \phi \mathrm{d}x\notag\\
				&\quad+\frac{\sigma}{c}\int_{\mathbb{R}^3} \nabla^\alpha\Big(\frac{2\nu|\mathbb{D}w|^2+\eta({\rm{div}}w)^2}{1+n}\Big)\cdot\nabla^\alpha \phi \mathrm{d}x+b\int_{\mathbb{R}^3}\nabla^\alpha\Big(\frac{n\text{div}\psi}{1+n}\Big)\cdot\nabla^\alpha \phi \mathrm{d}x\notag\\
				&\triangleq J_1+J_2+J_3+J_4.\label{lem-th2}
			\end{align}
			First, by Lemma \ref{lem-commu}, a direct calculation gives
			\begin{align}
				|J_1|&\leq C\|\nabla^\alpha(w\cdot\nabla \phi)-w\nabla^\alpha\nabla \phi\|_{L^2}\|\nabla^\alpha \phi\|_{L^2}
				+C\|\text{div}w\|_{L^\infty}\|\nabla^\alpha \phi\|_{L^2}^2\notag\\
				&\leq C(\|\nabla w\|_{L^\infty}\|\nabla^\alpha \phi\|_{L^2}+\|\nabla \phi\|_{L^\infty}\|\nabla^\alpha w\|_{L^2})\|\nabla^\alpha \phi\|_{L^2}
				+C\|\text{div} w\|_{L^\infty}\|\nabla^\alpha \phi\|_{L^2}^2\notag\\
				&\leq C\delta(\|\nabla^\alpha w\|_{L^2}^2+\|\nabla^\alpha \phi\|_{L^2}^2).\label{lem-th3}
			\end{align}
			In view of \eqref{lem-commu1} in Lemma \ref{lem-commu}, we have
			\begin{align}
				|J_2|&\leq C(\|\text{div}w\|_{L^\infty}\|\nabla^\alpha \phi\|_{L^2}+\|\phi\|_{L^\infty}\|\nabla^\alpha\text{div}w\|_{L^2})\|\nabla^\alpha \phi\|_{L^2}\notag\\
				&\leq C\delta(\|\nabla^\alpha \phi\|_{L^2}^2+\|\nabla^{\alpha}\text{div}w\|_{L^2}^2).
			\end{align}	
			For the term $J_3$, it holds
			\begin{align}
				|J_3|&\leq C\Big\|\nabla^\alpha\Big(\frac{1}{1+n}|\nabla w|^2\Big)-\frac{1}{1+n}\nabla^\alpha(|\nabla w|^2)\Big\|_{L^2}\|\nabla^\alpha \phi\|_{L^2}+\Big|\int_{\mathbb{R}^3} \frac{1}{1+n}\nabla^\alpha(|\nabla w|^2)\cdot\nabla^\alpha \phi \mathrm{d}x\Big| \notag\\
				&\leq C(\|\nabla n\|_{L^\infty}\|\nabla^\alpha w\|_{L^2}+\|\nabla w\|_{L^\infty}\|\nabla^\alpha n\|_{L^2})\|\nabla^\alpha \phi\|_{L^2}+C\delta\|\nabla^\alpha \phi\|_{L^2}\|\nabla^\alpha \nabla w\|_{L^2} \notag\\
				&\leq C\delta(\|\nabla^\alpha n\|_{L^2}^2+\|\nabla^\alpha w\|_{L^2}^2+\|\nabla^\alpha \nabla w\|_{L^2}^2+\|\nabla^\alpha \phi\|_{L^2}^2).
			\end{align}
			As for $J_4$, we rewrite it as follows
			$$
			\begin{aligned}
				J_4&=b\int_{\mathbb{R}^3}\Big(\nabla^\alpha\Big(\frac{n\text{div}\psi}{1+n}\Big)- \frac{n}{1+n}\nabla^\alpha\text{div}\psi\Big)\cdot\nabla^\alpha \phi \mathrm{d}x
				+b\int_{\mathbb{R}^3}\frac{n}{1+n}\nabla^\alpha\text{div}\psi \cdot\nabla^\alpha\phi \mathrm{d}x\\
				&\triangleq J_{41}+J_{42}.
			\end{aligned}
			$$
			It is easy to verify that
			$$
			\begin{aligned}
				|J_{41}|&\lesssim \Big(\Big\|\nabla\Big(\frac{n}{1+n}\Big)\Big\|_{L^\infty}\|\nabla^\alpha\psi\|_{L^2}+
				\Big\|\nabla^\alpha\Big(\frac{n}{1+n}\Big)\Big\|_{L^2}\|\text{div}\psi\|_{L^\infty}\Big)\|\nabla^\alpha\phi\|_{L^2}\\
				&\lesssim \delta(\|\nabla^\alpha\psi\|_{L^2}+\|\nabla^\alpha n\|_{L^2})\|\nabla^\alpha \phi\|_{L^2}\\
				&\lesssim \delta(\|\nabla^\alpha n\|_{L^2}^2+\|\nabla^\alpha\phi\|_{L^2}^2+\|\nabla^\alpha\psi\|_{L^2}^2).
			\end{aligned}
			$$
			As for $J_{42}$, in terms of integration by parts, it holds
			$$
			\begin{aligned}
				J_{42}&=-b\int_{\mathbb{R}^3}\nabla\Big(\frac{n}{1+n}\Big)\nabla^\alpha\psi \cdot\nabla^\alpha\phi \mathrm{d}x-b\int_{\mathbb{R}^3}\frac{n}{1+n}\nabla^\alpha\psi \cdot\nabla^\alpha\nabla\phi \mathrm{d}x\\
				&\triangleq J_{421}+J_{422}.
			\end{aligned}
			$$
			The first term $J_{421}$ is bounded by
			\begin{equation*}
				|J_{421}|\lesssim \|\nabla n\|_{L^\infty}\|\nabla^\alpha \psi\|_{L^2} \|\nabla^\alpha \phi\|_{L^2}\lesssim \delta(\|\nabla^\alpha\phi\|_{L^2}^2+\|\nabla^\alpha\psi\|_{L^2}^2).
			\end{equation*}
			For	the second term $J_{422}$, we utilize $\eqref{nsac}_4$ to get
			$$
			\begin{aligned}
				J_{422}&=-b\int_{\mathbb{R}^3}\frac{n}{1+n}\nabla^\alpha\psi \cdot\nabla^\alpha\nabla\phi \mathrm{d}x\\
				&=\int_{\mathbb{R}^3}\frac{n}{1+n}\nabla^\alpha\psi \cdot\nabla^\alpha(\partial_t\psi+\frac{1}{\tau}\psi) \mathrm{d}x\\
				&=\frac{1}{2}\frac{\mathrm{d}}{\mathrm{d}t}\int_{\mathbb{R}^3}\frac{n}{1+n}|\nabla^\alpha\psi|^2\mathrm{d}x-\frac{1}{2}\int_{\mathbb{R}^3}\Big(\frac{n}{1+n}\Big)_t|\nabla^\alpha\psi|^2\mathrm{d}x+\int_{\mathbb{R}^3}\frac{n}{\tau(1+n)}|\nabla^\alpha\psi|^2\mathrm{d}x\\
				&=\frac{1}{2}\frac{\mathrm{d}}{\mathrm{d}t}\int_{\mathbb{R}^3}\frac{n}{1+n}|\nabla^\alpha\psi|^2\mathrm{d}x-\frac{1}{2}\int_{\mathbb{R}^3}\frac{\partial_tn}{(1+n)^2}|\nabla^\alpha\psi|^2\mathrm{d}x+\int_{\mathbb{R}^3}\frac{n}{\tau(1+n)}|\nabla^\alpha\psi|^2\mathrm{d}x\\
				&=\frac{1}{2}\frac{\mathrm{d}}{\mathrm{d}t}\int_{\mathbb{R}^3}\frac{n}{1+n}|\nabla^\alpha\psi|^2\mathrm{d}x-\frac{1}{2}\int_{\mathbb{R}^3}\frac{f_1-c\,\text{div}w}{(1+n)^2}|\nabla^\alpha\psi|^2\mathrm{d}x+\int_{\mathbb{R}^3}\frac{n}{\tau(1+n)}|\nabla^\alpha\psi|^2\mathrm{d}x\\
				&\leq \frac{1}{2}\frac{\mathrm{d}}{\mathrm{d}t}\int_{\mathbb{R}^3}\frac{n}{1+n}|\nabla^\alpha\psi|^2\mathrm{d}x+C\delta \|\nabla^\alpha\psi\|_{L^2}^2.
			\end{aligned}
			$$
			It then concludes from above to prove
			\begin{equation}\label{lem-th4}
				|J_4|\leq |J_{41}|+|J_{421}|+|J_{422}|\leq \frac{1}{2}\frac{\mathrm{d}}{\mathrm{d}t}\int_{\mathbb{R}^3}\frac{n}{1+n}|\nabla^\alpha\psi|^2\mathrm{d}x+C\delta(\|\nabla^\alpha\phi\|_{L^2}^2+\|\nabla^\alpha\psi\|_{L^2}^2).
			\end{equation}
			In terms of \eqref{lem-th3}-\eqref{lem-th4} and \eqref{lem-th2}, we verify that
			\begin{align}
				&\frac{1}{2}\frac{\mathrm{d}}{\mathrm{d}t}\|\nabla^\alpha \phi\|_{L^2}^2+\sigma\int_{\mathbb{R}^3} \nabla^\alpha {\rm{div}}w\cdot\nabla^\alpha \phi \mathrm{d}x+b\int_{\mathbb{R}^3}\nabla^\alpha{\rm{div}}\psi\cdot\nabla^\alpha\phi \mathrm{d}x\notag\\
				&\leq C\delta(\|\nabla^\alpha n\|_{L^2}^2+\|\nabla^\alpha w\|_{L^2}^2+\|\nabla^{\alpha}\nabla w\|_{L^2}^2+\|\nabla^\alpha \text{div}w\|_{L^2}^2+\|\nabla^\alpha \phi\|_{L^2}^2+\|\nabla^{\alpha}\psi\|_{L^2}^2)\notag\\
				&\quad+\frac{1}{2}\frac{\mathrm{d}}{\mathrm{d}t}\int_{\mathbb{R}^3}\frac{n}{1+n}|\nabla^\alpha\psi|^2\mathrm{d}x.\label{lem-th51}
			\end{align}
			Applying the operator $\nabla^\alpha$ to $\eqref{nsac}_4$, multiplying it by $\nabla^\alpha \psi$, and then integrating over $\mathbb{R}^3$, we obtain
			\begin{equation}\label{lemn1}
				\begin{aligned}
					&\frac{1}{2}\frac{\mathrm{d}}{\mathrm{d}t}\|\nabla^\alpha \psi\|_{L^2}^2+\frac{1}{\tau}\|\nabla^\alpha\psi\|_{L^2}^2-
					b\int_{\mathbb{R}^3}\nabla^\alpha{\rm{div}}\psi\cdot\nabla^\alpha\phi \mathrm{d}x=0.
				\end{aligned}
			\end{equation}
			It then follows from \eqref{lem-th51} and \eqref{lemn1} that
			\begin{align*}
				&\frac{1}{2}\frac{\mathrm{d}}{\mathrm{d}t}\Big(\|\nabla^\alpha\phi\|_{L^2}^2+\frac{1}{1+n}\|\nabla^\alpha\psi\|_{L^2}^2\Big)+\frac{1}{\tau}\|\nabla^\alpha\psi\|_{L^2}^2+\sigma\int_{\mathbb{R}^3} \nabla^\alpha {\rm{div}}w\cdot\nabla^\alpha \phi \mathrm{d}x\\
				&\leq C\delta(\|\nabla^\alpha n\|_{L^2}^2+\|\nabla^\alpha w\|_{L^2}^2+\|\nabla^{\alpha}\nabla w\|_{L^2}^2+\|\nabla^\alpha \text{div}w\|_{L^2}^2+\|\nabla^\alpha \phi\|_{L^2}^2+\|\nabla^{\alpha}\psi\|_{L^2}^2).
			\end{align*}
			Since  $\delta$ is small, we get
			\begin{align}
				&\frac{1}{2}\frac{\mathrm{d}}{\mathrm{d}t}\Big(\|\nabla^\alpha\phi\|_{L^2}^2+\frac{1}{1+n}\|\nabla^\alpha\psi\|_{L^2}^2\Big)+\frac{1}{2\tau}\|\nabla^\alpha\psi\|_{L^2}^2+\sigma\int_{\mathbb{R}^3} \nabla^\alpha {\rm{div}}w\cdot\nabla^\alpha \phi \mathrm{d}x\notag\\
				&\leq C\delta(\|\nabla^\alpha n\|_{L^2}^2+\|\nabla^\alpha w\|_{L^2}^2+\|\nabla^{\alpha}\nabla w\|_{L^2}^2+\|\nabla^\alpha \text{div}w\|_{L^2}^2+\|\nabla^\alpha \phi\|_{L^2}^2).\label{lem-th5}
			\end{align}
			This completes the proof of Lemma \ref{lem-th}.
		\end{proof}
		
		\begin{lem}\label{lem-nx}
			Assume that $(n,w,\phi,\psi)$ is the classical solution of the problem \eqref{nsac}-\eqref{non-f} satisfying the a priori assumption \eqref{a-priori est}, then it holds for $0\leq t\leq T$ and $1\leq \alpha\leq s$ that
			\begin{align}
				&\frac{c}{2}\|\nabla^\alpha n\|_{L^2}^2+\frac{\mathrm{d}}{\mathrm{d}t}\int_{\mathbb{R}^3} \nabla^{\alpha-1}w\cdot\nabla^\alpha n\mathrm{d}x\notag\\
				&\leq C(\|\nabla^\alpha w\|_{L^2}^2+\|\nabla^{\alpha}\nabla w\|_{L^2}^2
				+\|\nabla^{\alpha}{\rm{div}} w\|_{L^2}^2)+\frac{\sigma^2}{c}\|\nabla^\alpha \phi\|_{L^2}^2\notag\\
				&\quad +C\delta(\|\nabla^2 n\|_{H^1}^2+\|\nabla^2 w\|_{H^2}^2+\|\nabla^2\phi\|_{L^2}^2),\label{lem-nx1}
			\end{align}	
			where the positive constant $C$ is independent of time.
		\end{lem}
		\begin{proof}
			Applying the operator $\nabla^{\alpha-1}$ to $\eqref{nsac}_2$, multiplying the resulting equation by $\nabla^{\alpha} n$, and then integrating over $\mathbb{R}^3$, we obtain
			\begin{align}
				&c\|\nabla^\alpha n\|_{L^2}^2+\frac{\mathrm{d}}{\mathrm{d}t}\int_{\mathbb{R}^3} \nabla^{\alpha-1}w\cdot\nabla^\alpha n\mathrm{d}x\notag\\
				&=\int_{\mathbb{R}^3} (\nu\nabla^{\alpha-1}\Delta w+(\nu+\eta)\nabla^{\alpha}{\rm{div}}w-\sigma \nabla^{\alpha}\phi)\cdot\nabla^\alpha n\mathrm{d}x+\int_{\mathbb{R}^3} \nabla^{\alpha-1}f_2\cdot\nabla^{\alpha}n\mathrm{d}x\notag\\
				&\quad-c\int_{\mathbb{R}^3}(\nabla^{\alpha}{\rm{div}}w+\nabla^{\alpha}f_1)\cdot \nabla^{\alpha-1}w\mathrm{d}x\notag\\
				&\triangleq J_5+J_6+J_7.\label{lem-nx2}
			\end{align}
			By Young's inequality, the term $J_5$ can be estimated as
			\begin{align}
				|J_5|&\leq  (\nu+\eta)(\|\nabla^\alpha \nabla w\|_{L^2}+\|\nabla^\alpha \text{div}w\|_{L^2})\|\nabla^\alpha n\|_{L^2}+\sigma\|\nabla^\alpha \phi\|_{L^2}\|\nabla^\alpha n\|_{L^2}\notag\\
				&\leq\varepsilon\|\nabla^{\alpha}n\|_{L^2}^2+C_{\varepsilon}(\|\nabla^\alpha \nabla w\|_{L^2}^2+\|\nabla^\alpha \text{div}w\|_{L^2}^2)+\frac{c}{4}\|\nabla^\alpha n\|_{L^2}^2+\frac{\sigma^2}{c}\|\nabla^\alpha \phi\|_{L^2}^2\notag\\
				&\leq (\varepsilon+\frac{c}{4})\|\nabla^\alpha n\|_{L^2}^2+C(\|\nabla^\alpha \nabla w\|_{L^2}^2+\|\nabla^\alpha \text{div}w\|_{L^2}^2)+\frac{\sigma^2}{c}\|\nabla^\alpha \phi\|_{L^2}^2,
				\label{lem-nx3}
			\end{align}
			where $\varepsilon>0$ is arbitrary. For the term $J_6$, recalling the definition of $f_2$ in \eqref{non-f}, we  deal with the first term of $J_6$ as follows.
			\begin{align*}
				&\left|\int_{\mathbb{R}^3}\nabla^{\alpha-1}(w\cdot\nabla w)\cdot\nabla^\alpha n\mathrm{d}x\right|\\
				&\leq \left|\int_{\mathbb{R}^3}[\nabla^{\alpha-1}(w\cdot\nabla w)-w\cdot\nabla^\alpha w]\cdot\nabla^\alpha n\mathrm{d}x\right|+\left|\int_{\mathbb{R}^3}w\cdot\nabla^\alpha w\cdot\nabla^\alpha n\mathrm{d}x\right|\\
				&\leq C(\|\nabla w\|_{L^\infty}\|\nabla^{\alpha-1}w\|_{L^2}+\|\nabla^{\alpha-1}w\|_{L^2}\|\nabla w\|_{L^\infty})\|\nabla^\alpha n\|_{L^2}+C\|w\|_{L^\infty}\|\nabla^\alpha w\|_{L^2}\|\nabla^\alpha n\|_{L^2}\\
				&\leq C(\|\nabla^2 w\|_{H^1}\|\nabla^{\alpha-1}w\|_{L^2}+\|\nabla^{\alpha-1}w\|_{L^2}\|\nabla^2 w\|_{H^1})\|\nabla^\alpha n\|_{L^2}+C\|\nabla w\|_{H^1}\|\nabla^\alpha w\|_{L^2}\|\nabla^\alpha n\|_{L^2}\\
				&\leq C\delta\|\nabla^2w\|_{H^1}\|\nabla^\alpha n\|_{L^2}+C\delta\|\nabla^\alpha w\|_{L^2}\|\nabla^\alpha n\|_{L^2}\\
				&\leq C\delta(\|\nabla^\alpha n\|_{L^2}^2+\|\nabla^2w\|_{H^1}^2+\|\nabla^\alpha w\|_{L^2}^2),
			\end{align*}
			where we have used the \emph{a priori} assumption \eqref{a-priori est} that for $1\leq \alpha \leq s$ that $\|\nabla^{\alpha-1}w\|_{L^2}\leq \delta$. The remainders of $J_6$ can be estimated similarly. Thus, we have
			\begin{align}
				|J_6|&\leq \Big|\int_{\mathbb{R}^3} \nabla^{\alpha-1}\Big(-cw\cdot\nabla w+\frac{(cn-\sigma\phi)\nabla n}{1+n}-\frac{\nu n\Delta w+(\nu+\eta)n\nabla{\rm{div}}w}{1+n}\Big)\cdot \nabla^{\alpha}n\mathrm{d}x\Big|\notag\\
				&\leq C\delta(\|\nabla^\alpha n\|_{L^2}^2+\|\nabla^{\alpha} w\|_{L^2}^2
				+\|\nabla^\alpha\nabla w\|_{L^2}^2\notag\\
				&\quad+\|\nabla^\alpha\text{div} w\|_{L^2}^2+\|\nabla^2n\|_{H^1}^2+\|\nabla^2 w\|_{H^2}^2+\|\nabla^2\phi\|_{H^1}^2).
			\end{align}
			Finally, we deal with the term $J_{7}$. We have
			\begin{align}
				|J_7|
				&\leq \Big| c\int_{\mathbb{R}^3}(\nabla^{\alpha}w+\nabla^{\alpha}(nw))\cdot \nabla^{\alpha}w \mathrm{d}x\Big|\notag\\
				&\leq C\|\nabla^\alpha w\|_{L^2}^2+C(\|n\|_{L^\infty}\|\nabla^{\alpha} w\|_{L^2}+\|w\|_{L^\infty}\|\nabla^\alpha n\|_{L^2})\|\nabla^\alpha w\|_{L^2}\notag\\
				&\leq C\|\nabla^\alpha w\|_{L^2}^2+C\delta(\|\nabla^\alpha n\|_{L^2}^2+\|\nabla^\alpha w\|_{L^2}^2).\label{lem-nx4}
			\end{align}
			
			Combining the above estimates together yields
			\begin{align*}
				&c\|\nabla^\alpha n\|_{L^2}^2+\frac{\mathrm{d}}{\mathrm{d}t}\int_{\mathbb{R}^3} \nabla^{\alpha-1}w\cdot\nabla^\alpha n\mathrm{d}x\\
				&\leq (C\delta+\varepsilon+\frac{c}{4})\|\nabla^{\alpha}n\|_{L^2}^2+C(\|\nabla^\alpha w\|_{L^2}^2+\|\nabla^{\alpha}\nabla w\|_{L^2}^2
				+\|\nabla^{\alpha}\text{div} w\|_{L^2}^2)+\frac{\sigma^2}{c}\|\nabla^\alpha \phi\|_{L^2}^2\\
				&\quad +C\delta(\|\nabla^2 n\|_{H^1}^2+\|\nabla^2 w\|_{H^2}^2+\|\nabla^2\phi\|_{H^1}^2).
			\end{align*}
			Due to the smallness of $\delta$, we can chose $\varepsilon$ sufficiently small to get \eqref{lem-nx1}. This completes the proof of Lemma \ref{lem-nx}.
		\end{proof}

		\begin{lem}\label{lem-1nx}
			Assume $(n,w,\phi,\psi)$ that is the classical solution of the problem \eqref{nsac}-\eqref{non-f} satisfying the a priori assumption \eqref{a-priori est}, then it holds for $0\leq t\leq T$ and $1\leq \alpha\leq s$ that
			\begin{align}
				&\frac{3b}{4}\|\nabla^\alpha \phi\|_{L^2}^2+\frac{\mathrm{d}}{\mathrm{d}t}\int_{\mathbb{R}^3} \nabla^{\alpha-1}\psi\cdot\nabla^\alpha \phi \mathrm{d}x\notag\\
				&\leq C\delta(\|\nabla^\alpha n\|_{L^2}^2+\|\nabla^\alpha\nabla w\|_{L^2}^2+\|\nabla^\alpha\emph{div} w\|_{L^2}^2)\notag\\
				&\quad+C(\|\nabla^\alpha w\|_{L^2}^2+\|\nabla^{\alpha-1}\psi\|_{L^2}^2+\|\nabla^\alpha\psi\|_{L^2}^2),\label{lem-1nx1}
			\end{align}	
			where the positive constant $C$ is independent of time.
		\end{lem}
		\begin{proof}
			Applying the operator $\nabla^{\alpha-1}$ to $\eqref{nsac}_4$, multiplying the resulting equation by $\nabla^{\alpha} \phi$, and integrating over $\mathbb{R}^3$, we obtain
			\begin{align}
				&b\|\nabla^\alpha \phi\|_{L^2}^2+\frac{\mathrm{d}}{\mathrm{d}t}\int_{\mathbb{R}^3} \nabla^{\alpha-1}\psi\cdot\nabla^\alpha \phi \mathrm{d}x\notag\\
				&=-\int_{\mathbb{R}^3}\nabla^\alpha(\sigma\text{div}w+b\,\text{div}\psi)\cdot\nabla^{\alpha-1}\psi\mathrm{d}x+\int_{\mathbb{R}^3}\nabla^{\alpha}f_3\cdot\nabla^{\alpha-1}\psi\mathrm{d}x-\frac{1}{\tau}\int_{\mathbb{R}^3}\nabla^\alpha \phi \cdot\nabla^{\alpha-1}\psi\mathrm{d}x\notag\\
				&\triangleq J_8+J_9+J_{10}.\label{lem-nx12}
			\end{align}
			By Young's inequality, the term $J_8$ is estimated as follows
			\begin{align}
				|J_8|&\leq \sigma\|\nabla^\alpha w\|_{L^2}\|\nabla^{\alpha}\psi\|_{L^2}
				+b\|\nabla^\alpha\psi\|_{L^2}^2\notag\\
				&\leq\frac{1}{2}\sigma\|\nabla^{\alpha}\psi\|_{L^2}^2+\frac{1}{2}\sigma\|\nabla^\alpha w\|_{L^2}^2+b\|\nabla^\alpha\psi\|_{L^2}^2.\label{lem-nx13}
			\end{align}
			For the term $J_9$, we have
			\begin{align}
				|J_9|&\leq \Big|\int_{\mathbb{R}^3} \nabla^{\alpha-1}\Big(-cw\cdot\nabla\phi-c(\gamma-1)\phi{\rm{div}}w+\frac{\sigma(2\nu|\mathbb{D}w|^2+\eta({\rm{div}}w)^2)}{c(1+n)}+\frac{b n\text{div} \psi}{1+n}\Big)\cdot\nabla^{\alpha}\psi\mathrm{d}x\Big|\notag\\
				&\leq \Big| \int_{\mathbb{R}^3} \nabla^{\alpha}\Big(-cw\cdot\nabla\phi-c(\gamma-1)\phi{\rm{div}}w+\frac{\sigma(2\nu|\mathbb{D}w|^2+\eta({\rm{div}}w)^2)}{c(1+n)}+\frac{b n\text{div} \psi}{1+n}\Big)\cdot\nabla^{\alpha-1}\psi\mathrm{d}x\Big|\notag\\
				&\leq C\delta (\|\nabla^\alpha n\|_{L^2}^2+\|\nabla^\alpha w\|_{L^2}^2+\|\nabla^\alpha\nabla w\|_{L^2}^2+\|\nabla^\alpha\text{div} w\|_{L^2}^2\notag\\
				&\quad+\|\nabla^\alpha\phi\|_{L^2}^2+\|\nabla^{\alpha-1}\psi\|_{L^2}^2+\|\nabla^\alpha\psi\|_{L^2}^2).
			\end{align}
			Finally, we deal with the term $J_{10}$ as follows:
			\begin{equation}\label{lem-nx14}
				\begin{aligned}
					|J_{10}|&\leq \frac{1}{\tau}\|\nabla^\alpha\phi\|_{L^2}\|\nabla^{\alpha-1}\psi\|_{L^2}\leq\frac{1}{4}b\|\nabla^\alpha\phi\|_{L^2}^2+ \frac{1}{b\tau^2}
					\|\nabla^{\alpha-1}\psi\|_{L^2}^2.
				\end{aligned}
			\end{equation}
			Combining the above estimates and noticing the smallness of $\delta$ yield
			\begin{align*}
				\frac{3b}{4}\|\nabla^\alpha \phi\|_{L^2}^2+\frac{\mathrm{d}}{\mathrm{d}t}\int_{\mathbb{R}^3} \nabla^{\alpha-1}\psi\cdot\nabla^\alpha \phi \mathrm{d}x
				&\leq C\delta(\|\nabla^\alpha n\|_{L^2}^2+\|\nabla^\alpha\nabla w\|_{L^2}^2+\|\nabla^\alpha\text{div} w\|_{L^2}^2)\\
				&\quad+C(\|\nabla^\alpha w\|_{L^2}^2+\|\nabla^{\alpha-1}\psi\|_{L^2}^2+\|\nabla^\alpha\psi\|_{L^2}^2).
			\end{align*}
			This completes the proof of Lemma \ref{lem-1nx}.
		\end{proof}

		With the above lemmas in hand, we are now in a position to obtain the uniform bound of the classical solution under the \emph{a priori} assumption \eqref{a-priori est}.
		\begin{prop}\label{pro-est}
			Let $s\geq 3$ be an integer and $\delta$ be a sufficiently small positive constant. Under the conditions of Theorem \ref{thm1},  assume further that $(n,w,\phi,\psi)$ is the classical solution of the problem \eqref{nsac}-\eqref{non-f} for any time $T>0$ and satisfying the a priori assumption
			$$
			\sup_{0\leq t\leq T}\|(n,w,\phi,\psi)(t)\|_{H^s}\leq \delta.
			$$
			Then it holds
			\begin{equation}\label{pro-est1}
				\|(n,w,\phi,\psi)(t)\|_{H^s}^2+\int_0^t\left(\|\nabla (n,\phi)(r)\|_{H^{s-1}}^2+\|\nabla w(r)\|_{H^{s}}^2+\|\psi(r)\|_{H^{s}}^2\right)\mathrm{d}r\leq C\delta_0^2,
			\end{equation}	
			where $\delta_0$ is a small positive constant given by \eqref{in-data-e0}.	
		\end{prop}
		
		\begin{proof}
			We conclude by Lemmas \ref{lem-n}-\ref{lem-th} that
			\begin{equation*}
				\begin{aligned}
					&\frac{1}{2}\frac{\mathrm{d}}{\mathrm{d}t}\left(\|\nabla^\alpha n\|_{L^2}^2+\|\nabla^\alpha w\|_{L^2}^2+\|\nabla^\alpha \phi\|_{L^2}^2+\frac{1}{1+n}\|\nabla^{\alpha}\psi\|_{L^2}^2\right)\\
					&\quad+\frac{3\nu}{4}\|\nabla^{\alpha}\nabla w\|_{L^2}^2+\frac{3(\nu+\eta)}{4}\|\nabla^\alpha \text{div} w\|_{L^2}^2
					+\frac{1}{2\tau}\|\nabla^{\alpha} \psi\|_{L^2}^2\\
					&\leq C\delta (\|\nabla^\alpha n\|_{L^2}^2+\|\nabla^\alpha w\|_{L^2}^2+\|\nabla^{\alpha}\nabla w\|_{L^2}^2+\|\nabla^{\alpha}\text{div} w\|_{L^2}^2+\|\nabla^\alpha \phi\|_{L^2}^2).
				\end{aligned}
			\end{equation*}
			Since $\delta$ is sufficiently small, one has
			\begin{align}
				&\frac{1}{2}\frac{\mathrm{d}}{\mathrm{d}t}\left(\|\nabla^\alpha n\|_{L^2}^2+\|\nabla^\alpha w\|_{L^2}^2+\|\nabla^\alpha \phi\|_{L^2}^2+\frac{1}{1+n}\|\nabla^{\alpha}\psi\|_{L^2}^2\right)\notag\\
				&\quad+\frac{\nu}{2}\|\nabla^{\alpha}\nabla w\|_{L^2}^2+\frac{\nu+\eta}{2}\|\nabla^\alpha \text{div} w\|_{L^2}^2
				+\frac{1}{2\tau}\|\nabla^{\alpha} \psi\|_{L^2}^2\notag\\
				&\leq C\delta (\|\nabla^\alpha n\|_{L^2}^2+\|\nabla^\alpha w\|_{L^2}^2+\|\nabla^\alpha \phi\|_{L^2}^2).\label{THZ2}
			\end{align}
			For the above inequalities, summing $\alpha$ from $1$ to $s$ gives
			\begin{align}
				&\frac{1}{2}\frac{\mathrm{d}}{\mathrm{d}t}\left(\|\nabla n\|_{H^{s-1}}^2+\|\nabla w\|_{H^{s-1}}^2+\|\nabla \phi\|_{H^{s-1}}^2+\frac{1}{1+n}\|\nabla\psi\|_{H^{s-1}}^2\right)+\frac{\nu}{2}\|\nabla^2 w\|_{H^{s-1}}^2+\frac{1}{2\tau}\|\nabla \psi\|_{H^{s-1}}^2\notag\\
				&\leq C\delta (\|\nabla n\|_{H^{s-1}}^2+\|\nabla w\|_{H^{s-1}}^2+\|\nabla \phi\|_{H^{s-1}}^2).\label{pro-est2}
			\end{align}
			Similarly, by Lemmas \ref{lem-nx} and \ref{lem-1nx} , taking $\eqref{lem-nx1}\times\frac{bc}{4\sigma^2}+\eqref{lem-1nx1}$, and we sum $\alpha$ from 1 to $s$ to obtain that
			\begin{align*}
				&\frac{bc^2}{8\sigma^2}\|\nabla n\|_{H^{s-1}}^2+\frac{3b}{4}\|\nabla\phi\|_{H^{s-1}}^2+\sum_{\alpha=1}^s\frac{\mathrm{d}}{\mathrm{d}t}\left(\frac{bc}{4\sigma^2}\int_{\mathbb{R}^3} \nabla^{\alpha-1}w\cdot\nabla^\alpha n\mathrm{d}x+\int_{\mathbb{R}^3} \nabla^{\alpha-1}\psi\cdot\nabla^\alpha \phi \mathrm{d}x\right)\\
				&\leq C(\|\nabla w\|_{H^{s}}^2+\|\psi\|_{H^s}^2)+C\delta \|\nabla n\|_{H^{s-1}}^2+(C\delta+\frac{b}{4})\|\nabla\phi\|_{H^{s-1}}^2.		
			\end{align*}
			Since $\delta$ is small enough, we have
			\begin{align}
				&\frac{bc^2}{16\sigma^2}\|\nabla n\|_{H^{s-1}}^2+\frac{b}{4}\|\nabla\phi\|_{H^{s-1}}^2+\sum_{\alpha=1}^s\frac{\mathrm{d}}{\mathrm{d}t}\left(\frac{bc}{4\sigma^2}\int_{\mathbb{R}^3} \nabla^{\alpha-1}w\cdot\nabla^\alpha n\mathrm{d}x+\int_{\mathbb{R}^3} \nabla^{\alpha-1}\psi\cdot\nabla^\alpha \phi \mathrm{d}x\right)\notag\\
				&\leq C(\|\nabla w\|_{H^{s}}^2+\|\psi\|_{H^s}^2).\label{pro-est3}
			\end{align}
			Choosing $\beta_1$ sufficiently small, then taking $(\ref{pro-est2} )+\beta_1\times(\ref{pro-est3} )$ yields
			$$
			\begin{aligned}
				&	\frac{\mathrm{d}}{\mathrm{d}t}\Big\{\frac{1}{2}\Big(\|\nabla n\|_{H^{s-1}}^2+\|\nabla w\|_{H^{s-1}}^2+\|\nabla \phi\|_{H^{s-1}}^2+\frac{1}{1+n}\| \nabla\psi\|_{H^{s}}^2\Big)\\
				&\quad+\beta_1\sum_{\alpha=1}^s\Big(\frac{bc}{4\sigma^2}\int_{\mathbb{R}^3} \nabla^{\alpha-1}w\cdot\nabla^\alpha n\mathrm{d}x+\int_{\mathbb{R}^3} \nabla^{\alpha-1}\psi\cdot\nabla^\alpha \phi \mathrm{d}x\Big)\Big\}\\
				&\quad +\frac{bc^2\beta_1}{16\sigma^2}\|\nabla n\|_{H^{s-1}}^2+\frac{\nu}{2}\|\nabla^2 w\|_{H^{s-1}}^2+\frac{b\beta_1}{4}\|\nabla\phi\|_{H^{s-1}}^2+\frac{1}{2\tau}\|\nabla  \psi\|_{H^{s-1}}^2\\
				&\leq C\delta (\|\nabla n\|_{H^{s-1}}^2+\|\nabla w\|_{H^{s-1}}^2+\|\nabla \phi\|_{H^{s-1}}^2)+C\beta_1(\|\nabla w\|_{H^{s}}^2+\|\psi\|_{H^s}^2).
			\end{aligned}
			$$
			Since $\beta_1$ and $\delta$ are sufficiently small,  one has
			$$
			\begin{aligned}
				&	\frac{\mathrm{d}}{\mathrm{d}t}\Big\{\frac{1}{2}\Big(\|\nabla n\|_{H^{s-1}}^2+\|\nabla w\|_{H^{s-1}}^2+\|\nabla \phi\|_{H^{s-1}}^2+\frac{1}{1+n}\| \nabla\psi\|_{H^{s}}^2\Big)\\
				&\quad+\beta_1\sum_{\alpha=1}^s\Big(\frac{bc}{4\sigma^2}\int_{\mathbb{R}^3} \nabla^{\alpha-1}w\cdot\nabla^\alpha n\mathrm{d}x+\int_{\mathbb{R}^3} \nabla^{\alpha-1}\psi\cdot\nabla^\alpha \phi \mathrm{d}x\Big)\Big\}\\
				&\quad+\frac{bc^2\beta_1}{32\sigma^2}\|\nabla n\|_{H^{s-1}}^2+\frac{\nu}{4}\|\nabla^2 w\|_{H^{s-1}}^2+\frac{b\beta_1}{8}\|\nabla\phi\|_{H^{s-1}}^2+\frac{1}{4\tau}\|\nabla  \psi\|_{H^{s-1}}^2\\
				&\leq C\beta_1(\|\nabla w\|_{L^2}^2+\|\psi\|_{L^2}^2).
			\end{aligned}
			$$
			We integrate the above inequality over $(0,t)$ to obtain
			\begin{align*}
				&\frac{1}{2}\Big(\|\nabla(n,w,\phi)(t)\|_{H^{s-1}}^2+\frac{1}{1+n}\|\nabla\psi(t)\|_{H^{s-1}}^2\Big)\\
				&\quad +\beta_1\sum_{\alpha=1}^s\Big(\frac{bc}{4\sigma^2}\int_{\mathbb{R}^3} \nabla^{\alpha-1}w\cdot\nabla^\alpha n\mathrm{d}x+\int_{\mathbb{R}^3} \nabla^{\alpha-1}\psi\cdot\nabla^\alpha \phi \mathrm{d}x\Big)\\
				&\quad+C_1\int_0^t\Big(\|\nabla n(r)\|_{H^{s-1}}^2+\|\nabla^2w(r)\|_{H^{s-1}}^2+\|\nabla \phi(r)\|_{H^{s-1}}^2+\|\nabla\psi(r)\|_{H^{s-1}}^2\Big)\mathrm{d}r\\
				& \leq C\int_0^t(\|\nabla w(r)\|_{L^2}^2+\|\psi(r)\|_{L^2}^2)\mathrm{d}r,
			\end{align*}
			where $C_1$ is a positive constant. Then, thanks to the smallness of $\beta_1$ and Lemma \ref{lem-b}, we arrive at
			\begin{equation*}
				\|(n,w,\phi,\psi)(t)\|_{H^s}^2+\int_0^t\left(\|\nabla (n,\phi)(r)\|_{H^{s-1}}^2+\|\nabla w(r)\|_{H^{s}}^2+\|\psi(r)\|_{H^{s}}^2\right)\mathrm{d}r\leq C\delta_0^2.
			\end{equation*}	
			Thus, we finish the proof of Proposition \ref{pro-est}.
		\end{proof}

		\section{Large-time behavior of the classical solution}\label{Sec4}
		\subsection{ Time-decay rates  for the linearized  problem }\label{Sec4-1}
		\hspace{2em}In this subsection, we consider the linearized case of the problem \eqref{nsac}-\eqref{non-f} as follows
		\begin{equation}\label{linear-nsac}
			\left\{
			\begin{aligned}
				&\partial_tn +c\,{\rm{div}}w=0, \\
				&\partial_tw+c\nabla n+\sigma \nabla \phi=\nu\Delta w+(\nu+\eta)\nabla{\rm{div}}w,\\
				&\partial_t\phi+\sigma {\rm{div}} w+b\text{div}\psi=0,\\
				&\partial_t\psi+\frac{1}{\tau}\psi+b\nabla\phi=0,
			\end{aligned}
			\right.
		\end{equation}
		with the initial data
		\begin{equation}\label{linear-in}
			(n,w,\phi,\psi)(x,0)=(n_{0},w_{0},\phi_0,\psi_{0})(x),
		\end{equation}
		and the far-field states
		\begin{equation}\label{TT1}
			\lim_{|x|\rightarrow \infty}(n,w,\phi,\psi)(x,t)=(0,0,0,0).
		\end{equation}
		Let $U=(n,w,\phi,\psi)^\top, U_0=(n_0,w_0,\phi_0,\psi_0)^\top$. The problem \eqref{linear-nsac}-\eqref{TT1} can be rewritten as
		\begin{equation}\label{I1}
			\partial_tU+\mathcal{L}U=0,\quad U(x,0)=U_0(x),\quad t\geq 0,
		\end{equation}
		where the linear operator $\mathcal{L}$ is defined as
		$$
		\mathcal{L}=\left(\begin{array}{cccc}
			0 & c\,\rm{div} & 0  &0\\
			c\nabla & -\nu\Delta-(\nu+\eta)\nabla{\rm{div}}& \sigma \nabla &0 \\
			0 & \sigma \rm{div} & 0 &b\text{div}\\
			0& 0& b\nabla & \frac{1}{\tau}\mathbb{I}_3
		\end{array}\right).
		$$
		Let $G(x,t)=(G_{ij})_{4\times 4}$ be the Green function of \eqref{I1}. After taking Fourier transform in $x$, we have
		$$\partial_t\hat{G}(\xi,t)+B\hat{G}(\xi,t)=0,\quad \hat{G}(\xi,0)=\mathbb{I}_8,
		$$
		where $B=\mathscr{F}[\mathcal{L}](\xi)$ satisifes
		$$
		B=\left(\begin{array}{cccc}
			0     & ic\xi^{t} & 0 & 0  \\
			ic\xi & \nu|\xi|^2\mathbb{I}_3+(\nu+\eta)\xi\xi^t   & i\sigma\xi & 0 \\
			0 &  i\sigma\xi^{t} & 0 & ib\xi^t\\
			0&0&ib\xi &\frac{1}{\tau} \mathbb{I}_3
		\end{array}\right).
		$$
		The eigenvalues of $B$ consist of the double eigenvalues $\lambda_1=-\nu |\xi|^2$ and $\lambda_2=-\frac{1}{\tau},$ and the other four eigenvalues recorded as $\lambda_k, (k=3,...,6)$ satisfying the following algebraic equation
		\begin{align*}
			&\tau\lambda^4+(1+\tau(2\nu+\eta)|\xi|^2)\lambda^3+(\tau(c^2+b^2+\sigma^2)+2\nu+\eta)|\xi|^2\lambda^2\\
			&+(c^2+\sigma^2+\tau b^2(2\nu+\eta)|\xi|^2)|\xi|^2\lambda+\tau c^2b^2|\xi|^4=0.	
		\end{align*}
		After a tedious but straightforward calculation, we obtain the explicit formula of the Fourier transform of the Green function
		$$
		\hat{G}(\xi,t)=\left(\begin{array}{cccc}
			\hat{G}_{11} (\xi,t)   &  \hat{G}_{12}(\xi,t) & \hat{G}_{13}(\xi,t) & \hat{G}_{14}(\xi,t)  \\
			\hat{G}_{21}(\xi,t)   &  \hat{G}_{22}(\xi,t) & \hat{G}_{23}(\xi,t) & \hat{G}_{24}(\xi,t)  \\
			\hat{G}_{31}(\xi,t)   &  \hat{G}_{32}(\xi,t) & \hat{G}_{33}(\xi,t) & \hat{G}_{34}(\xi,t)  \\
			\hat{G}_{41}(\xi,t)   &  \hat{G}_{42}(\xi,t) & \hat{G}_{43}(\xi,t) & \hat{G}_{44}(\xi,t)  \\
		\end{array}\right),
		$$
		where $\hat{G}_{ij}(\xi,t)$ are written explicitly as
		\begin{align*}
			&\hat{G}_{11}(\xi,t)=\sum_{k=3}^6\frac{-c^2|\xi|^2g_ke^{\lambda_kt}}{\prod_{j\neq k}\lambda_k(\lambda_k-\lambda_j)},\\
			&\hat{G}_{12}(\xi,t)=\hat{G}_{21}(\xi,t)=\sum_{k=3}^6\frac{-cg_ke^{\lambda_kt} i\xi}{\prod_{j\neq k}(\lambda_k-\lambda_j)}, \\
			&  \hat{G}_{13}(\xi,t)=\hat{G}_{31}(\xi,t)=\sum_{k=3}^6\frac{-c\sigma |\xi|^2e^{\lambda_kt}(\lambda_k+\frac{1}{\tau})}{\prod_{j\neq k}(\lambda_k-\lambda_j)},\\ &\hat{G}_{14}(\xi,t)=\hat{G}_{41}(\xi,t)=\sum_{k=3}^6\frac{c\sigma b|\xi|^2e^{\lambda_kt} i\xi}{\prod_{j\neq k}(\lambda_k-\lambda_j)},\\
			& \hat{G}_{22}(\xi,t)=\sum_{k=3}^6\frac{g_k\lambda_ke^{\lambda_kt}\xi\xi^t}{\prod_{j\neq k}|\xi|^2(\lambda_k-\lambda_j)}+\Big(\mathbb{I}_3-\frac{\xi\xi^t}{|\xi|^2}\Big)e^{\lambda_1t},\\ &\hat{G}_{23}(\xi,t)=\hat{G}_{32}(\xi,t)= \sum_{k=3}^6\frac{-\sigma\lambda_ke^{\lambda_kt}(\lambda_k+\frac{1}{\tau})i\xi}{\prod_{j\neq k}(\lambda_k-\lambda_j)},\\
			& \hat{G}_{24}(\xi,t)= \hat{G}_{42}(\xi,t)=\sum_{k=3}^6\frac{-\sigma b\lambda_ke^{\lambda_kt} \xi\xi^t }{\prod_{j\neq k}(\lambda_k-\lambda_j)},\\
			&\hat{G}_{33}(\xi,t)=\sum_{k=3}^6\frac{e^{\lambda_kt}(\lambda_k+\frac{1}{\tau})h_k }{\prod_{j\neq k}(\lambda_k-\lambda_j)},\\
			&\hat{G}_{34}(\xi,t)=\hat{G}_{43}(\xi,t)=\sum_{k=3}^6\frac{-bh_ke^{\lambda_kt}i\xi}{\prod_{j\neq k}(\lambda_k-\lambda_j)},\\ &\hat{G}_{44}(\xi,t)=\sum_{k=3}^6\frac{\lambda_k(h_k+\sigma^2|\xi|^2)e^{\lambda_kt}\xi\xi^t}{\prod_{j\neq k}|\xi|^2(\lambda_k-\lambda_j)}+\Big(\mathbb{I}_3-\frac{\xi\xi^t}{|\xi|^2}\Big)e^{\lambda_2t},
		\end{align*}
		and the functions $g_k$ and $h_k$ satisfy
		$$
		g_k=\lambda_k^2+\frac{1}{\tau}\lambda_k+b^2|\xi|^2,\quad h_k=\lambda_k^2+(2\nu+\eta)|\xi|^2\lambda_k+c^2|\xi|^2, \quad k=3,...,6.
		$$
		The asymptotic behaviors of the eigenvalues are established in the following lemma, which can be proved by applying the implicit function theorem. Here we omit the proof for brevity.
		\begin{lem}\label{lem-lxi}
			Assume that the positive constant $r_0$ is sufficiently small and $R_0$ is sufficiently large.
			\begin{itemize}
				\item For low-frequency part $|\xi|\leq r_0$, the eigenvalues $\lambda_{k}\,(k=3,...,6)$ satisfy
				\begin{equation}\label{lem-lxi1}
					\left\{
					\begin{aligned}
						&\lambda_3=-\frac{1}{\tau}+\tau b^2|\xi|^2+O(|\xi|^4),\\
						&\lambda_4=-\frac{\tau b^2\sigma^2}{2(\sigma^2+c^2)}|\xi|^2-\frac{1}{2}(2\nu+\eta)|\xi|^2+i\sqrt{\sigma^2+c^2}|\xi|+O(|\xi|^3),\\
						&\lambda_5=-\frac{\tau b^2\sigma^2}{2(\sigma^2+c^2)}|\xi|^2-\frac{1}{2}(2\nu+\eta)|\xi|^2-i\sqrt{\sigma^2+c^2}|\xi|+O(|\xi|^3),\\
						&\lambda_6=-\frac{\tau b^2 c^2}{\sigma^2+c^2}|\xi|^2+O(|\xi|^4).
					\end{aligned}
					\right.
				\end{equation}
				In addition, there exists a positive constant $\beta$ such that
				\begin{equation}\label{lem-lxi2}
					{\rm{Re}} \lambda_k\leq -\beta|\xi|^2, \quad k=3,...,6.
				\end{equation}
				\item For high-frequency part $|\xi|\geq R_0$, the eigenvalues satisfy
				\begin{equation}\label{lem-hxi1}
					\left\{
					\begin{aligned}
						&\lambda_3=-\frac{c^2}{2\nu+\eta}+O(|\xi|^{-2}),\\
						&\lambda_4=-\frac{\sigma^2}{2(2\nu+\eta)}-\frac{1}{2\tau}+ib|\xi|+O(|\xi|^{-1}),\\
						&\lambda_5=-\frac{\sigma^2}{2(2\nu+\eta)}-\frac{1}{2\tau}-ib|\xi|+O(|\xi|^{-1}),\\
						&\lambda_6=-(2\nu+\eta)|\xi|^2+\frac{c^2+\sigma^2}{2\nu+\eta}+O(|\xi|^{-1}).
					\end{aligned}
					\right.
				\end{equation}
				In addition, there exists a positive constant $R_1$ such that
				\begin{equation}\label{lem-hxi2}
					{\rm{Re}}\lambda_k\leq -R_1, \quad k=3,...,6.
				\end{equation}
				\item For medium-frequency part $r_0\leq|\xi|\leq R_0$, the eigenvalues satisfy
				\begin{equation}\label{lem-mxi1}
					{\rm{Re}}\lambda_k\leq -R_2,\quad  k=3,...,6,
				\end{equation}
				where $R_2$ is a positive constant.
			\end{itemize}
		\end{lem}
		Now we are in a position to establish time-decay estimates of the solution $(n,w,\phi,\psi)$ to the problem \eqref{linear-nsac}-\eqref{TT1} for some special initial data.
		\begin{prop}\label{pro-li-d}
			Assume that the conditions in Theorem \ref{thm3} hold. Then, the global solution $(n,w,\phi,\psi)$ of the linearized problem \eqref{linear-nsac}-\eqref{TT1} satisfies, for large time $t$, that
			\begin{equation}\label{pro-li-d1}
				\begin{aligned}
					&\bar{c}_1(1+t)^{-\frac{3}{4}}\leq \|(n,w,\phi)(t)\|_{L^2}\leq  C(1+t)^{-\frac{3}{4}},
					&\bar{c}_1(1+t)^{-\frac{5}{4}}\leq \|\psi(t)\|_{L^2}\leq  C(1+t)^{-\frac{5}{4}},
				\end{aligned}
			\end{equation}
			where $\bar{c}_1$ is a positive constant independent of time.
		\end{prop}
		\begin{proof}
			The solution $\hat{U}(\xi,t)=(\hat{n},\hat{w},\hat{\phi},\hat{\psi})^\top$ for the linear problem \eqref{linear-nsac}-\eqref{TT1} can be solved as
			$$
			\left\{\begin{aligned}
				&\hat{n}(\xi,t)=\hat{G}_{11}\hat{n}_0+\hat{G}_{12}\hat{w}_0+\hat{G}_{13}\hat{\phi}_0+\hat{G}_{14}\hat{\psi}_0,\\
				&\hat{w}(\xi,t)=\hat{G}_{21}\hat{n}_0+\hat{G}_{22}\hat{w}_0+\hat{G}_{23}\hat{\phi}_0+\hat{G}_{24}\hat{\psi}_0,\\
				&\hat{\phi}(\xi,t)=\hat{G}_{31}\hat{n}_0+\hat{G}_{32}\hat{w}_0+\hat{G}_{33}\hat{\phi}_0+\hat{G}_{34}\hat{\psi}_0,\\
				&\hat{\psi}(\xi,t)=\hat{G}_{41}\hat{n}_0+\hat{G}_{42}\hat{w}_0+\hat{G}_{43}\hat{\phi}_0+\hat{G}_{44}\hat{\psi}_0.
			\end{aligned}	
			\right.
			$$
			For the low-frequency part $|\xi|\leq r_0$, by a tedious calculation, we can obtain the leading term of $\hat{G}_{ij}$ as
			\begin{align*}
				&\hat{G}_{11}(\xi,t)=-c^2b^4\tau^6|\xi|^6e^{\lambda_3t}+		\frac{c^2}{2(c^2+\sigma^2)}e^{\lambda_4t}+	\frac{c^2}{2(c^2+\sigma^2)}e^{\lambda_5t}+\frac{\sigma^2}{c^2+\sigma^2}e^{\lambda_6t},\\
				&\hat{G}_{12}(\xi,t)=\hat{G}_{21}(\xi,t)=cb^4\tau^5|\xi|^4 i\xi e^{\lambda_3t}-\frac{c\xi}{2\sqrt{c^2+\sigma^2}|\xi|}e^{\lambda_4t}
				+\frac{c\xi}{2\sqrt{c^2+\sigma^2}|\xi|}e^{\lambda_5t}-\frac{cb^2\sigma^2\tau}{(c^2+\sigma^2)^2}e^{\lambda_6t}i\xi,\\
				&\hat{G}_{13}(\xi,t)=\hat{G}_{31}(\xi,t)=cb^2\sigma \tau^4|\xi|^4e^{\lambda_3t}+\frac{c\sigma}{2(c^2+\sigma^2)}e^{\lambda_4t}
				+\frac{c\sigma}{2(c^2+\sigma^2)}e^{\lambda_5t}-\frac{c\sigma}{c^2+\sigma^2}e^{\lambda_6t},\\
				&\hat{G}_{14}(\xi,t)=\hat{G}_{41}(\xi,t)=-cb\sigma\tau^3|\xi|^2i\xi e^{\lambda_3t}-\frac{cb\sigma\tau}{2(c^2+\sigma^2)}e^{\lambda_4t}i\xi
				+\frac{cb\sigma\tau}{2(c^2+\sigma^2)}e^{\lambda_5t}i\xi+\frac{cb\sigma\tau}{c^2+\sigma^2}e^{\lambda_6t}i\xi,\\
				&\hat{G}_{22}(\xi,t)=\Big( \mathbb{I}_3-\frac{\xi\xi^t}{|\xi|^2}\Big)e^{\lambda_1t}+b^4\tau^4|\xi|^2\xi\xi^te^{\lambda_3t}
				+\frac{1}{2}\frac{\xi\xi^t}{|\xi|^2}e^{\lambda_4t}
				+\frac{1}{2}\frac{\xi\xi^t}{|\xi|^2}e^{\lambda_5t}-\frac{c^2b^4\sigma^2\tau^2}{(c^2+\sigma^2)^3}\xi\xi^te^{\lambda_6t},\\
				&\hat{G}_{23}(\xi,t)=\hat{G}_{32}(\xi,t)=-b^2\sigma\tau^3|\xi|^2i\xi e^{\lambda_3t}-\frac{\sigma \xi}{2\sqrt{c^2+\sigma^2}|\xi|}e^{\lambda_4t}+\frac{\sigma \xi}{2\sqrt{c^2+\sigma^2}|\xi|}e^{\lambda_5t}+\frac{c^2b^2\sigma\tau}{(c^2+\sigma^2)^2}i\xi e^{\lambda_6t},\\
				&\hat{G}_{24}(\xi,t)=\hat{G}_{42}(\xi,t)=-b\sigma\tau^2\xi\xi^te^{\lambda_3t}
				+\frac{ib\sigma\tau}{2\sqrt{c^2+\sigma^2}|\xi|}\xi\xi^te^{\lambda_4t}-\frac{ib\sigma\tau}{2\sqrt{c^2+\sigma^2}|\xi|}\xi\xi^t e^{\lambda_5t}
				+\frac{c^2b^3\sigma\tau^2}{(c^2+\sigma^2)^2}\xi\xi^te^{\lambda_6t},\\
				&\hat{G}_{33}(\xi,t)=-b^2\tau^2|\xi|^2e^{\lambda_3t}+\frac{\sigma^2}{2(c^2+\sigma^2)}e^{\lambda_4t}+\frac{\sigma^2}{2(c^2+\sigma^2)}e^{\lambda_5t}
				+\frac{c^2}{c^2+\sigma^2}e^{\lambda_6t},\\
				&\hat{G}_{34}(\xi,t)=\hat{G}_{43}(\xi,t)= b\tau i\xi e^{\lambda_3t}-\frac{b\sigma^2\tau}{2(c^2+\sigma^2)}i\xi e^{\lambda_4t}
				-\frac{b\sigma^2\tau}{2(c^2+\sigma^2)}i\xi e^{\lambda_5t}-\frac{c^2b\tau}{c^2+\sigma^2}i\xi e^{\lambda_6t},\\
				&\hat{G}_{44}(\xi,t)=\Big( \mathbb{I}_3-\frac{\xi\xi^t}{|\xi|^2}\Big)e^{\lambda_2t}+\frac{\xi\xi^t}{|\xi|^2}e^{\lambda_3t}-\frac{b^2\sigma^2\tau^2}{2(c^2+\sigma^2)}\xi\xi^t e^{\lambda_4t}-\frac{b^2\sigma^2\tau^2}{2(c^2+\sigma^2)}\xi\xi^t e^{\lambda_5t}-\frac{c^2b^2\tau^2}{c^2+\sigma^2}\xi\xi^t e^{\lambda_6 t}.
			\end{align*}
			
			For $|\xi|\geq r_0$, we follow a similar method developed in \cite{MR2164944} to prove the boundedness of $\hat{G}_{ij}(\xi,t)$, which, together with Lemma \ref{lem-lxi}, yields
			$$
			\hat{G}_{ij}(\xi,t)=\mathcal{O}(1)e^{-Rt}, ~~1\leq i,j\leq 4,
			$$
			where $R=\max\{R_1, R_2\}$.
			
			Therefore, due to Parseval's equality and Lemma \ref{lem-lxi}, it is enough to estimate the decay rate of solution $n(x,t)$ as
			$$
			\begin{aligned}
				\|n(t)\|_{L^2}
				&=\|\hat{n}(t)\|_{L^2}\leq \|\hat{G}_{11}\hat{n}_0+\hat{G}_{12}\hat{w}_0+\hat{G}_{13}\hat{\phi}_0+\hat{G}_{14}\hat{\psi}_0\|_{L^2}\\
				&\leq \big\| -c^2b^4\tau^6|\xi|^6e^{\lambda_3t}+		\frac{c^2}{2(c^2+\sigma^2)}e^{\lambda_4t}+	\frac{c^2}{2(c^2+\sigma^2)}e^{\lambda_5t}+\frac{\sigma^2}{c^2+\sigma^2}e^{\lambda_6t}\big\|_{L^2}\|n_0\|_{L^1}\\
				&~~+\big\| cb^4\tau^5|\xi|^4 i\xi e^{\lambda_3t}-\frac{c\xi}{2\sqrt{c^2+\sigma^2}|\xi|}e^{\lambda_4t}
				+\frac{c\xi}{2\sqrt{c^2+\sigma^2}|\xi|}e^{\lambda_5t}-\frac{cb^2\sigma^2\tau}{(c^2+\sigma^2)^2}e^{\lambda_6t}i\xi\big\|_{L^2}\|w_0\|_{L^1}\\
				&~~+\big\| cb^2\sigma \tau^4|\xi|^4e^{\lambda_3t}+\frac{c\sigma}{2(c^2+\sigma^2)}e^{\lambda_4t}
				+\frac{c\sigma}{2(c^2+\sigma^2)}e^{\lambda_5t}-\frac{c\sigma}{c^2+\sigma^2}e^{\lambda_6t}\big\|_{L^2}\big\|\phi_0\|_{L^1}\\
				&~~+\big\|-cb\sigma\tau^3|\xi|^2i\xi e^{\lambda_3t}-\frac{cb\sigma\tau}{2(c^2+\sigma^2)}e^{\lambda_4t}i\xi
				+\frac{cb\sigma\tau}{2(c^2+\sigma^2)}e^{\lambda_5t}i\xi+\frac{cb\sigma\tau}{c^2+\sigma^2}e^{\lambda_6t}i\xi\big\|_{L^2}\|\psi_0\|_{L^1}\\
				&~~+Ce^{-Rt}\|(n_0,w_0,\phi_0,\psi_0)\|_{L^2},\\
				&\leq C(1+t)^{-\frac{3}{4}} (\|U_0\|_{L^1}+\|U_0\|_{L^2})\\
				&\leq C(1+t)^{-\frac{3}{4}}.
			\end{aligned}
			$$
			The same process can be repeated for $w(x,t)$ and $\phi(x,t)$. Thus, we have
			$$
			\|(w,\phi)(t)\|_{L^2}\leq C(1+t)^{-\frac{3}{4}}.	
			$$
			For $\psi(x,t)$, it holds
			\begin{align*}
				\|\psi(t)\|_{L^2}
				&=\|\hat{\psi}(t)\|_{L^2}\leq\|\hat{G}_{41}\hat{n}_0+\hat{G}_{42}\hat{w}_0+\hat{G}_{43}\hat{\phi}_0+\hat{G}_{44}\hat{\psi}_0\|_{L^2}\\	
				& \leq \| -cb\sigma\tau^3|\xi|^2i\xi e^{\lambda_3t}-\frac{cb\sigma\tau}{2(c^2+\sigma^2)}e^{\lambda_4t}i\xi
				+\frac{cb\sigma\tau}{2(c^2+\sigma^2)}e^{\lambda_5t}i\xi+\frac{cb\sigma\tau}{c^2+\sigma^2}e^{\lambda_6t}i\xi\|_{L^2}\|n_0\|_{L^1}\\
				&~~+\|-b\sigma\tau^2\xi\xi^te^{\lambda_3t}
				+\frac{ib\sigma\tau}{2\sqrt{c^2+\sigma^2}|\xi|}\xi\xi^te^{\lambda_4t}-\frac{ib\sigma\tau}{2\sqrt{c^2+\sigma^2}|\xi|}\xi\xi^t e^{\lambda_5t}
				+\frac{c^2b^3\sigma\tau^2}{(c^2+\sigma^2)^2}\xi\xi^te^{\lambda_6t}\|_{L^2}\|w_0\|_{L^1}	\\
				&~~+\|b\tau i\xi e^{\lambda_3t}-\frac{b\sigma^2\tau}{2(c^2+\sigma^2)}i\xi e^{\lambda_4t}
				-\frac{b\sigma^2\tau}{2(c^2+\sigma^2)}i\xi e^{\lambda_5t}-\frac{c^2b\tau}{c^2+\sigma^2}i\xi e^{\lambda_6t}\|_{L^2}\|\phi_0\|_{L^1}\\
				&~~+\|\Big( \mathbb{I}_3-\frac{\xi\xi^t}{|\xi|^2}\Big)e^{\lambda_2t}+\frac{\xi\xi^t}{|\xi|^2}e^{\lambda_3t}-\frac{b^2\sigma^2\tau^2}{2(c^2+\sigma^2)}\xi\xi^t e^{\lambda_4t}-\frac{b^2\sigma^2\tau^2}{2(c^2+\sigma^2)}\xi\xi^t e^{\lambda_5t}-\frac{c^2b^2\tau^2}{c^2+\sigma^2}\xi\xi^t e^{\lambda_6 t}\|_{L^2}\|\psi_0\|_{L^1}\\
				&~~+Ce^{-Rt} \|(n_0,w_0,\phi_0,\psi_0)\|_{L^2}\\
				&\leq C(1+t)^{-\frac{5}{4}}(\|U_0\|_{L^1}+\|U_0\|_{L^2})\\
				&\leq C(1+t)^{-\frac{5}{4}}.
			\end{align*}
			Now we are going to establish the lower bound of the time-decay rates. For $|\xi|\leq r_0$, since $\hat{w}_0(\xi)=\hat{\phi}_0(\xi)=0$, we write $\hat{n}(\xi,t)$ as follows
			\begin{align}
				\hat{n}(\xi,t)
				&=\left(-c^2b^4\tau^6|\xi|^6e^{\lambda_3t}+\frac{c^2}{2(c^2+\sigma^2)}(e^{\lambda_4t}+e^{\lambda_5t})+\frac{\sigma^2}{c^2+\sigma^2}e^{\lambda_6t}\right)\hat{n}_0(\xi)\notag\\
				&\quad+\left(-cb\sigma\tau^3|\xi|^2e^{\lambda_3t}+\frac{cb\sigma\tau}{2(c^2+\sigma^2)} (e^{\lambda_5t}-e^{\lambda_4t})+\frac{cb\sigma\tau}{c^2+\sigma^2}e^{\lambda_6t} \right)i\xi\cdot\hat{\psi}_0 (\xi)\notag\\
				&\triangleq K_1+K_2.\label{pro-li-d33}
			\end{align}
			To deal with the term $K_1$, we rewrite it as
			\begin{equation}\label{2-03-2}
				K_1= -c^2b^4\tau^6|\xi|^6e^{\lambda_3t}\hat{n}_0(\xi)+\left(\frac{c^2}{2(c^2+\sigma^2)}(e^{\lambda_4t}+e^{\lambda_5t})+\frac{\sigma^2}{c^2+\sigma^2}e^{\lambda_6t}\right)\hat{n}_0(\xi)
				\triangleq K_{11}+K_{12}.
			\end{equation}
			First, a direct calculation yields
			\begin{equation}\label{2-03-1}
				\int_{|\xi|\leq r_0}|K_{11}|^2\mathrm{d}\xi \leq \int_{|\xi|\leq r_0}\big|c^2b^4\tau^6|\xi|^6e^{\lambda_3t}\hat{n}_0(\xi)\big|^2\mathrm{d}\xi \leq  C \|n_0\|_{L^1}^2\int_{\mathbb{R}^3}\big||\xi|^6e^{\lambda_3t}\big|^2\mathrm{d}\xi\leq Ce^{-\frac{1}{2\tau}t}.
			\end{equation}
			Let $\nu_1=\frac{\tau b^2\sigma^2}{2(\sigma^2+c^2)}+\frac{1}{2}(2\nu+\eta)$, $\hat{c}=\sqrt{\sigma^2+c^2}$, $\nu_2=\frac{\tau b^2c^2}{\sigma^2+c^2}$. Then, it holds
			\begin{align}
				\int_{|\xi|\leq r_0}|K_{12}|^2\mathrm{d}\xi&\geq C\inf_{|\xi|\leq r_0}|\hat{n}_0(\xi)|^2\int_{|\xi|\leq r_0}(e^{-\nu_1|\xi|^2t}\cos (\hat{c}|\xi|t)+e^{-\nu_2|\xi|^2t})^2\mathrm{d}\xi\notag\\
				&\quad-C\int_{|\xi|\leq r_0}e^{-2\upsilon_1|\xi|^2t}\mathcal{O}(|\xi|^6t^2)\mathrm{d}\xi
				-C\int_{|\xi|\leq r_0}e^{-2\nu_2|\xi|^2t}\mathcal{O}(|\xi|^8t^2)\mathrm{d}\xi\notag\\
				&\geq  C\inf_{|\xi|\leq r_0}|\hat{n}_0(\xi)|^2\int_{|\xi|\leq r_0}e^{-(\nu_1+\nu_2)|\xi|^2t}(1+\cos (\hat{c}|\xi|t))^2\mathrm{d}\xi\notag\\
				&\quad-C\int_{|\xi|\leq r_0}(e^{-2\nu_1|\xi|^2t}\mathcal{O}(|\xi|^6t^2)+e^{-2\nu_2|\xi|^2t}\mathcal{O}(|\xi|^8t^2))\mathrm{d}\xi\notag\\
				&\geq C\mu_0^2\int_{|\xi|\leq r_0}e^{-(\nu_1+\nu_2)|\xi|^2t}(1+\cos(\hat{c}|\xi|t))^2\mathrm{d}\xi\notag\\
				&\quad-C\int_{|\xi|\leq r_0}(e^{-2\nu_1|\xi|^2t}\mathcal{O}(|\xi|^6t^2)+e^{-2\nu_2|\xi|^2t}\mathcal{O}(|\xi|^8t^2))\mathrm{d}\xi\notag\\
				&\triangleq Z_1+Z_2.\label{pro-li-d3}
			\end{align}
			Let $m=\sqrt{t}|\xi|$. For large time $t\geq t_0$, we obtain
			\begin{align}
				Z_1&=C\mu_0^2t^{-\frac{3}{2}}\int_0^{r_0\sqrt{t}}
				e^{-(\nu_1+\nu_2)m^2}(1+\text{cos}(\hat{c}m\sqrt{t}))^2m^2\mathrm{d}m\notag\\
				&\geq C\mu_0^2(1+t)^{-\frac{3}{2}}\sum_{k=0}^{[\hat{c}r_0t/\pi]-1}
				\int_{2k\pi/(\hat{c}\sqrt{t})}^{(2k\pi+\frac{\pi}{3})/(\hat{c}\sqrt{t})}e^{-(\nu_1+\nu_2)m^2}
				(1+\text{cos}(\hat{c}m\sqrt{t}))^2m^2\mathrm{d}m\notag\\
				&\geq\frac{9}{4}C\mu_0^2(1+t)^{-\frac{3}{2}}\sum_{k=0}^{[\hat{c}r_0t/\pi]-1}
				\int_{2k\pi/(\hat{c}\sqrt{t})}^{(2k\pi+\frac{\pi}{3})/(\hat{c}\sqrt{t})}e^{-(\nu_1+\nu_2)m^2}m^2\mathrm{d}m\notag\\
				&\geq \tilde{g}\mu_0^2(1+t)^{-\frac{3}{2}},\label{pro-li-d4}
			\end{align}
			where $\tilde{g}$ is a positive constant. For the term $Z_2$, it is easy to verify that
			\begin{equation}\label{pro-li-d5}
				|Z_2|\leq C(1+t)^{-\frac{5}{2}}.
			\end{equation}
			Substituting \eqref{pro-li-d4} and \eqref{pro-li-d5} into \eqref{pro-li-d3} yields
			$$
			\int_{|\xi|\leq r_0}|K_{12}|^2\mathrm{d}\xi\geq \tilde{g}\mu_0^2(1+t)^{-\frac{3}{2}}-C(1+t)^{-\frac{5}{2}},
			$$
			which, together with \eqref{2-03-1} and \eqref{2-03-2}, yields
			\begin{equation}\label{pro-li-d6}
				\int_{|\xi|\leq r_0} |K_1|^2\mathrm{d}\xi\geq \tilde{g}\mu_0^2(1+t)^{-\frac{3}{2}}-C(1+t)^{-\frac{5}{2}}-Ce^{-\frac{1}{2\tau}t}.
			\end{equation}
			It is easy to show
			\begin{equation}\label{pro-li-d7}
				\int_{|\xi|\leq r_0} |K_2|^2\mathrm{d}\xi\leq C(1+t)^{-\frac{5}{2}}.
			\end{equation}
			Thus, for large-time $t\geq t_0$, combining \eqref{pro-li-d6} and \eqref{pro-li-d7}, we obtain  from \eqref{pro-li-d33} that
			\begin{align}
				\|\hat{n}(t)\|_{L^2}^2&\geq \int_{|\xi|\leq r_0} |K_1+K_2|^2\mathrm{d}\xi\notag\\
				&\geq \frac{1}{2}\int_{|\xi|\leq r_0} |K_1|^2\mathrm{d}\xi-\int_{|\xi|\leq r_0} |K_2|^2\mathrm{d}\xi\notag\\
				&\geq \frac{1}{2}\tilde{g}^2\mu_0^2(1+t)^{-\frac{3}{2}}-C(1+t)^{-\frac{5}{2}}-Ce^{-\frac{1}{2\tau}t}\notag\\
				&\geq \frac{1}{4}\tilde{g}^2\mu_0^2(1+t)^{-\frac{3}{2}}\triangleq \tilde{c}_1^2(1+t)^{-\frac{3}{2}},\label{pro-li-d8}
			\end{align}
			with a positive constant $\tilde{c}_1>0$. Therefore, Parseval's equality and \eqref{pro-li-d8} yield
			$$
			\|n(t)\|_{L^2}=\|\hat{n}(t)\|_{L^2}\geq \tilde{c}_1(1+t)^{-\frac{3}{4}}.
			$$
			Similarly, we can obtain
			\begin{align*}
				\|w(t)\|_{L^2}\geq \tilde{c}_2(1+t)^{-\frac{3}{4}},\quad\|\phi(t)\|_{L^2}\geq \tilde{c}_3(1+t)^{-\frac{3}{4}},\quad\|\psi(t)\|_{L^2}\geq \tilde{c}_4(1+t)^{-\frac{5}{4}},
			\end{align*}
			where the positive constants $\tilde{c}_2, \tilde{c}_3, \tilde{c}_4$ are independent of time.
			
			Choosing $\bar{c}_1=\min\{\tilde{c}_1, \tilde{c}_2, \tilde{c}_3,\tilde{c}_4\}$, then we  complete the proof of Proposition \ref{pro-li-d}.
		\end{proof}
		\subsection{Nonlinear time-decay estimates}\label{subsec5-non-decay}
		\hspace{2em}To investigate the  upper bound time-decay rates of classical solutions to the nonlinear system \eqref{nsac}, we first introduce the energy $\mathcal{E}_k^s(t)$ as
		\begin{equation}\label{eskt}
			\begin{aligned}
				&\mathcal{E}_0^s(t)\triangleq\|(n,w,\phi,\psi)(t)\|_{H^s}^2,\\
			&\mathcal{E}_k^s(t)\triangleq\|\nabla^k(n,w,\phi,\psi)(t)\|_{H^{s-k}}^2\\
			&\quad\quad\quad+\beta_2\sum_{\alpha=k}^s\int_{\mathbb{R}^3}\Big(\nabla^{\alpha-1}w^h\cdot\nabla^\alpha n^h+\nabla^{\alpha-1}\psi^h\cdot\nabla^\alpha \phi^h\Big)\mathrm{d}x,\quad 1\leq k\leq s,
			\end{aligned}
		\end{equation}
	with $\beta_2$ a small positive constant and the time-weighted energy functional
		\begin{equation}\label{mskt}
			\mathcal{M}(t)\triangleq\sup_{0\leq r \leq t}\left\{(1+r)^{\frac{3}{4}}\|(n,w,\phi,\psi)(r)\|_{H^{s}}\right\}.
		\end{equation}
		Then, it yields
		\begin{equation}\label{THZT1}
			\|(n,w,\phi,\psi)(t)\|_{H^s}\leq (1+t)^{-\frac{3}{4}}\mathcal{M}(t).
		\end{equation}
		The next goal is to prove the uniform time-independent bound of $\mathcal{M}(t)$.
		
		Let $U=(n,w,\phi,\psi)^\top, U_0=(n_0,w_0,\phi_0,\psi_0)^\top, F=(f_1,f_2,f_3,0)^\top$. We rewrite the nonlinear system  \eqref{nsac} into the operator form
		\begin{align*}
			\left\{
			\begin{aligned}
				&\partial_tU+\mathcal{L}U=F,\\
				&U|_{t=0}=U_0.
			\end{aligned}
			\right.
		\end{align*}
		Making use of Duhamel's principle yields
		\begin{equation}\label{nonlinear-v}
			U(x,t)=G*U_0+\int_0^tG(t-r)*F(r)\mathrm{d}r.
		\end{equation}
		To analyze the large-time behavior of the solution in frequency space, we introduce the low-high frequency decomposition. Decompose the solution into two parts:
		\begin{equation}\label{l-h-com}
			U(x,t)=U^{\ell}(x,t)+U^{h}(x,t)=(n^{\ell},w^{\ell},\phi^{\ell},\psi^\ell)^\top+(n^{h},w^{h},\phi^{h},\psi^h)^\top.
		\end{equation}
		Here $U^{\ell}(x,t)=P_1U(x,t)$ and $U^{h}(x,t)=P_{\infty}U(x,t)$ represent the low-frequency part and high-frequency part, respectively. The projection operator $P_1$ and $P_{\infty}$ are introduced in Section \ref{Sec2}. It is easy to verify that for $0\leq k<m$,
		\begin{equation}\label{2-07-01}
			\|\nabla^{m} U^\ell\|_{L^2}\leq C\|\nabla^{k} U^\ell\|_{L^2},\quad \|\nabla^{k}U^h\|_{L^2}\leq C\|\nabla^{m}U^h\|_{L^2},
		\end{equation}
		\begin{equation}\label{2-07-02}
			\|\nabla^{m} U^\ell\|_{L^2}\leq C\|\nabla^{m} U\|_{L^2},\quad\|\nabla^{m}U^h\|_{L^2}\leq C\|\nabla^{m}U\|_{L^2}.
		\end{equation}
		
		Now, we are in a position to establish the decay estimates for $\|(n,w,\phi,\psi)(t)\|_{H^s}$.
		
		\begin{lem}\label{lem-hs}
			Assume that the conditions in Theorem \ref{thm1} hold. Let $(n,w,\phi,\psi)$ be the classical solution of the problem \eqref{nsac}-\eqref{non-f}, then it holds
			\begin{equation}\label{lem-hs0}
				\|(n,w,\phi,\psi)(t)\|_{H^s}\leq C(1+t)^{-\frac{3}{4}},
			\end{equation}
			where the positive constant $C$ is independent of time.
		\end{lem}
		\begin{proof}
			According to \eqref{pro-est1}, we can obtain
			\begin{align*}
				\frac{\mathrm{d}}{\mathrm{d}t}\|(n,w,\phi,\psi)(t)\|_{H^s}^2+\|\nabla n (t)\|_{H^{s-1}}^2+\|\nabla w(t)\|_{H^s}^2+\|\nabla \phi(t)\|_{H^{s-1}}^2+\|\psi(t)\|_{H^{s}}^2\leq 0.
			\end{align*}
			By the definition of $\mathcal{E}_0^s(t)$, one has
			\begin{equation*}
				\frac{\mathrm{d}}{\mathrm{d}t}\mathcal{E}^s_0(t)+\|\nabla n(t)\|^2_{H^{s-1}}+\|\nabla w (t)\|_{H^{s}}^2+\|\nabla \phi(t)\|_{H^{s-1}}^2+\|\psi(t)\|_{H^{s}}^2\leq 0,
			\end{equation*}
			which implies that there exists a positive constant $C_2$ such that
			\begin{equation}\label{lem-hs1}
				\frac{\mathrm{d}}{\mathrm{d}t}\mathcal{E}^s_0(t)+C_2\mathcal{E}^s_0(t)\leq C\|(n^{\ell},w^{\ell},\phi^{\ell},\psi^\ell)(t)\|^2_{L^2},
			\end{equation}
			where we have used the fact that $\|(n^{h},w^{h},\phi^{h},\psi^h)(t)\|_{L^2}\leq C \|\nabla(n,w,\phi,\psi)(t)\|_{L^2}$.
			
			In terms of Duhamel's principle, the low-frequency part $U^\ell=(n^{\ell},w^{\ell},\phi^{\ell},\psi^\ell)$ can be written as
			\begin{equation}\label{lem-hs2}
				U^\ell(x,t)=G^\ell*U_0+\int_0^tG^\ell(t-r)*F(r)\mathrm{d}r.
			\end{equation}
			By Parseval's equality and H$\ddot{\text{o}}$lder's inequality, one has
			\begin{align}
				\|n^{\ell}\|_{L^2}
				&\leq \|G^\ell_{11}*n_0+G^\ell_{12}*w_0+G^\ell_{13}*\phi_0+G^\ell_{14}*\psi_0\|_{L^2}\notag\\
				&\quad +\int_0^t  \|G^\ell_{11}(t-r)*f_1(r)+G^\ell_{12}(t-r)*f_2(r)+G^\ell_{13}(t-r)*f_3(r)\|_{L^2}\mathrm{d}r\notag\\
				&\leq C(1+t)^{-\frac{3}{4}}\|(n_0,w_0,\phi_0,\psi_0)\|_{L^1}+C\int_0^t(1+t-r)^{-\frac{3}{4}}\|F(r)\|_{L^1}\mathrm{d}r\notag\\
				&\leq CA_0(1+t)^{-\frac{3}{4}}+C\mathcal{M}(t)\int_0^t(1+t-r)^{-\frac{3}{4}}(1+r)^{-\frac{3}{4}}\notag\\
				&\quad \times(\|\nabla(n,w,\phi,\psi)(r)\|_{L^2}+\|\nabla^2w(r)\|_{L^2})\mathrm{d}r\notag\\
				&\leq CA_0(1+t)^{-\frac{3}{4}}+C\mathcal{M}(t)\Big(\int_0^t(1+t-r)^{-\frac{3}{2}}(1+r)^{-\frac{3}{2}}\mathrm{d}r\Big)^{\frac{1}{2}}\notag\\
				&\quad \times\Big(\int_0^t(\|\nabla(n,w,\phi,\psi)(r)\|_{L^2}^2+\|\nabla^2w(r)\|_{L^2}^2)\mathrm{d}r\Big)^{\frac{1}{2}}\notag\\
				&\leq C(1+t)^{-\frac{3}{4}}\big(A_0+\delta_0\mathcal{M}(t)\big),\label{lem-hs31}
			\end{align}
			where $A_0\triangleq\|(n_0,w_0,\phi_0,\psi_0)\|_{L^1}$. Similarly, we can prove
			\begin{equation}\label{lem-hs32}
				\|(w^\ell,\phi^\ell)\|_{L^2}\leq C(1+t)^{-\frac{3}{4}}\big(A_0+\delta_0\mathcal{M}(t)\big),\quad
				\|\psi^\ell\|_{L^2}\leq C(1+t)^{-\frac{5}{4}}\big(A_0+\delta_0\mathcal{M}(t)\big).
			\end{equation}	
			Substituting \eqref{lem-hs31} and \eqref{lem-hs32} into \eqref{lem-hs1} yields
			\begin{equation*}
				\frac{\mathrm{d}}{\mathrm{d}t}\mathcal{E}^s_0(t)+C_2\mathcal{E}^s_0(t)\leq C(1+t)^{-\frac{3}{2}}\left(A_0+\delta_0\mathcal{M}(t)\right)^2.
			\end{equation*}
			With the help of Gr\"{o}nwall's inequality, it is easy to show
			\begin{equation*}
				\begin{aligned}
					\mathcal{E}^s_0(t)&\leq e^{-C_2t}\mathcal{E}^s_0(0)+ C\int_0^te^{-C_2(t-r)}(1+r)^{-\frac{3}{2}}\left(A_0+\delta_0\mathcal{M}(t)\right)^2\mathrm{d}r\\
					&\leq C(1+t)^{-\frac{3}{2}}\left(\delta_0^2+A_0^2+\delta_0^2\mathcal{M}(t)^2\right),
				\end{aligned}
			\end{equation*}
			which implies that
			\begin{equation}\label{lem-hs4}
				(1+t)^{\frac{3}{2}}\|(n,w,\phi,\psi)(t)\|_{H^s}^2\le C\left(\delta_0^2+A_0^2+\delta_0^2\mathcal{M}(t)^2\right).
			\end{equation}
			It then follows from \eqref{lem-hs4} that
			$$
			\mathcal{M}(t)^2\leq C(\delta_0^2+A_0^2+\delta_0^2\mathcal{M}(t)^2).
			$$
			Due to the smallness of $\delta_0$ and the boundedness of $A_0$, it gives
			\begin{equation}\label{lem-hs5}
				\mathcal{M}(t)\leq C(\delta_0+A_0).
			\end{equation}
			Noting that since $(\rho_0-\rho_*,u_0,\theta_0-\theta_*,q_0)\in L^1(\mathbb{R}^3)$, we can obtain from \eqref{in-data2} that $A_0=\|(n_0,w_0,\phi_0,\psi_0)\|_{L^1}$ is bounded. As a result of \eqref{lem-hs5} and \eqref{THZT1}, it holds
			$$
			\|(n,w,\phi,\psi)(t)\|_{H^s}\leq C(\delta_0+A_0)(1+t)^{-\frac{3}{4}}\leq C(1+t)^{-\frac{3}{4}}.
			$$
			This completes the proof of Lemma \ref{lem-hs}.
		\end{proof}
		
		Next, we are going to enhance the time-decay rate of the derivatives $\nabla^k(n,w,\phi,\psi)$ in $L^2$ norm for $ 1\leq k\leq s$.
		\begin{lem}\label{lem-hk}
			Assume that the conditions in Theorem \ref{thm1} hold. Let $(n,w,\phi,\psi)$ be the classical solution of the problem \eqref{nsac}-\eqref{non-f}. Then, it holds for $1\leq k\leq s-1$ that
			\begin{equation}\label{lem-hk0}
				\|\nabla^k(n,w,\phi)(t)\|_{L^2}\leq C(1+t)^{-\frac{3}{4}-\frac{k}{2}},\quad
				\|\nabla^k\psi(t)\|_{L^2}\leq C(1+t)^{-\frac{5}{4}-\frac{k}{2}},
			\end{equation}
			and
			\begin{equation}\label{lem-hk1}
				\|\nabla^s(n,w,\phi,\psi)(t)\|_{L^2}\leq C(1+t)^{-\frac{3}{4}-\frac{s}{2}},
			\end{equation}
			where the positive constant $C$ is independent of time.
		\end{lem}
		\begin{proof}
			In terms of Lemmas \ref{lem-n}-\ref{lem-th} and Proposition \ref{pro-est},  it hold for $\alpha\geq 1$ that
			\begin{align}
				&\frac{1}{2}\frac{\mathrm{d}}{\mathrm{d}t}\left(\|\nabla^\alpha n\|_{L^2}^2+\|\nabla^\alpha w\|_{L^2}^2+\|\nabla^\alpha \phi\|_{L^2}^2+\frac{1}{1+n}\|\nabla^{\alpha}\psi\|_{L^2}^2\right)\notag\\
				&\quad+\frac{3\nu}{4}\|\nabla^{\alpha}\nabla w\|_{L^2}^2+\frac{3(\nu+\eta)}{4}\|\nabla^\alpha \text{div} w\|_{L^2}^2
				+\frac{1}{2\tau}\|\nabla^{\alpha} \psi\|_{L^2}^2\notag\\
				&\leq C\delta_0(\|\nabla^\alpha n\|_{L^2}^2+\|\nabla^\alpha w\|_{L^2}^2+\|\nabla^{\alpha}\nabla w\|_{L^2}^2+\|\nabla^{\alpha}\text{div} w\|_{L^2}^2+\|\nabla^\alpha \phi\|_{L^2}^2).\label{lem-hk2}
			\end{align}
			Applying the operator $\nabla^{\alpha-1} P_{\infty}$ to $\eqref{nsac}_2$, multiplying the resulting equation by $\nabla^{\alpha} n^h$, and then integrating it over $\mathbb{R}^3$, and using \eqref{2-07-01} and \eqref{2-07-02},  we obtain, for $\alpha\geq 1$ that
			\begin{align}
				&c\|\nabla^\alpha n^h\|_{L^2}^2+\frac{\mathrm{d}}{\mathrm{d}t}\int_{\mathbb{R}^3} \nabla^{\alpha-1}w^h\cdot\nabla^\alpha n^h\mathrm{d}x\notag\\
				&=\int_{\mathbb{R}^3} (\nu\nabla^{\alpha-1}\Delta w^h+(\nu+\eta)\nabla^{\alpha}{\rm{div}}w^h-\sigma \nabla^{\alpha}\phi^h)\cdot\nabla^\alpha n^h\mathrm{d}x+\int_{\mathbb{R}^3} \nabla^{\alpha-1}f_2^h\cdot\nabla^{\alpha}n^h\mathrm{d}x\notag\\
				&\quad-c\int_{\mathbb{R}^3}(\nabla^{\alpha}{\rm{div}}w^h+\nabla^{\alpha}f_1^h)\cdot \nabla^{\alpha-1}w\mathrm{d}x\notag\\
				&\triangleq X_1+X_2+X_3.
			\end{align}
			The first term $X_1$ is bounded by
			\begin{align}
				|X_1|&\leq  (\nu+\eta)(\|\nabla^\alpha \nabla w^h\|_{L^2}+\|\nabla^\alpha \text{div}w^h\|_{L^2})\|\nabla^\alpha n^h\|_{L^2}+\sigma\|\nabla^\alpha \phi^h\|_{L^2}\|\nabla^\alpha n^h\|_{L^2}\notag\\
				&\leq\varepsilon\|\nabla^{\alpha}n^h\|_{L^2}^2+C_{\varepsilon}(\|\nabla^\alpha \nabla w^h\|_{L^2}^2+\|\nabla^\alpha \text{div}w^h\|_{L^2}^2)+\frac{c}{4}\|\nabla^\alpha n^h\|_{L^2}^2+\frac{\sigma^2}{c}\|\nabla^\alpha \phi^h\|_{L^2}^2\notag\\
				&\leq (\varepsilon+\frac{c}{4})\|\nabla^\alpha n^h\|_{L^2}^2+C(\|\nabla^\alpha \nabla w^h\|_{L^2}^2+\|\nabla^\alpha \text{div}w^h\|_{L^2}^2)+\frac{\sigma^2}{c}\|\nabla^\alpha \phi^h\|_{L^2}^2\notag\\
				&\leq (\varepsilon+\frac{c}{4})\|\nabla^\alpha n\|_{L^2}^2+C(\|\nabla^\alpha \nabla w\|_{L^2}^2+\|\nabla^\alpha \text{div}w\|_{L^2}^2)+\frac{\sigma^2}{c}\|\nabla^\alpha \phi^h\|_{L^2}^2,
			\end{align}
			where $\varepsilon>0$ is arbitrary. For the term $X_2$, recalling the definition of $f_2$ in \eqref{non-f}, we only deal with the first term of $X_2$, i.e.,
			\begin{align}
				&	\left|\int_{\mathbb{R}^3}\nabla^{\alpha-1}P_\infty(w\cdot\nabla w)\cdot\nabla^\alpha n^h\mathrm{d}x\right|\notag\\
				&\leq \left|\int_{\mathbb{R}^3}\nabla^{\alpha}P_\infty(w\cdot\nabla w)\cdot\nabla^{\alpha-1} n^h\mathrm{d}x\right|\notag
				\\
				&\leq \left|\int_{\mathbb{R}^3}\nabla^{\alpha}(w\cdot\nabla w)\cdot\nabla^{\alpha-1} n^h\mathrm{d}x\right|+\left|\int_{\mathbb{R}^3}\nabla^{\alpha}P_1(w\cdot\nabla w)\cdot\nabla^{\alpha-1} n^h\mathrm{d}x\right|.\label{THZ01}
			\end{align}
			It is easy to prove that
			\begin{align*}
				&\left|\int_{\mathbb{R}^3}\nabla^{\alpha}(w\cdot\nabla w)\cdot\nabla^{\alpha-1} n^h\mathrm{d}x\right| \\
				&\leq \left|\int_{\mathbb{R}^3}[\nabla^{\alpha}(w\cdot\nabla w)-w\cdot\nabla^\alpha\nabla w]\cdot\nabla^{\alpha-1} n^h\mathrm{d}x\right|+\left|\int_{\mathbb{R}^3}w\cdot\nabla^\alpha\nabla w\cdot\nabla^{\alpha-1} n^h\mathrm{d}x\right|\\
				&\leq C(\|\nabla w\|_{L^\infty}\|\nabla^{\alpha}w\|_{L^2}+\|\nabla^{\alpha}w\|_{L^2}\|\nabla w\|_{L^\infty})\|\nabla^{\alpha-1} n^h\|_{L^2}+\|w\|_{L^\infty}\|\nabla^\alpha\nabla w\|_{L^2}\|\nabla^{\alpha-1} n^h\|_{L^2}\\
				&\leq C\|\nabla^2 w\|_{H^1}\|\nabla^{\alpha}w\|_{L^2}\|\nabla^{\alpha} n^h\|_{L^2}+C\|\nabla w\|_{H^1}\|\nabla^\alpha\nabla w\|_{L^2}\|\nabla^\alpha n^h\|_{L^2}\\
				&\leq C\delta_0(\|\nabla^\alpha n^h\|_{L^2}^2+\|\nabla^\alpha w\|_{L^2}^2+\|\nabla^\alpha\nabla w\|_{L^2}^2),
			\end{align*}
			where we have used the fact that $\|\nabla^{\alpha-1}n^h\|_{L^2}\leq C\|\nabla^\alpha n^h\|_{L^2}$.
			The second term on the right-hand side of 	\eqref{THZ01} is bounded by
			\begin{align*}
				\left|\int_{\mathbb{R}^3}\nabla^{\alpha}P_1(w\cdot\nabla w)\cdot\nabla^{\alpha-1} n^h\mathrm{d}x\right|
				&\leq \|\nabla^\alpha P_1(w\cdot\nabla w)\|_{L^2}\|\nabla^{\alpha-1}n^h\|_{L^2}\\
				&\leq C\|\nabla^\alpha (w\cdot\nabla w)\|_{L^2}\|\nabla^{\alpha}n^h\|_{L^2}\\
				&\leq C(\|w\|_{L^\infty}\|\nabla^{\alpha}\nabla w\|_{L^2}+\|\nabla^\alpha w\|_{L^2}\|\nabla w\|_{L^\infty})\|\nabla^{\alpha}n^h\|_{L^2}\\
				&\leq C\delta_0 (\|\nabla^\alpha n^h\|_{L^2}^2+\|\nabla^\alpha w\|_{L^2}^2+\|\nabla^\alpha \nabla w\|_{L^2}^2).
			\end{align*}
			Therefore, we have
			\begin{align*}
				\left|\int_{\mathbb{R}^3}\nabla^{\alpha-1}P_\infty(w\cdot\nabla w)\cdot\nabla^\alpha n^h\mathrm{d}x\right|
				\leq C\delta_0 (\|\nabla^\alpha n^h\|_{L^2}^2+\|\nabla^\alpha w\|_{L^2}^2+\|\nabla^\alpha \nabla w\|_{L^2}^2).
			\end{align*}
			The remainders of $X_2$ can be estimated similarly. Thus, we have
			\begin{align*}
				|X_2|&\leq \Big|\int_{\mathbb{R}^3} \nabla^{\alpha-1}P_{\infty}\Big(-cw\cdot\nabla w+\frac{(cn-\sigma\phi)\nabla n}{1+n}-\frac{\nu n\Delta w+(\nu+\eta)n\nabla{\rm{div}}w}{1+n}\Big)\cdot \nabla^{\alpha}n^h\mathrm{d}x\Big|\\
				&\leq C\delta_0(\|\nabla^\alpha n\|_{L^2}^2+\|\nabla^\alpha n^h\|_{L^2}^2+\|\nabla^{\alpha} w\|_{L^2}^2
				+\|\nabla^\alpha\nabla w\|_{L^2}^2+\|\nabla^\alpha\text{div} w\|_{L^2}^2+\|\nabla^\alpha \phi\|_{L^2}^2).
			\end{align*}
			Finally, we deal with the term $X_3$. It holds
			\begin{align}
				|X_3|
				&\leq \Big| c\int_{\mathbb{R}^3}(\nabla^{\alpha}w^h+\nabla^{\alpha}P_\infty(nw))\cdot \nabla^{\alpha}w^h \mathrm{d}x\Big|\notag\\
				&\leq c\|\nabla^\alpha w^h\|_{L^2}^2+C\|\nabla^\alpha P_\infty(nw)\|_{L^2}\|\nabla^\alpha w^h\|_{L^2}\notag\\
				&\leq c\|\nabla^\alpha w^h\|_{L^2}^2+C\|\nabla^\alpha (nw)\|_{L^2}\|\nabla^\alpha w^h\|_{L^2}\notag\\
				&\leq c\|\nabla^\alpha w^h\|_{L^2}^2+C(\|n\|_{L^\infty}\|\nabla^{\alpha} w\|_{L^2}+\|w\|_{L^\infty}\|\nabla^\alpha n\|_{L^2})\|\nabla^\alpha w\|_{L^2}\notag\\
				&\leq C\|\nabla^\alpha w^h\|_{L^2}^2+C\delta_0(\|\nabla^\alpha n\|_{L^2}^2+\|\nabla^\alpha w\|_{L^2}^2).
			\end{align}
			Combining the estimates of $X_i$ $(i=1,2,3)$ and noting the smallness of $\delta_0$ and $\varepsilon$ yield
			\begin{align}
				&\frac{c}{2}\|\nabla^\alpha n^h\|_{L^2}^2+\frac{\mathrm{d}}{\mathrm{d}t}\int_{\mathbb{R}^3} \nabla^{\alpha-1}w^h\cdot\nabla^\alpha n^h\mathrm{d}x\notag\\
				&\leq C\delta_0(\|\nabla^\alpha n\|_{L^2}^2+\|\nabla^\alpha \phi\|_{L^2}^2)+\frac{\sigma^2}{c}\|\nabla^\alpha\phi^h\|_{L^2}^2\notag\\
				&\quad+C(\|\nabla^\alpha w\|_{L^2}^2+\|\nabla^{\alpha}\nabla w\|_{L^2}^2
				+\|\nabla^{\alpha}\text{div} w\|_{L^2}^2).\label{THZ02}
			\end{align}
			Similarly, it also holds for $\alpha\geq 1$ that
			\begin{align}
				&\frac{3b}{4}\|\nabla^\alpha \phi^h\|_{L^2}^2+\frac{\mathrm{d}}{\mathrm{d}t}\int_{\mathbb{R}^3} \nabla^{\alpha-1}\psi^h\cdot\nabla^\alpha \phi^h\mathrm{d}x\notag\\
				&\leq  C\delta_0(\|\nabla^\alpha n\|_{L^2}^2+\|\nabla^\alpha \nabla w\|_{L^2}^2+\|\nabla^\alpha\text{div}w\|_{L^2}^2+\|\nabla^\alpha \phi\|_{L^2}^2+\|\nabla^\alpha\psi\|_{L^2}^2)\notag\\
				&\quad+C(\|\nabla^\alpha  w\|_{L^2}^2+\|\nabla^{\alpha}\psi^h \|_{L^2}^2).\label{2-001}
			\end{align}
			Taking $\eqref{THZ02}\times\frac{bc}{4\sigma^2}+\eqref{2-001}$, one has
			\begin{align}
				&\frac{bc^2}{8\sigma^2}\|\nabla^\alpha n^h\|_{L^2}^2+\frac{b}{2}\|\nabla^\alpha\phi^h\|_{L^2}^2+\frac{\mathrm{d}}{\mathrm{d}t}\Big\{\int_{\mathbb{R}^3} \Big(\frac{bc}{4\sigma^2}\nabla^{\alpha-1}w^h\cdot\nabla^\alpha n^h+\nabla^{\alpha-1}\psi^h\cdot\nabla^\alpha \phi^h\Big)\mathrm{d}x\Big\}\notag\\
				&\leq C\delta_0(\|\nabla^\alpha n\|_{L^2}^2+\|\nabla^\alpha \phi\|_{L^2}^2+\|\nabla^\alpha\psi\|_{L^2}^2)\notag\\
				&\quad+C(\|\nabla^\alpha w\|_{L^2}^2+\|\nabla^\alpha\nabla w\|_{L^2}^2+\|\nabla^\alpha \text{div}w\|_{L^2}^2+\|\nabla^{\alpha}\psi^h\|_{L^2}^2).\label{THZ03}
			\end{align}
			Let the positive constant $\beta_2$ be sufficiently small. Then, taking $\eqref{lem-hk2}+\beta_2\times\eqref{THZ03}$ yields
			$$
			\begin{aligned}
				&	\frac{\mathrm{d}}{\mathrm{d}t}\Big\{\frac{1}{2}\Big(\|\nabla^\alpha n\|_{L^2}^2+\|\nabla^\alpha w\|_{L^2}^2+\|\nabla^\alpha \phi\|_{L^2}^2+\frac{1}{1+n}\|\nabla^{\alpha}\psi\|_{L^2}^2\Big)\\
				&\quad+\beta_2\int_{\mathbb{R}^3} \Big(\frac{bc}{4\sigma^2}\nabla^{\alpha-1}w^h\cdot\nabla^\alpha n^h+\nabla^{\alpha-1}\psi^h\cdot\nabla^\alpha \phi^h\Big)\mathrm{d}x\Big\}\\
				&\quad+\frac{3\nu}{4}\|\nabla^{\alpha}\nabla w\|_{L^2}^2+\frac{3(\nu+\eta)}{4}\|\nabla^\alpha \text{div} w\|_{L^2}^2+\frac{1}{2\tau}\|\nabla^\alpha\psi\|_{L^2}^2+\frac{bc^2\beta_2}{8\sigma^2}\|\nabla^\alpha n^h\|_{L^2}^2+\frac{b\beta_2}{2}\|\nabla^\alpha \phi^h\|_{L^2}^2\\
				&\leq  C\delta_0(\|\nabla^\alpha n\|_{L^2}^2+\|\nabla^\alpha w\|_{L^2}^2+\|\nabla^{\alpha}\nabla w\|_{L^2}^2+\|\nabla^{\alpha}\text{div} w\|_{L^2}^2+\|\nabla^\alpha \phi\|_{L^2}^2+\|\nabla^\alpha\psi\|_{L^2}^2)\\
				&\quad+ C\delta_0\beta_2(\|\nabla^\alpha n\|_{L^2}^2+\|\nabla^\alpha \phi\|_{L^2}^2)+C\beta_2(\|\nabla^\alpha w\|_{L^2}^2+\|\nabla^\alpha\nabla w\|_{L^2}^2+\|\nabla^\alpha \text{div}w\|_{L^2}^2+\|\nabla^{\alpha}\psi^h\|_{L^2}^2).\label{THZ04}
			\end{aligned}
			$$
			Thanks to the smallness of $\beta_2$ and $\delta_0$, we verify for $\alpha\geq 1$ that
			\begin{align}
				&	\frac{\mathrm{d}}{\mathrm{d}t}\Big\{ \frac{1}{2}\Big(\|\nabla^\alpha n\|_{L^2}^2+\|\nabla^\alpha w\|_{L^2}^2+\|\nabla^\alpha \phi\|_{L^2}^2+\frac{1}{1+n}\|\nabla^{\alpha}\psi\|_{L^2}^2\Big)\notag\\
				&\quad+\beta_2\int_{\mathbb{R}^3}\Big(\frac{bc}{4\sigma^2}\nabla^{\alpha-1}w^h\cdot\nabla^\alpha n^h+\nabla^{\alpha-1}\psi^h\cdot\nabla^\alpha \phi^h\Big)\mathrm{d}x\Big\}\notag\\
				&\quad+\frac{\nu}{2}\|\nabla^{\alpha}\nabla w\|_{L^2}^2+\frac{\nu+\eta}{2}\|\nabla^\alpha \text{div} w\|_{L^2}^2+\frac{1}{4\tau}\|\nabla^\alpha\psi\|_{L^2}^2+\frac{bc^2\beta_2}{16\sigma^2}\|\nabla^\alpha n^h\|_{L^2}^2+\frac{b\beta_2}{4}\|\nabla^\alpha \phi^h\|_{L^2}^2\notag\\
				&\leq  C\delta_0(\|\nabla^{\alpha}n\|_{L^2}^2+\|\nabla^{\alpha}w\|_{L^2}^2+\|\nabla^{\alpha}\phi\|_{L^2}^2+\|\nabla^\alpha \psi^h\|_{L^2}^2).\label{lem-hk4}
			\end{align}
			It is easy to verify that
			\begin{equation*}
				\begin{aligned}
					&\frac{1}{2}\Big(\|\nabla^\alpha n\|_{L^2}^2+\|\nabla^\alpha w\|_{L^2}^2+\|\nabla^\alpha \phi\|_{L^2}^2+\frac{1}{1+n}\|\nabla^{\alpha}\psi\|_{L^2}^2\Big)\\
					&\quad +\beta_2\int_{\mathbb{R}^3}\Big(\frac{bc}{4\sigma^2}\nabla^{\alpha-1}w^h\cdot\nabla^\alpha n^h+\nabla^{\alpha-1}\psi^h\cdot\nabla^\alpha \phi^h\Big)\mathrm{d}x\\
					&\sim (\|\nabla^\alpha n\|_{L^2}^2+\|\nabla^\alpha w\|_{L^2}^2+\|\nabla^\alpha \phi\|_{L^2}^2+\|\nabla^{\alpha}\psi\|_{L^2}^2),
				\end{aligned}
			\end{equation*}
			which, together with \eqref{lem-hk4}, \eqref{2-07-01} and \eqref{2-07-02}, yields the following results, for $\alpha\geq 1$,
			\begin{align}
				&\frac{\mathrm{d}}{\mathrm{d}t}\Big(\|\nabla^\alpha (n,w,\phi,\psi)\|_{L^2}^2+\beta_2\int_{\mathbb{R}^3}\big(\nabla^{\alpha-1}w^h\cdot\nabla^\alpha n^h+\nabla^{\alpha-1}\psi^h\cdot\nabla^\alpha \phi^h\big)\mathrm{d}x\Big)\notag\\
				&\quad+C_3\Big(\|\nabla^\alpha (n,w,\phi,\psi)\|_{L^2}^2+\beta_2\int_{\mathbb{R}^3}\big(\nabla^{\alpha-1}w^h\cdot\nabla^\alpha n^h+\nabla^{\alpha-1}\psi^h\cdot\nabla^\alpha \phi^h\big)\mathrm{d}x\Big)\notag\\
				&\leq C(\|\nabla^{\alpha}n^\ell\|_{L^2}^2+\|\nabla^{\alpha}w^\ell\|_{L^2}^2+\|\nabla^{\alpha}\phi^\ell\|_{L^2}^2+\|\nabla^\alpha \psi^\ell \|_{L^2}^2),\label{lem-hk5}
			\end{align}
			where $C_3$ is a positive constant independent of time. Summing up \eqref{lem-hk5} from $k\leq \alpha\leq s$ and using the definition of $\mathcal{E}_k^s(t)$, one has
			\begin{equation}\label{lem-hk-e}
				\frac{\mathrm{d}}{\mathrm{d}t}\mathcal{E}_{k}^s(t)+C_4\mathcal{E}_{k}^s(t)\leq C\|\nabla^{k}(n^{\ell},w^{\ell},\phi^{\ell},\psi^\ell)(t)\|_{L^2}^2.
			\end{equation}
			
			To investigate the time-decay rates of the norm $\|\nabla(n,w,\phi,\psi)(t)\|_{H^{s-1}}$, we follow in a similar method as above. First, we calculate
			\begin{align}
				\|\nabla(n^{\ell},w^{\ell},\phi^{\ell})(t)\|_{L^2}&\leq C(1+t)^{-\frac{5}{4}}\|(n_0,w_0,\phi_0,\psi_0)\|_{L^1}+C\int_0^{\frac{t}{2}}(1+t-r)^{-\frac{5}{4}}\|F(r)\|_{L^1}\mathrm{d}r\notag\\
				&\quad+C\int_{\frac{t}{2}}^t(1+t-r)^{-\frac{3}{4}}\|\nabla F(r)\|_{L^1}\mathrm{d}r,\label{lem-hk6}
			\end{align}
			and
			\begin{align}
				\|\nabla\psi^\ell(t)\|_{L^2}&\leq C(1+t)^{-\frac{7}{4}}\|(n_0,w_0,\phi_0,\psi_0)\|_{L^1}+C\int_0^{\frac{t}{2}}(1+t-r)^{-\frac{7}{4}}\|F(r)\|_{L^1}\mathrm{d}r\notag\\
				&\quad+C\int_{\frac{t}{2}}^t(1+t-r)^{-\frac{5}{4}}\|\nabla F(r)\|_{L^1}\mathrm{d}r.\label{lem-hk61}
			\end{align}
			Next, we need to estimate the second term on the right-hand side of \eqref{lem-hk6}.  By the definition of $F$, \eqref{l-inf-est}, \eqref{non-f} and \eqref{lem-hs0}, one has
			\begin{align}
				\|F(r)\|_{L^1}&\leq C(\|n\|_{L^2}\|\nabla w\|_{L^2}+\|w\|_{L^2}\|\nabla n\|_{L^2}+\|w\|_{L^2}\|\nabla w\|_{L^2}+\|\phi\|_{L^2}\|\nabla n\|_{L^2}+\|n\|_{L^2}\|\nabla n\|_{L^2}\notag\\
				&\quad+\|n\|_{L^2}\|\nabla^2w\|_{L^2}+\|w\|_{L^2}\|\nabla \phi\|_{L^2}+\|\phi\|_{L^2}\|\nabla w\|_{L^2}+\|n\|_{L^2}\|\nabla\psi\|_{L^2}+\|\nabla w\|_{L^2}^2)\notag\\
				&\leq C(1+r)^{-\frac{3}{2}}.\label{lem-hk7}
			\end{align}
			Similarly, we also have
			\begin{equation}\label{lem-hk8}
				\|\nabla F(r)\|_{L^1}\leq C(1+r)^{-\frac{3}{2}}.
			\end{equation}
			Substituting \eqref{lem-hk7}, \eqref{lem-hk8} into \eqref{lem-hk6}, \eqref{lem-hk61} yields
			\begin{equation*}
				\begin{aligned}
					\|\nabla(n^{\ell},w^{\ell},\phi^{\ell})(t)\|_{L^2}&\leq C(1+t)^{-\frac{5}{4}}+C\int_0^{\frac{t}{2}}(1+t-r)^{-\frac{5}{4}}(1+r)^{-\frac{3}{2}}\mathrm{d}r\\
					&\quad+C\int_{\frac{t}{2}}^t(1+t-r)^{-\frac{3}{4}}(1+r)^{-\frac{3}{2}}\mathrm{d}r\\
					&\leq C(1+t)^{-\frac{5}{4}},
				\end{aligned}
			\end{equation*}
			and
			\begin{equation*}
				\begin{aligned}
					\|\nabla \psi^{\ell}(t)\|_{L^2}&\leq C(1+t)^{-\frac{7}{4}}+C\int_0^{\frac{t}{2}}(1+t-r)^{-\frac{7}{4}}(1+r)^{-\frac{3}{2}}\mathrm{d}r\\
					&\quad+C\int_{\frac{t}{2}}^t(1+t-r)^{-\frac{5}{4}}(1+r)^{-\frac{3}{2}}\mathrm{d}r\\
					&\leq C(1+t)^{-\frac{3}{2}}.
				\end{aligned}
			\end{equation*}
			By \eqref{lem-hk-e}, we have
			\begin{equation*}
				\frac{\mathrm{d}}{\mathrm{d}t}\mathcal{E}_{1}^s(t)+C_4\mathcal{E}_{1}^s(t)\leq C(\|\nabla(n^{\ell},w^{\ell},\phi^{\ell})(t)\|_{L^2}^2+\|\nabla\psi^\ell(t)\|_{L^2}^2) \leq C(1+t)^{-\frac{5}{2}}.
			\end{equation*}
		 It holds
			\begin{equation*}
				\mathcal{E}_1^s(t)\leq C(1+t)^{-\frac{5}{2}}.
			\end{equation*}
		Due to the smallness of $\beta_2$, we obtain that $\mathcal{E}_k^s(t)\sim \|\nabla^k(n,w,\phi,\psi)\|_{H^{s-k}}^2$, which gives that
			\begin{equation}\label{THZ1}
				\|\nabla (n,w,\phi,\psi)(t)\|_{H^{s-1}}\leq C(1+t)^{-\frac{5}{4}},
			\end{equation}
			for $s\geq 3$. Based on the above facts, we can further show that, for high-order derivatives $(1\leq k \leq s)$, it holds
			\begin{equation}\label{A-01}
				\|\nabla^k (n,w,\phi,\psi)(t)\|_{H^{s-k}}\leq C(1+t)^{-\frac{3}{4}-\frac{k}{2}}.
			\end{equation}
			In fact, we can prove the above time-decay rates \eqref{A-01} by induction method. The case of $k=1$ is proved by \eqref{THZ1}.  Then, we assume it holds for $k-1$, i.e.,
			\begin{align*}
				\|\nabla^{k-1}(n,w,\phi,\psi)(t)\|_{H^{s-k+1}}\leq C(1+t)^{-\frac{1}{4}-\frac{k}{2}},\quad 2\leq k\leq s.
			\end{align*}
			The next goal is to establish the time-decay rates of $\|\nabla^k(n,w,\phi,\psi)(t)\|_{H^{s-k}}$ for $2\leq k\leq s$. Noting that
			\begin{align}
				\|\nabla^k(n^{\ell},w^{\ell},\phi^{\ell},\psi^\ell)(t)\|_{L^2}&\leq C(1+t)^{-\frac{3}{4}-\frac{k}{2}}\|(n_0,w_0,\phi_0,\psi_0)\|_{L^1}+C\int_0^{\frac{t}{2}}(1+t-r)^{-\frac{3}{4}-\frac{k}{2}}\|F(r)\|_{L^1}\mathrm{d}r\notag\\
				&\quad+C\int_{\frac{t}{2}}^t(1+t-r)^{-\frac{7}{4}}\|\nabla^{k-2}F(r)\|_{L^1}\mathrm{d}r.\label{lem-hk9}
			\end{align}
			It should be noted that at this moment, we have
			\begin{equation*}
				\|(n,w,\phi,\psi)(t)\|_{L^2}\leq C(1+t)^{-\frac{3}{4}},\quad \|\nabla^{k-1}(n,w,\phi,\psi)(t)\|_{H^{s-k+1}}\leq C(1+t)^{-\frac{1}{4}-\frac{k}{2}}.
			\end{equation*}
			As a result, we obtain
			\begin{equation}\label{lem-hk10}
				\begin{aligned}
					\|F(r)\|_{L^1} \leq C(1+r)^{-2},\quad \|\nabla^{k-2}F(r)\|_{L^1}\leq C(1+r)^{-\frac{3}{4}-\frac{k}{2}}.
				\end{aligned}
			\end{equation}
			Substituting \eqref{lem-hk10} into \eqref{lem-hk9} yields
			\begin{equation*}
				\begin{aligned}
					\|\nabla^k(n^{\ell},w^{\ell},\phi^{\ell},\psi^\ell)(t)\|_{L^2}&\leq C(1+t)^{-\frac{3}{4}-\frac{k}{2}}+C\int_0^{\frac{t}{2}}(1+t-r)^{-\frac{3}{4}-\frac{k}{2}}(1+r)^{-2}\mathrm{d}r\\
					&\quad+C\int_{\frac{t}{2}}^t(1+t-r)^{-\frac{7}{4}}(1+r)^{-\frac{3}{4}-\frac{k}{2}}\mathrm{d}r\\
					&\leq C(1+t)^{-\frac{3}{4}-\frac{k}{2}}.
				\end{aligned}
			\end{equation*}
			In view of \eqref{lem-hk-e}, it holds
			\begin{equation}\label{lem-hk12}
				\frac{d}{\mathrm{d}r}\mathcal{E}_{k}^s(t)+C_4\mathcal{E}_{k}^s(t)\leq C(1+t)^{-\frac{3}{2}-k}.
			\end{equation}
			Applying Gr\"{o}nwall's inequality to \eqref{lem-hk12}, we have
			\begin{equation*}
				\mathcal{E}_k^s(t)\leq C(1+t)^{-\frac{3}{2}-k},
			\end{equation*}
			which implies for $1\leq k\leq s$ and $s\geq 3$ that
			\begin{equation}\label{321}
				\|\nabla^k(n,w,\phi,\psi)(t)\|_{H^{s-k}}\leq C(1+t)^{-\frac{3}{4}-\frac{k}{2}}.
			\end{equation}
			This completes the proof of \eqref{A-01}.
			
			To enhance the time-decay rate of $\|\psi\|_{L^2}$, we need to apply the damping structure of $\psi$. By $\eqref{nsac}_4$, the solution $\psi$ can be written as
			\begin{align*}
				\psi(x,t)=e^{-\frac{1}{\tau} t}\psi_0-b\int_0^te^{-\frac{1}{\tau}(t-r)}\nabla\phi(x,r)\mathrm{d}r.
			\end{align*}
			In terms of \eqref{321}, one has
			$$
			\begin{aligned}
				\|\psi(t)\|_{L^2}
				&\leq  e^{-\frac{1}{\tau} t}\|\psi_0\|_{L^2}+b\int_0^t e^{-\frac{1}{\tau}(t-r)}\|\nabla\phi(r)\|_{L^2}\mathrm{d}r\\
				&\leq C(1+t)^{-\frac{5}{4}}\|\psi_0\|_{L^2}+C\int_0^t(1+t-r)^{-\frac{5}{4}}(1+r)^{-\frac{5}{4}}\mathrm{d}r\\
				&\leq C(1+t)^{-\frac{5}{4}}.
			\end{aligned}
			$$
			It then follows the same method for the derivatives $\nabla^k\psi$ to show that
			\begin{equation}\label{322}
				\|\nabla^{k}\psi(t)\|_{L^2}\leq C(1+t)^{-\frac{5}{4}-\frac{k}{2}}\quad \text{for}\quad 1\leq k\leq s-1.
			\end{equation}
			The combination of \eqref{321} and \eqref{322} implies for $1\leq k\leq s-1$ that
			\begin{align*}
				\|\nabla^k(n,w,\phi)(t)\|_{L^2}\leq C(1+t)^{-\frac{3}{4}-\frac{k}{2}},\quad
				\|\nabla^{k}\psi(t)\|_{L^2}\leq C(1+t)^{-\frac{5}{4}-\frac{k}{2}},
			\end{align*}
			and
			\begin{align*}
				\|\nabla^{s}(n,w,\phi,\psi)(t)\|_{L^2}\leq C(1+t)^{-\frac{3}{4}-\frac{s}{2}}.
			\end{align*}
			Therefore, we complete the proof Lemma \ref{lem-hk}.
		\end{proof}
		
		It should be pointed out that at this moment, the time-decay rate of the highest order derivatives of $\psi$ is not optimal and can be improved. Thus, we aim to establish the decay estimates $\|\nabla^s\psi\|_{L^2}$ under some new observations in the following lemma.
		
		\begin{lem}\label{lemphi}
			Assume that the conditions in Theorem \ref{thm1} hold. Let $(n,w,\phi,\psi)$ be the classical solution of the problem \eqref{nsac}-\eqref{non-f}. Then it holds
			\begin{align*}
				\|\nabla^s\psi(t)\|_{L^2}\leq C(1+t)^{-\frac{5}{4}-\frac{s}{2}},
			\end{align*}
			where the positive constant $C$ is independent of time.
		\end{lem}
		\begin{proof}
			Applying the operator $ P_{\infty}$ to $\eqref{nsac}$, the high-frequency part of the solution can be written as
			\begin{equation}\label{nsac1}
				\left\{
				\begin{aligned}
					&\partial_tn^h +c\,{\rm{div}}w^h=f_1^h, \\
					&\partial_tw^h+c\nabla n^h+\sigma \nabla \phi^h=\nu\Delta w^h+(\nu+\eta)\nabla{\rm{div}}w^h+f_2^h,\\
					&\partial_t\phi^h+\sigma {\rm{div}}w^h+b\text{div}\psi^h=f_3^h,\\
					&\partial_t\psi^h+\frac{1}{\tau}\psi^h+b\nabla\phi^h=0,
				\end{aligned}
				\right.
			\end{equation}
			where the nonliner terms $f_i^h=P_{\infty}f_i$ and $f_i$ are defined in \eqref{non-f}.
			
			Applying the operator $\nabla^s$ to $\eqref{nsac1}_1$, multiplying the resulting equation by $\nabla^s n^h$, and then integrating over $\mathbb{R}^3$, we obtain
			\begin{align*}
				&\frac{1}{2}\frac{\mathrm{d}}{\mathrm{d}t}\|\nabla^sn^h\|_{L^2}^2+c\int_{\mathbb{R}^3} \nabla^s\text{div}w^h\cdot\nabla^sn^h \mathrm{d}x\\			&=\int_{\mathbb{R}^3}\nabla^sf_1^h\cdot\nabla^sn^h\mathrm{d}x=\int_{\mathbb{R}^3}\nabla^sf_1\cdot\nabla^sn^h\mathrm{d}x-\int_{\mathbb{R}^3}\nabla^s f_1^\ell\cdot\nabla^sn^h\mathrm{d}x.
			\end{align*}
			By Lemma \ref{lem-commu}, \eqref{pro-est1} and H$\ddot{\text{o}}$lder's inequality, one has
			\begin{align*}
				&\Big|\int_{\mathbb{R}^3}\nabla^sf_1\cdot\nabla^sn^h\mathrm{d}x\Big|\\
				&\leq c\Big|\int_{\mathbb{R}^3} \nabla^s(w\cdot\nabla n)\cdot\nabla^s n^h \mathrm{d}x\Big|
				+c\Big|\int_{\mathbb{R}^3}\nabla^s (n\text{div}w)\cdot\nabla^s n^h \mathrm{d}x\Big|\\
				&\leq c\left| \int_{\mathbb{R}^3} (\nabla^s(w\cdot\nabla n)-w\cdot\nabla^s\nabla n)\cdot\nabla^s n^h\mathrm{d}x\right|+c\left|\int_{\mathbb{R}^3} w\cdot\nabla\nabla^s n\cdot\nabla^sn^h\mathrm{d}x\right|\\
				&\quad+c\left| \int_{\mathbb{R}^3} (\nabla^s(n\text{div}w )-n\nabla^s\text{div}w)\cdot\nabla^s n^h\mathrm{d}x\right|+c\left|\int_{\mathbb{R}^3} n\nabla^s\text{div} w\cdot\nabla^s n^h\mathrm{d}x\right|\\
				&\leq C(\|\nabla w\|_{L^\infty}\|\nabla^s n\|_{L^2}+\|\nabla n\|_{L^\infty}\|\nabla^s w\|_{L^2})\|\nabla^sn^h\|_{L^2}+C\|\text{div} w\|_{L^\infty}\|\nabla^s n^h\|_{L^2}^2
				+C\|w\|_{L^\infty}\|\nabla^sn\|_{L^2}^2\\
				&\quad+C(\|\nabla n\|_{L^\infty}\|\nabla^s w\|_{L^2}+\|\nabla^s n\|_{L^2}\|\text{div}w\|_{L^\infty})\|\nabla^s n^h\|_{L^2}+C\|n\|_{L^\infty}\|\nabla^s\text
				{div}w\|_{L^2}\|\nabla^sn^h\|_{L^2}  \\
				&\leq C(\|\nabla w\|_{L^\infty}\|\nabla^s n\|_{L^2}+\|\nabla n\|_{L^\infty}\|\nabla^s w\|_{L^2})\|\nabla^sn\|_{L^2}
				+C\|w\|_{L^\infty}\|\nabla^sn\|_{L^2}^2\\
				&\quad+C\|n\|_{L^\infty}\|\nabla^s\text
				{div}w^h\|_{L^2}\|\nabla^sn\|_{L^2}+C\|n\|_{L^\infty}\|\nabla^s\text
				{div}w^\ell\|_{L^2}\|\nabla^sn\|_{L^2}  \\
				&\leq C(\|\nabla w\|_{L^\infty}\|\nabla^s n\|_{L^2}+\|\nabla n\|_{L^\infty}\|\nabla^s w\|_{L^2})\|\nabla^sn\|_{L^2}
				+C\|w\|_{L^\infty}\|\nabla^sn\|_{L^2}^2\\
				&\quad+C\|n\|_{L^\infty}\|\nabla^s\text
				{div}w^h\|_{L^2}\|\nabla^sn\|_{L^2}+C\|n\|_{L^\infty}\|\nabla^sw\|_{L^2}\|\nabla^sn\|_{L^2}  \\
				&\leq\varepsilon\|\nabla^s\text{div} w^h\|_{L^2}^2+C(1+t)^{-3-s},
			\end{align*}
			where we have used \eqref{lem-hk1} and the fact that
			\begin{align*}
				\|\nabla^\alpha (n,w,\phi,\psi)\|_{L^\infty}\leq C\|\nabla^{\alpha+1} (n,w,\phi,\psi)\|_{L^2}^{\frac{1}{2}} \|\nabla^{\alpha+2} (n,w,\phi,\psi)\|_{L^2}^{\frac{1}{2}}\leq C(1+t)^{-\frac{3}{2}-\frac{\alpha}{2}},~~~0\leq \alpha\leq s-2.
			\end{align*}
			It also holds
			\begin{align*}
				\left| \int_{\mathbb{R}^3}\nabla^s f_1^\ell\cdot\nabla^sn^h\mathrm{d}x \right|
				&\leq  \|\nabla^{s-1}f_1^\ell\|_{L^2}\|\nabla^s n^h\|_{L^2}\\
				&\leq  \|\nabla^{s-1}f_1\|_{L^2}\|\nabla^s n\|_{L^2}\\
				&\leq \|\nabla^{s-1}(w\cdot\nabla n)\|_{L^2}\|\nabla^sn\|_{L^2}
				+\|\nabla^{s-1}(n\text{div}w)\|_{L^2}\|\nabla^sn\|_{L^2}\\
				&\leq C(\|\nabla^{s-1}w\|_{L^2}\|\nabla n\|_{L^\infty}+\|w\|_{L^\infty}\|\nabla^sn\|_{L^2})\|\nabla^sn\|_{L^2}\\
				&\quad+C(\|\nabla^{s-1}n\|_{L^2}\|\text{div}w\|_{L^\infty}+\|n\|_{L^\infty}\|\nabla^sw\|_{L^2})\|\nabla^sn\|_{L^2}\\
				&\leq C(1+t)^{-3-s}.
			\end{align*}
			Therefore, we have
			\begin{align}
				\frac{1}{2}\frac{\mathrm{d}}{\mathrm{d}t}\|\nabla^sn^h\|_{L^2}^2+c\int_{\mathbb{R}^3} \nabla^s\text{div}w^h\cdot\nabla^sn^h \mathrm{d}x
				\leq \varepsilon\|\nabla^s\text{div} w^h\|_{L^2}^2+C(1+t)^{-3-s}.\label{331}
			\end{align}
			Similarly,	applying the operator $\nabla^s$ to $\eqref{nsac1}_2$, multiplying the resulting equation by $\nabla^s w^h$, and then integrating over $\mathbb{R}^3$, we obtain 	
			\begin{align}
				&\frac{1}{2}\frac{\mathrm{d}}{\mathrm{d}t}\|\nabla^s w^h\|_{L^2}^2-c\int_{\mathbb{R}^3} \nabla^s{\rm{div}}w^h\cdot\nabla^s n^h \mathrm{d}x-\sigma\int_{\mathbb{R}^3} \nabla^s {\rm{div}}w^h\cdot\nabla^s\phi^h\mathrm{d}x\notag\\		
				&\quad+\frac{3\nu}{4}\|\nabla^{s}\nabla w^h\|_{L^2}^2+\frac{3(\nu+\eta)}{4}\|\nabla^s{\rm{div}} w^h\|_{L^2}^2\notag\\
				&\leq C(1+t)^{-3-s}.\label{332}
			\end{align}
			Applying $\nabla^s$ to $\eqref{nsac1}_3$, and taking the inner product with $\nabla^s\phi^h$ over $\mathbb{R}^3$ gives
			\begin{align}
				&\frac{1}{2}\frac{\mathrm{d}}{\mathrm{d}t}\|\nabla^s\phi^h\|_{L^2}^2+\sigma\int_{\mathbb{R}^3} \nabla^s {\rm{div}}w^h\cdot\nabla^s\phi^h \mathrm{d}x+b\int_{\mathbb{R}^3}\nabla^s{\rm{div}}\psi^h\cdot\nabla^s\phi^h \mathrm{d}x\notag\\
				&=\int_{\mathbb{R}^3}\nabla^sf_3^h\cdot\nabla^s\phi^h\mathrm{d}x
				=\int_{\mathbb{R}^3}\nabla^sf_3\cdot\nabla^s\phi^h\mathrm{d}x
				-\int_{\mathbb{R}^3}\nabla^sf_3^\ell\cdot\nabla^s\phi^h\mathrm{d}x.\label{32501}
			\end{align}	
			The first term on the right-hand side of \eqref{32501} is estimated below
			\begin{align*}
				\int_{\mathbb{R}^3}\nabla^sf_3\cdot\nabla^s\phi^h\mathrm{d}x
				&=-c\int_{\mathbb{R}^3} \nabla^s(w\cdot\nabla \phi)\cdot \nabla^s\phi^h \mathrm{d}x-c(\gamma-1)\int_{\mathbb{R}^3} \nabla^s(\phi\text{div}w)\cdot\nabla^s \phi^h\mathrm{d}x\\
				&\quad+\frac{\sigma}{c}\int_{\mathbb{R}^3} \nabla^s\Big(\frac{2\nu|\mathbb{D}w|^2+\eta({\rm{div}}w)^2}{1+n}\Big)\cdot\nabla^s\phi^h \mathrm{d}x+b\int_{\mathbb{R}^3}\nabla^s\Big(\frac{n\text{div}\psi}{1+n}\Big)\cdot\nabla^s\phi^h \mathrm{d}x\\
				&\triangleq Y_1+Y_2+Y_3+Y_4.
			\end{align*}
			Thus, we have
			\begin{align*}
				\left|\int_{\mathbb{R}^3} \nabla^s(w\cdot\nabla \phi)\cdot \nabla^s\phi^h \mathrm{d}x\right|
				&\leq \left|\int_{\mathbb{R}^3} (\nabla^s(w\cdot\nabla \phi)-w\cdot\nabla^s \nabla \phi)\cdot \nabla^s\phi^h \mathrm{d}x\right|+\left|\int_{\mathbb{R}^3}w\cdot\nabla^s\nabla\phi\cdot\nabla^s\phi^h\mathrm{d}x\right|\\
				&\leq C(\|\nabla w\|_{L^\infty}\|\nabla^s\phi\|_{L^2}+\|\nabla \phi\|_{L^\infty}\|\nabla^s w\|_{L^2})\|\nabla^s\phi^h\|_{L^2}\\
				&~~+\left|\int_{\mathbb{R}^3}w\cdot\nabla^s\nabla\phi^\ell\cdot\nabla^s\phi^h\mathrm{d}x\right|
				+\left|\int_{\mathbb{R}^3}w\cdot\nabla^s\nabla\phi^h\cdot\nabla^s\phi^h\mathrm{d}x\right|\\
				&\leq C (\|\nabla w\|_{L^\infty}\|\nabla^s\phi\|_{L^2}+\|\nabla\phi\|_{L^\infty}\|\nabla^sw\|_{L^2})\|\nabla^s\phi\|_{L^2}\\
				&~~+C \|w\|_{L^\infty}\|\nabla^{s+1}\phi^\ell\|_{L^2}\|\nabla^s\phi^h\|_{L^2}
				+C\|\text{div}w\|_{L^\infty}\|\nabla^s\phi^h\|_{L^2}^2\\
				&\leq C(\|\nabla w\|_{L^\infty}\|\nabla^s\phi\|_{L^2}+\|\nabla\phi\|_{L^\infty}\|\nabla^sw\|_{L^2})\|\nabla^s\phi\|_{L^2}+C \|w\|_{L^\infty}\|\nabla^s\phi\|_{L^2}^2\\
				&\leq C(1+t)^{-3-s},
			\end{align*}
			which implies that
			\begin{align}
				|Y_1|\leq C(1+t)^{-3-s}.\label{y1}
			\end{align}
			In the same way, one has
			\begin{align}
				|Y_2|+|Y_3|\leq \varepsilon(\|\nabla^s\text{div}w^h\|_{L^2}^2+\|\nabla^s\nabla w^h\|_{L^2}^2)+C(1+t)^{-3-s}.\label{y23}
			\end{align}
			Noting that $Y_4$ can be rewritten as
			$$
			\begin{aligned}
				Y_4&=b \int_{\mathbb{R}^3}\Big(\nabla^s\Big(\frac{n\text{div}\psi}{1+n}\Big)- \frac{n}{1+n}\nabla^s\text{div}\psi\Big)\cdot\nabla^s \phi^h\mathrm{d}x
				+b\int_{\mathbb{R}^3}\frac{n}{1+n}\nabla^s\text{div}\psi \cdot\nabla^s\phi^h\mathrm{d}x\\
				&=b \int_{\mathbb{R}^3}\Big(\nabla^s\Big(\frac{n\text{div}\psi}{1+n}\Big)- \frac{n}{1+n}\nabla^s\text{div}\psi\Big)\cdot\nabla^s \phi^h\mathrm{d}x
				+b\int_{\mathbb{R}^3}\frac{n}{1+n}\nabla^s\text{div}\psi^h \cdot\nabla^s\phi^h\mathrm{d}x\\
				&\quad+b\int_{\mathbb{R}^3}\frac{n}{1+n}\nabla^s\text{div}\psi^\ell\cdot\nabla^s\phi^h\mathrm{d}x\\
				&\triangleq Y_{41}+Y_{42}+Y_{43}.
			\end{aligned}
			$$
			It is easy to verify that
			$$
			\begin{aligned}
				|Y_{41}|&\lesssim \Big(\Big\|\nabla\Big(\frac{n}{1+n}\Big)\Big\|_{L^\infty}\|\nabla^s\psi\|_{L^2}+
				\Big\|\nabla^s\Big(\frac{n}{1+n}\Big)\Big\|_{L^2}\|\text{div}\psi\|_{L^\infty}\Big)\|\nabla^s\phi^h\|_{L^2}\\
				&\lesssim (\|\nabla n\|_{L^\infty}\|\nabla^s\psi\|_{L^2}+\|\nabla^s n\|_{L^2}\|\nabla \psi\|_{L^\infty})\|\nabla^s \phi^h\|_{L^2}\\
				&\lesssim (\|\nabla n\|_{L^\infty}\|\nabla^s\psi\|_{L^2}+\|\nabla^s n\|_{L^2}\|\nabla \psi\|_{L^\infty})\|\nabla^s \phi\|_{L^2}\\
				&\lesssim (1+t)^{-3-s}.
			\end{aligned}
			$$
			As for $Y_{42}$, in terms of integration by parts, we obtain
			$$
			\begin{aligned}
				Y_{42}&=-b\int_{\mathbb{R}^3}\nabla\Big(\frac{n}{1+n}\Big)\nabla^s\psi^h\cdot\nabla^s\phi^h \mathrm{d}x-b\int_{\mathbb{R}^3}\frac{n}{1+n}\nabla^s\psi^h\cdot\nabla^s\nabla\phi^h\mathrm{d}x\\
				&\triangleq Y_{421}+Y_{422}.
			\end{aligned}
			$$
			The first term $Y_{421}$ is bounded by
			\begin{align*}
				|Y_{421}|\leq C\|\nabla n\|_{L^\infty}\|\nabla^s\psi^h\|_{L^2}\|\nabla^s\phi^h\|_{L^2}
				\leq C\|\nabla n\|_{L^\infty}\|\nabla^s\psi\|_{L^2}\|\nabla^s\phi\|_{L^2}\leq C(1+t)^{-3-s}.
			\end{align*}
			With the help of $\eqref{nsac1}_{4}$ and $\eqref{nsac1}_{1}$, the second term $Y_{422}$ is calculated below
			$$
			\begin{aligned}
				Y_{422}&=-b\int_{\mathbb{R}^3}\frac{n}{1+n}\nabla^s\psi^h \cdot\nabla^s\nabla\phi^h \mathrm{d}x\\
				&=\int_{\mathbb{R}^3}\frac{n}{1+n}\nabla^s\psi^h \cdot\nabla^s(\partial_t\psi^h+\frac{1}{\tau}\psi^h) \mathrm{d}x\\
				&=\frac{1}{2}\frac{\mathrm{d}}{\mathrm{d}t}\int_{\mathbb{R}^3}\frac{n}{1+n}|\nabla^s\psi^h|^2\mathrm{d}x-\frac{1}{2}\int_{\mathbb{R}^3}\Big(\frac{n}{1+n}\Big)_t|\nabla^s\psi^h|^2\mathrm{d}x+\int_{\mathbb{R}^3}\frac{n}{\tau(1+n)}|\nabla^s\psi^h|^2\mathrm{d}x\\
				&=\frac{1}{2}\frac{\mathrm{d}}{\mathrm{d}t}\int_{\mathbb{R}^3}\frac{n}{1+n}|\nabla^s\psi^h|^2\mathrm{d}x-\frac{1}{2}\int_{\mathbb{R}^3}\frac{n_t}{(1+n)^2}|\nabla^s\psi^h|^2\mathrm{d}x+\int_{\mathbb{R}^3}\frac{n}{\tau(1+n)}|\nabla^s\psi^h|^2\mathrm{d}x\\
				&=\frac{1}{2}\frac{\mathrm{d}}{\mathrm{d}t}\int_{\mathbb{R}^3}\frac{n}{1+n}|\nabla^s\psi^h|^2\mathrm{d}x-\frac{1}{2}\int_{\mathbb{R}^3}\frac{f_1-c\text{div}w}{(1+n)^2}|\nabla^s\psi^h|^2\mathrm{d}x+\int_{\mathbb{R}^3}\frac{n}{\tau(1+n)}|\nabla^s\psi^h|^2\mathrm{d}x\\
				&\leq \frac{1}{2}\frac{\mathrm{d}}{\mathrm{d}t}\int_{\mathbb{R}^3}\frac{n}{1+n}|\nabla^s\psi^h|^2\mathrm{d}x+C(1+t)^{-3-s}.
			\end{aligned}
			$$
			It then concludes from above to get
			\begin{align}
				|Y_4|\leq |Y_{41}|+|Y_{421}|+|Y_{422}|\leq \frac{1}{2}\frac{\mathrm{d}}{\mathrm{d}t}\int_{\mathbb{R}^3}\frac{n}{1+n}|\nabla^s\psi^h|^2\mathrm{d}x
				+C(1+t)^{-3-s}.\label{y4}
			\end{align}
			It is easy to show that the second term on the right-hand side of \eqref{32501} satisfies
			\begin{align}\label{THZ003}
				\int_{\mathbb{R}^3}\nabla^sf_3^\ell\cdot\nabla^s\phi^h\mathrm{d}x\leq C(1+t)^{-3-s}.
			\end{align}
			In terms of \eqref{y1}-\eqref{THZ003} and \eqref{32501}, we verify that
			\begin{align}
				&\frac{1}{2}\frac{\mathrm{d}}{\mathrm{d}t}\|\nabla^s\phi^h\|_{L^2}^2+\sigma\int_{\mathbb{R}^3} \nabla^s {\rm{div}}w^h\cdot\nabla^s\phi^h \mathrm{d}x+b\int_{\mathbb{R}^3}\nabla^s{\rm{div}}\psi^h\cdot\nabla^s\phi^h\mathrm{d}x\notag\\
				&\leq C(1+t)^{-3-s}+\varepsilon(\|\nabla^s\nabla w^h\|_{L^2}^2+\|\nabla^s\text{div}w^h\|_{L^2}^2)\notag\\
				&\quad+\frac{1}{2}\frac{\mathrm{d}}{\mathrm{d}t}\int_{\mathbb{R}^3}\frac{n}{1+n}|\nabla^s\psi^h|^2\mathrm{d}x.\label{3T1}
			\end{align}
			Taking $\int_{\mathbb{R}^3}\nabla^\alpha \eqref{nsac1}_4 \cdot \nabla^s\psi^h\mathrm{d}x$, one has
			\begin{equation}\label{3T2}
				\begin{aligned}
					&\frac{1}{2}\frac{\mathrm{d}}{\mathrm{d}t}\|\nabla^s \psi^h\|_{L^2}^2+\frac{1}{\tau}\|\nabla^s\psi^h\|_{L^2}^2-
					b\int_{\mathbb{R}^3}\nabla^s{\rm{div}}\psi^h\cdot\nabla^s\phi^h\mathrm{d}x=0.
				\end{aligned}
			\end{equation}
			It then follows from \eqref{3T1} and \eqref{3T2} that
			\begin{align}
				&\frac{1}{2}\frac{\mathrm{d}}{\mathrm{d}t}\Big(\|\nabla^s\phi^h\|_{L^2}^2+\frac{1}{1+n}\|\nabla^s\psi^h\|_{L^2}^2\Big)+\frac{1}{\tau}\|\nabla^s\psi^h\|_{L^2}^2+\sigma\int_{\mathbb{R}^3} \nabla^s {\rm{div}}w^h\cdot\nabla^s\phi^h\mathrm{d}x\notag\\
				&\leq C(1+t)^{-3-s}+\varepsilon(\|\nabla^s\nabla w^h\|_{L^2}^2+\|\nabla^s\text{div}w^h\|_{L^2}^2).\label{336}
			\end{align}
			Similarly, we also have
			\begin{align}
				&\frac{c}{2}\|\nabla^sn^h\|_{L^2}^2+\frac{\mathrm{d}}{\mathrm{d}t}\int_{\mathbb{R}^3} \nabla^{s-1}w^h\cdot\nabla^sn^h\mathrm{d}x\notag\\
				&\leq C(1+t)^{-3-s}+C(\|\nabla^s\nabla w^h\|_{L^2}^2
				+\|\nabla^s{\rm{div}} w^h\|_{L^2}^2+\|\nabla^s\phi^h\|_{L^2}^2)+\frac{\sigma^2}{c}\|\nabla^\alpha \phi^h\|_{L^2}^2, \label{334}
			\end{align}	
			and
			\begin{align}
				&\frac{3b}{4}\|\nabla^s\phi^h\|_{L^2}^2+\frac{\mathrm{d}}{\mathrm{d}t}\int_{\mathbb{R}^3} \nabla^{s-1}\psi^h\cdot\nabla^s\phi^h\mathrm{d}x\notag\\
				&\leq C(1+t)^{-3-s}+C(\|\nabla^s\nabla w^h\|_{L^2}^2+\|\nabla^s\text{div}w^h\|_{L^2}^2+\|\nabla^s\psi^h\|_{L^2}^2).\label{335}
			\end{align}	
			Choosing the suitably small constant $\beta_3>0$, taking \eqref{331}+\eqref{332}+\eqref{336}+$\frac{bc\beta_3}{4\sigma^2}\times$\eqref{334}+$\beta_3\times$\eqref{335} yields
			\begin{align*}
				&\frac{\mathrm{d}}{\mathrm{d}t}
				\Big\{\frac{1}{2}\Big(\|\nabla^sn^h\|_{L^2}^2+\|\nabla^sw^h\|_{L^2}^2+\|\nabla^s\phi^h\|_{L^2}^2+\frac{1}{1+n}\|\nabla^s\psi^h\|_{L^2}^2\Big)\\
				&\quad+\frac{bc\beta_3}{4\sigma^2}\int_{\mathbb{R}^3} \nabla^{s-1} w^h\cdot\nabla^sn^h\mathrm{d}x+\beta_3\int_{\mathbb{R}^3} \nabla^{s-1}\psi^h\cdot\nabla^s\phi^h\mathrm{d}x
				\Big\}\\
				&\quad+\frac{3\nu}{4}\|\nabla^{s}\nabla w^h\|_{L^2}^2+\frac{3(\nu+\eta)}{4}\|\nabla^s{\rm{div}} w^h\|_{L^2}^2+\frac{1}{\tau}\|\nabla^s\psi^h\|_{L^2}^2+\frac{bc^2\beta_3}{8\sigma^2}\|\nabla^\alpha n^h\|_{L^2}^2+\frac{b\beta_3}{2}\|\nabla^\alpha \phi^h\|_{L^2}^2\\
				&\leq C(1+t)^{-3-s}+C(\varepsilon+\beta_3)(\|\nabla^s\nabla w^h\|_{L^2}^2+\|\nabla^s\text{div}w^h\|_{L^2}^2).
			\end{align*}
			Due to the smallness of $\varepsilon$ and $\beta_3$, we have
			\begin{align*}
				&\frac{\mathrm{d}}{\mathrm{d}t}
				\Big\{\frac{1}{2}\Big(\|\nabla^sn^h\|_{L^2}^2+\|\nabla^sw^h\|_{L^2}^2+\|\nabla^s\phi^h\|_{L^2}^2+\frac{1}{1+n}\|\nabla^s\psi^h\|_{L^2}^2\Big)\\
				&\quad+\frac{bc\beta_3}{4\sigma^2}\int_{\mathbb{R}^3} \nabla^{s-1} w^h\cdot\nabla^sn^h\mathrm{d}x+\beta_3\int_{\mathbb{R}^3} \nabla^{s-1}\psi^h\cdot\nabla^s\phi^h\mathrm{d}x
				\Big\}\\
				&\quad+\frac{\nu}{2}\|\nabla^{s}\nabla w^h\|_{L^2}^2+\frac{(\nu+\eta)}{2}\|\nabla^s{\rm{div}} w^h\|_{L^2}^2+\frac{1}{\tau}\|\nabla^s\psi^h\|_{L^2}^2+\frac{bc^2\beta_3}{16\sigma^2}\|\nabla^\alpha n^h\|_{L^2}^2+\frac{b\beta_3}{4}\|\nabla^\alpha \phi^h\|_{L^2}^2\\
				&\leq C(1+t)^{-3-s},
			\end{align*}
			and
			\begin{align*}
				&\frac{1}{2}\Big(\|\nabla^sn^h\|_{L^2}^2+\|\nabla^sw^h\|_{L^2}^2+\|\nabla^s\phi^h\|_{L^2}^2+\frac{1}{1+n}\|\nabla^s\psi^h\|_{L^2}^2\Big)\notag\\
				&\quad+\frac{bc\beta_3}{4\sigma^2}\int_{\mathbb{R}^3} \nabla^{s-1} w^h\cdot\nabla^sn^h\mathrm{d}x+\beta_3\int_{\mathbb{R}^3} \nabla^{s-1}\psi^h\cdot\nabla^s\phi^h\mathrm{d}x\\
				&\sim  (\|\nabla^sn^h\|_{L^2}^2+\|\nabla^sw^h\|_{L^2}^2+\|\nabla^s\phi^h\|_{L^2}^2+\|\nabla^s\psi^h\|_{L^2}^2).
			\end{align*}
			As a consequence, it holds
			\begin{align*}
				&\frac{\mathrm{d}}{\mathrm{d}t}\big(\|\nabla^s(n^h,w^h,\phi^h,\psi^h)\|_{L^2}^2+\beta_3\int_{\mathbb{R}^3}(\nabla^{s-1}w^h\cdot\nabla^sn^h+ \nabla^{s-1}\psi^h\cdot\nabla^s\phi^h)\mathrm{d}x
				\big)\notag\\
				&\quad+C_5\big( \|\nabla^s(n^h,w^h,\phi^h,\psi^h)\|_{L^2}^2+\beta_3\int_{\mathbb{R}^3}(\nabla^{s-1}w^h\cdot\nabla^sn^h+ \nabla^{s-1}\psi^h\cdot\nabla^s\phi^h)\mathrm{d}x\big)\notag\\
				&\leq C(1+t)^{-3-s},
			\end{align*}
			where the positive constant $C_5$ is independent of time. In terms of Gr\"{o}nwall's inequality, one has
			\begin{equation}\label{337}
				\|\nabla^s(n^h,w^h,\phi^h,\psi^h)(t)\|_{L^2}\leq C(1+t)^{-\frac{3}{2}-\frac{s}{2}}.	
			\end{equation}
			Now we are in a position to establish the time-decay estimates of $\|\nabla^s\psi^\ell\|_{L^2}$. First, we have
			\begin{align*}
				\psi(x,t)
				=&G_{41}*n_0+G_{42}*w_0+G_{43}*\phi_0+G_{44}*\psi_0\notag\\
				&+\int_0^t\Big\{G_{41}(t-r)*f_1(r)+G_{42}(t-r)*f_2(r)+G_{43}(t-r)*f_3(r)\Big\}\mathrm{d}r.
			\end{align*}
			Due to Parseval's equality, one has
			\begin{align*}
				\|\nabla^s\psi^\ell(t)\|_{L^2}
				&\leq \big\|\nabla^sG_{41}^\ell*n_0+\nabla^sG_{42}^\ell*w_0+\nabla^sG_{43}^\ell*\phi_0+\nabla^sG_{44}^\ell*\psi_0\big\|_{L^2}\notag\\
				&\quad+\int_0^t\big\|\nabla^sG_{41}^\ell(t-r)*f_1(r)+\nabla^sG_{42}^\ell(t-r)*f_2(r)+\nabla^sG_{43}^\ell(t-r)*f_3(r)\big\|_{L^2}\mathrm{d}r\notag\\
				&\leq C(1+t)^{-\frac{5}{4}-\frac{s}{2}}\|(n_0,w_0,\phi_0,\psi_0)\|_{L^1}+C\int_0^{\frac{t}{2}}
				(1+t-r)^{-\frac{5}{4}-\frac{s}{2}}(1+r)^{-2}\mathrm{d}r\notag\\
				&\quad +\int_{\frac{t}{2}}^t (1+t-r)^{-\frac{5}{4}}(1+r)^{-\frac{5}{4}-\frac{s}{2}}\mathrm{d}r\notag\\
				&\leq C(1+t)^{-\frac{5}{4}-\frac{s}{2}},
			\end{align*}
			which, together with \eqref{337}, yields
			\begin{align*}
				\|\nabla^s\psi(t)\|_{L^2}\leq \|\nabla^s\psi^\ell(t)\|_{L^2}+\|\nabla^s\psi^h(t)\|_{L^2}
				\leq C(1+t)^{-\frac{5}{4}-\frac{s}{2}}.
			\end{align*}
			This completes the proof of this lemma.
		\end{proof}
		\section{The Proofs of Theorems \ref{thm1} and \ref{thm3}}\label{Sec5}
		\hspace{2em}Since  the \emph{a priori} estimates and the time-decay rates have been established in Sections \ref{Sec3} and \ref{Sec4}, we are now in a position to complete the proofs of our main results.
		
		\noindent{\textbf{Proof of Theorem \ref{thm1}.}\ }
		By Proposition \ref{pro-est}, it gives
		\begin{equation}\label{32901}
			\|(n,w,\phi,\psi)(t)\|_{H^s}^2+\int_0^t\left(\|\nabla n(r)\|_{H^{s-1}}^2+\|\nabla w(r)\|_{H^{s}}^2+\|\nabla\phi(r)\|_{H^{s-1}}^2+\| \psi(r)\|_{H^{s}}^2\right)\mathrm{d}r\leq C\delta_0^2.
		\end{equation}
		Choosing the initial data $\delta_0$ sufficiently small such that $C\delta_0^2 \leq \frac{1}{4}\delta^2$, we could close the \emph{a priori} assumption \eqref{a-priori est}. By virtue of \eqref{new-var} and \eqref{32901}, we can obtain \eqref{main-est}. Then, the global well-posedness of the solution is obtained based on the continuous argument.
		It remains to establish the large-time behavior of the solution $(\rho,u,\theta,q)$. By Lemmas \ref{lem-hs}, \ref{lem-hk} and \ref{lemphi}, it is easy to verify that, for $0\leq k\leq s$,
		\begin{equation}\label{thm2-p1}
			\|\nabla^k(n,w,\phi)(t)\|_{L^2}\leq C(1+t)^{-\frac{3}{4}-\frac{k}{2}},\quad
			\|\nabla^k\psi(t)\|_{L^2}\leq C(1+t)^{-\frac{5}{4}-\frac{k}{2}},
		\end{equation}
		which, together with \eqref{new-var}, yields \eqref{upper}. This completes the proof of Theorem \ref{thm1}.\endproof
		\vskip 2mm
		\noindent{\textbf{Proof of Theorem \ref{thm3}.}\ }
		Due to Parseval's equality and \eqref{nonlinear-v}, one has
		\begin{align}
			\|n(t)\|_{L^2}
			&\geq \|G_{11}*n_0+G_{12}*w_0+G_{13}*\phi_0+G_{14}*\psi_0\|_{L^2}\notag\\
			&\quad-\int_0^t \|G_{11}(t-r)*f_1(r)+G_{12}(t-r)*f_2(r)+G_{13}(t-r)*f_3(r)\|_{L^2}\mathrm{d}r.\label{thm3-p1}
		\end{align}
		By the definition of $F$, \eqref{l-inf-est}, \eqref{lem-hs0} and \eqref{lem-hs5}, we get
		\begin{align}
			&\int_0^t \|G_{11}(t-r)*f_1(r)+G_{12}(t-r)*f_2(r)+G_{13}(t-r)*f_3(r)\|_{L^2}\mathrm{d}r\notag\\
			&\leq C\int_0^t\big\{(1+t-r)^{-\frac{3}{4}}
			\|F(r)\|_{L^1}+e^{-R(t-r)}\|F(r)\|_{L^2}\big\}\mathrm{d}r\notag\\
			&\leq C\mathcal{M}(t)\int_0^t(1+t-r)^{-\frac{3}{4}}(1+r)^{-\frac{3}{4}}(\|\nabla(n,w,\phi,\psi) (r)\|_{L^2}+\|\nabla^2w(r)\|_{L^2})\mathrm{d}r\notag\\
			&\leq C\mathcal{M}(t)\delta_0(1+t)^{-\frac{3}{4}}\notag\\
			&\leq C(1+t)^{-\frac{3}{4}}(\delta_0^2+\delta_0A_0).\label{thm3-p2}
		\end{align}
		Since $\delta_0$ is sufficiently small, substituting \eqref{pro-li-d1} and \eqref{thm3-p2} into \eqref{thm3-p1} yields
		\begin{equation}\label{thm3-p3}
			\|n(t)\|_{L^2}\geq \bar{c}_1(1+t)^{-\frac{3}{4}}- C(1+t)^{-\frac{3}{4}}
			(\delta_0^2+A_0\delta_0)\geq \frac{1}{2}\bar{c}_1(1+t)^{-\frac{3}{4}}.
		\end{equation}
		If $t$ is large enough, it gives rise to
		\begin{align}
			\|\Lambda^{-1}n(t)\|_{L^2}&\leq \|\Lambda^{-1}n^\ell(t)\|_{L^2}
			+\|\Lambda^{-1}n^h(t)\|_{L^2}\notag\\
			&\leq C(1+t)^{-\frac{1}{4}}+C\int_0^t(1+t-r)^{-\frac{1}{4}}\|F(r)\|_{L^1}\mathrm{d}r+C\|n^h(t)\|_{L^2}\notag\\
			&\leq  C(1+t)^{-\frac{1}{4}}+C\|n\|_{L^2}\notag\\
			&\leq C(1+t)^{-\frac{1}{4}}.\label{thm3-p4}
		\end{align}
		Thanks to Lemmas \ref{lem-gn} and \ref{lem-ff}, it is easy to verify that
		\begin{equation}\label{32502}
			\|U\|_{L^2}\leq C\|\Lambda^{-1}U\|_{L^2}^{\frac{k}{k+1}}
			\|\nabla^k U\|_{L^2}^{\frac{1}{k+1}}.
		\end{equation}
		Then, combining \eqref{32502} with \eqref{thm3-p3} and \eqref{thm3-p4} yields
		\begin{equation}\label{thm3-p5}
			\|\nabla^kn(t)\|_{L^2}\geq \tilde{c}_4(1+t)^{-\frac{3}{4}-\frac{k}{2}},
		\end{equation}
		where $\tilde{c}_4$ is a positive constant independent of time. By \eqref{new-var}, there exists a constant $\bar{C}_1>0$ such that
		\begin{align*}
			\|\nabla^k(\rho-\rho_*)(t)\|_{L^2}\geq \bar{C}_{1}(1+t)^{-\frac{3}{4}-\frac{k}{2}}.
		\end{align*}
		Similarly, there also exists positive constants $\bar{C}_2, \bar{C}_3$ and $\bar{C}_4$ independent of time such that
		$$
		\|\nabla^ku(t)\|_{L^2}\geq \bar{C}_2(1+t)^{-\frac{3}{4}-\frac{k}{2}},\quad
		\|\nabla^k(\theta-\theta_*)(t)\|_{L^2}\geq \bar{C}_3(1+t)^{-\frac{3}{4}-\frac{k}{2}},\quad
		\|\nabla^kq(t)\|_{L^2}\geq \bar{C}_4(1+t)^{-\frac{5}{4}-\frac{k}{2}}.
		$$
		Then choosing $\bar{C}_*=\min\{\bar{C}_1 ,\bar{C}_2,\bar{C}_3,\bar{C}_4\}$,  we complete the proof of Theorem \ref{thm3}.
		\endproof
		
		\section*{Acknowledgements}
		
		\hspace{2em} The authors would like to thank Professor Hailiang Li and Professor  Feimin Huang for helpful discussions and comments. Li's work is supported by the National Natural Science Foundation of China (Grant No.12331007), and the ``333 Project'' of Jiangsu Province. Tang's work is supported by the National Natural Science Foundation of China (Grant No.12501293), and Anhui Provincial Natural Science Foundation (Grant No.2408085QA031). Zhang's work is supported by the National Natural Science Foundation of China  (Grant No.12501300), the Basic Research Program of Jiangsu (Grant No. BK20250881), and the Natural
Science Foundation of the Jiangsu Higher Education Institutions of China (Grant No. 25KJB110001).
		\hspace{4em}
		
\section*{Conflict of interest statement}
		On behalf of all authors, the corresponding author states that there is no conflict of interest.

\section*{ Data Availability Statement}
	Data sharing not applicable to this article as no datasets were generated or analysed during the current study.

\end{document}